\documentclass{amsart}
\usepackage{subfiles}
\usepackage{amssymb,euscript,amsmath, mathrsfs,tensor}
\usepackage{mathtools}
\usepackage[dvips]{graphicx}
\usepackage[dvips]{color}
\include{headerstemplate}
\usepackage{stmaryrd}
\usepackage[all,cmtip]{xy}
\usepackage{xcolor}
\usepackage[margin=1in]{geometry}
\usepackage{tikz-cd}
\usepackage{enumitem}

\newcounter{ENUM}

\usepackage{hyperref}

\def\Sp{\mathrm{Sp}}
\def\SO{\mathrm{SO}}

\def\SL{\mathrm{SL}}

\def\Gal{\mathrm{Gal}}
\def\Hom{\mathrm{Hom}}
\def\End{\mathrm{End}}
\def\Res{\mathrm{Res}}

\def\Rep{\mathrm{Rep}}

\def\Irr{\mathrm{Irr}}

\def\LG{{\tensor*[^L]\G{}}}

\def\Fl{\overline{\mathbb{F}_\ell}}
\def\Zl{\overline{\mathbb{Z}_\ell}}

\def\Ql{\overline{\mathbb{Q}_\ell}}

\def\GL{\mathrm{GL}}

\def\diag{\mathrm{diag}}

\def\A{\mathrm{A}}
\def\B{\mathrm{B}}

\def\F{\mathrm{F}}
\def\E{\mathrm{E}}
\def\G{\mathrm{G}}
\def\H{\mathrm{H}}

\def\J{\mathrm{J}}
\def\K{\mathrm{K}}

\def\M{\mathrm{M}}
\def\N{\mathrm{N}}
\def\P{\mathrm{P}}
\def\Q{\mathrm{Q}}
\def\R{\mathrm{R}}
\def\T{\mathrm{T}}
\def\U{\mathrm{U}}
\def\X{\mathrm{X}}

\def\W{\mathrm{W}}
\def\Z{\mathrm{Z}}

\def\GSpin{\mathrm{GSpin}}

\def\Ind{\mathrm{Ind}}

\def\GSpin{\mathrm{GSpin}}

\def\Tr{\mathrm{Tr}}

\def\lsmat{\left(\begin{smallmatrix}}
	
	\def\rsmat{\end{smallmatrix}\right)}

\makeatletter
\def\Ddots{\mathinner{\mkern1mu\raise\p@
		\vbox{\kern7\p@\hbox{.}}\mkern2mu
		\raise4\p@\hbox{.}\mkern2mu\raise7\p@\hbox{.}\mkern1mu}}
\makeatother

\newcommand{\margh}[1]{}

\newtheorem{thm}{Theorem}[section]
\newtheorem{proposition}[thm]{Proposition}
\newtheorem{lemma}[thm]{Lemma}
\newtheorem{corollary}[thm]{Corollary}
\newtheorem{conj}[thm]{Conjecture}
\newtheorem{remark}[thm]{Remark}

\theoremstyle{definition}
\newtheorem{definition}[thm]{Definition}

\newtheorem{example}[thm]{Example}

\numberwithin{equation}{section}
\def\presuper#1#2{\mathop{}\mathopen{\vphantom{#2}}^{#1}\kern-\scriptspace#2}
\def\rob#1{\textcolor{black}{#1}}

\begin{document}
	\title{Langlands-Shahidi gamma factors and local transfer in families}
	\author{Johannes Girsch}
	\address{Johannes Girsch, Department of Mathematical Sciences, Durham University, DH1 3LE, United Kingdom.}
\email{johannes.girsch@live.de}

	\author{Robert Kurinczuk}
	\address{Robert Kurinczuk, School of Mathematics and Statistics, University of Sheffield, Sheffield, S3 7RH, United Kingdom.}
\email{robkurinczuk@gmail.com}
	\maketitle
	\begin{abstract}
	Let~$\G$ be a quasi-split~$p$-adic group, defined over a non-archimedean local field of residue cardinality~$q$, and~$\R$ be a~$\mathbb{Z}[\sqrt{q}^{-1},\mu_{4p^\infty}]$-algebra.	We generalize the constructions of Shahidi local coefficients and Langlands--Shahidi gamma factors to the setting of co-Whittaker~$\R$-representations.  For classical~$p$-adic groups, we prove there is a unique integral morphism of endomorphism rings of Gelfand--Graev representations, interpolating the (standard) local generic transfer to a general linear group. This yields a local transfer morphism in families, which allows us to deduce properties of the local~$\ell$-adic generic transfer and congruences modulo primes~$\ell\neq p$.
	\end{abstract}

	\setcounter{tocdepth}{1}
	\tableofcontents
	\section{Introduction}
	\subsection{Langlands--Shahidi gamma factors in the local Langlands programme}
	In the local Langlands programme for $\mathbb{C}$-representations of a $p$-adic group defined over a non-archimedean local field~$\F$, gamma factors are useful arithmetic invariants. Developed by Shahidi when~$\F/\mathbb{Q}_p$ and extended by Lomeli to~$\F$ of positive characteristic, the Langlands--Shahidi method is a general method which attaches families of gamma factors to generic~$\mathbb{C}$-representations of~$p$-adic groups which appear as Levi subgroups of a maximal parabolic subgroup in a quasi-split~$p$-adic group.  Langlands--Shahidi gamma factors in the complex local Langlands programme have numerous applications, including: to constructing and studying local transfer morphisms for classical groups, to reducibility points of parabolic inductions, and to the normalization of intertwining operators.
	

	\subsection{Local Langlands in families}
	%
	Suppose for this subsection that~$\F$ is a finite extension of $\mathbb Q_p$, and let~$A$ be a reduced complete Noetherian local ring with maximal ideal~$\mathfrak{m}$, with finite residue field~$k$ of characteristic~$\ell\neq p$, and suppose further that~$A$ is~$\ell$-torsion-free. 
	
	Then Emerton--Helm \cite{EH} conjectured an interpolation of the generic~$\ell$-adic local Langlands correspondence in families for~$\GL_n(\F)$, which associates
	\[\text{continuous representation $\rho:\Gal(\overline{\F}/\F)\rightarrow \GL_n(A)$} \mapsto \text{ \emph{co-Whittaker} $A$-representation~$\pi(\rho)$ of~$\GL_n(\F)$}\] (using the formulation of Helm in \cite{HelmCW}).  This conjecture was proved in subsequent work of Helm \cite{HelmCurtis} and Helm--Moss \cite{HM}, with one key step in their proof a descent theorem utilizing Moss' theory of gamma factors for co-Whittaker representations of~$\GL_n(\F)$ \cite{Moss}.
		
Families encode congruences between representations. For a simple example, given~$\rho,\rho':\Gal(\overline{\F}/\F)\rightarrow  \GL_n(\mathfrak{o}_{\K})$, where~$\mathfrak{o}_{\K}$ is the ring of integers in an~$\ell$-adic field~$\K$, which are congruent modulo~$\mathfrak{p}_{\K}$ the maximal ideal of~$\mathfrak{o}_{\K}$, i.e.,~$\rho\equiv \rho'\pmod{\mathfrak{p}_{\K}}$, then if we let~$\overline{\rho}=\rho\otimes k_{\K}$, where~$k_{\K}=\mathfrak{o}_{\K}/\mathfrak{p}_\K$ denotes the residue field, we can define a family:
\[\rho\times_{\overline{\rho}}\rho'=\{(v,v')\in \rho\times \rho': v\equiv v'\pmod{\mathfrak{p}_{\K}}\}\]
	over~$\mathfrak{o}_{\K}\times_{k_{\K}}\mathfrak{o}_{\K}$.  More universally, for~$\overline{\rho}$, there exists a universal framed deformation ring~$\R_{\overline{\rho}}^{\square}/\W(\Fl)$ and a continuous representation~$\overline{\rho}^{\square}:\Gal(\overline{\F}/\F)\rightarrow  \GL_n(\R_{\overline{\rho}}^{\square})$ which parametrizes all framed lifts of~$\overline{\rho}$ to complete Noetherian local~$\W(\Fl)$-algebras (cf., \cite[§7]{HelmCW}).  Helm then introduced a scheme interpolating the rings~$\R_{\overline{\rho}}^{\square}$ over all such~$\Fl$-representations~$\overline{\rho}$ in \cite{HelmCurtis}; the construction of this scheme was subsequently generalized to Langlands parameters for a general~$p$-adic group (and coefficients over~$\mathbb{Z}[1/p]$) in \cite{DHKMmoduli}, and local Langlands in families conjectured for a general~$p$-adic group in \cite{DHKMconj}.  
	
	To approach local Langlands in families for other groups via gamma factors, we thus need a theory of co-Whittaker~representations and gamma factors in families.  In his PhD thesis, and subsequent papers \cite{Girsch1,Girsch}, the first author developed the doubling method in families. This constructs gamma factors of classical groups under very mild hypotheses on the representation (which we expect to lead to applications outside the co-Whittaker setting), however it is currently not normalized to be compatible with local Langlands.   In this paper, on the other hand, we develop the Langlands--Shahidi method in families.  As well as applying in different instances to the doubling method, this method has the advantage that the canonical normalizations of Shahidi extend naturally to the families setting so the factors we introduce are compatible with local Langlands, and with local generic transfer of complex representations of classical groups.  This allows us to interpolate this transfer in families using gamma factors.

	\subsection{Co-Whittaker~$R$-representations for a quasi-split reductive group}
	Let~$\F$ be a non-archimedean local field, and~$\G$ be the~$\F$-points of a quasi-split reductive group defined over~$\F$.    Let~$\B=\T\U$ be a Borel subgroup of~$\G$, with unipotent radical~$\U$, and maximal torus~$\T$.  Let~$\mathfrak{w}=(\U,\Psi)$ be a Whittaker datum for~$\G$ over~$\R_0=\mathbb{Z}[\sqrt{q}^{-1},\mu_{4p^\infty}]$, i.e.,~$\Psi:\U\rightarrow \R_0^{\times}$ is a non-degenerate character valued in~$\R_0^{\times}$ acting on a free~$\R_0$-module of rank one.\footnote{If~$\F$ has positive characteristic, it is not necessary to include all~$p$-power roots of unity in the coefficients; we let the reader discover this in the main text.}
		Given an~$\R_0$-algebra~$\R$, by a~$\mathfrak{w}$-generic representation, we mean an (irreducible) representation which is~$\mathfrak{w}\otimes\R$-generic.
	
	Let~$\pi$ be an integral (irreducible)~$\mathfrak{w}$-generic~$\Ql$-representation of~$\G$.  Then its reduction modulo~$\ell$ is not necessarily~$\mathfrak{w}$-generic (such examples already occur for~$\GL_2(\F)$ when~$\ell\mid q+1$).  Thus to consider families of representations, we need a more flexible notion than genericity.  
		
Let~$\R$ denote an~$\R_0$-algebra.  We let~$\Gamma^{\G}_{\mathfrak w}=\operatorname{c-Ind}_{\U}^{\G}(\Psi)$ denote the compactly induced Gelfand-Graev representation over~$\R_0$, and~$\Gamma^{\G}_{\mathfrak w,\R}=\Gamma^{\G}_{\mathfrak w}\otimes_{\R_0}\R$.	An~$\R$-representation~$\rho$ is called a \emph{co-Whittaker~$\R$-representation of depth~$n$} if~$\rho$ has depth~$n$,~$\Hom_{\R[\G]}(\Gamma^{\G}_{\mathfrak w,\R},\rho)$ is free of rank one over~$\R$, and the natural map
\[\Gamma^{\G}_{\mathfrak w,\R}\otimes \Hom_{\R[\G]}(\Gamma^{\G}_{\mathfrak w,\R},\rho)\rightarrow \rho\]
is surjective.  In Section \ref{cwrepssect}, we introduce co-Whittaker $R$-representations and establish the basic properties we need for our construction and applications to local transfer.  

Let
\[\E^{\G}_{\mathfrak w}=\End_{\R_0[\G]}(\Gamma^{\G}_{\mathfrak w});\] it is a commutative~$\R_0$-torsion free reduced ring.  Let~$e_n$ denote the idempotent of the centre of the category corresponding to depth~$n$ projection (defined over~$\mathbb{Z}[1/p]$ by \cite{Datfinitude}), then the representation~$\Gamma^{\G}_{\mathfrak{w},n}:=e_n\Gamma^{\G}_{\mathfrak w}$ considered as a representation over its endomorphism ring~$ \E^{\G}_{\mathfrak w,n}$ is the \emph{universal} co-Whittaker representation of depth $n$ for Helm's dominance relation.  For an~$\R_0$-algebra~$\R$, an~$\R_0$-algebra homomorphism~$\phi:\E^{\G}_{\mathfrak w,n}\rightarrow \R$ defines a co-Whittaker~$\R$-representation by~$\Gamma^{\G}_{\mathfrak w,n}\otimes_{\E^{\G}_{\mathfrak w,n},\phi}\R$, and conversely given a co-Whittaker~$\R$-representation~$\rho$ of depth~$n$ the natural action of~$ \E^{\G}_{\mathfrak w,n}$ on~$\Hom_{\R[\G]}(\Gamma^{\G}_{\mathfrak w,n},\rho)$ defines a morphism~$\phi_{\rho}:\E^{\G}_{\mathfrak w,n}\rightarrow \R$; satisfying,~$\Gamma^{\G}_{\mathfrak w,n}\otimes_{\E^{\G}_{\mathfrak w,n},\phi_{\rho}}\R$ \emph{dominates}~$\rho$.



	\subsection{Shahidi's local coefficients in families}
	The first step is to construct Shahidi's local coefficients for co-Whittaker~$R$-representations, for which there are two main steps to algebraize and integralify: (i) Whittaker functionals on parabolic inductions and (ii) intertwining operators. 
	
	Let~$\P=\M\N$ be a standard Levi decomposition of a standard parabolic subgroup of~$\G$ corresponding to a subset~$\theta$ of our fixed set of simple roots; we write~$\P=\P_{\theta}$,~$\N=\N_{\theta}$, and~$\M=\M_{\theta}$.  Then\[\mathfrak{w}_{\M}:=(\U_\M:=\U\cap\M, \Psi_{\M}:=\Psi\mid_{\U_{\M}})\] is a Whittaker datum for~$\M$ over~$\R_0$.  Let~$w_0$ be the element of the Weyl group of~$\G$ associated to~$\P$ in Section \ref{Whittfunctionalssect}, and~$\widetilde{w_0}$ its Langlands--Shelstad lift.
	
	 First, in Proposition \ref{propcoinv}, we prove that a co-Whittaker $\R$-representation is of Whittaker type (our proof following a suggestion of Helm), and then study explicitly Whittaker functionals on parabolic inductions of Whittaker type~$\R$-representations, proving:
	
	\begin{proposition}[Section \ref{Whittfunctionalssect}]\label{11intro}
	Let~$\P=\M\N$ be a maximal standard parabolic subgroup of~$\G$ and~$\chi_{\M,\R}:\M\rightarrow \R[\M/\M^{\circ}]^\times$ the universal unramified character.     Let~$\rho$ be a~$\mathfrak{w}_\M$-Whittaker type~$\R$-representation of~$\M$,~$\vartheta\in \Hom_{\R[\U_{\M}]}(\rho,\Psi_{\M}\otimes\R)$ a surjective~$\R$-module generator, and~$\nu_{\N_{w_0(\theta)}}$ be an~$\R$-Haar measure on~$\N_{w_0(\theta)}$.  Then~$i_{\M,\P}^{\G}(\rho \otimes \chi_{\M,\R})$ is of $\mathfrak{w}$-Whittaker type, and there is a unique element
	\[\Omega_{\mathfrak{w}}(\rho,\vartheta,\nu_{\N_{w_0(\theta)}})\in \Hom_{\R[\M/\M^{\circ}][\U]}(i_{\M,\P}^{\G}(\rho \otimes \chi_{\M,\R}),\Psi_\R\otimes\R[\M/\M^\circ])\]
	satisfying, for~$f\in i_{\M,\P}^{\G}(\rho \otimes \chi_{\M,\R})$ with support contained in~$\P w_0^{-1}\B$,
	\[	\Omega_{\mathfrak w}(\rho,\vartheta,\nu_{\N_{w_0(\theta)}})(f)=\int_{\N_{w_0(\theta)}} \vartheta(f(\widetilde{w}_{0}^{-1}n))\Psi_\R^{-1}(n)dn.\]
	Moreover,~$\Omega_{\mathfrak{w}}(\rho,\vartheta,\nu_{\N_{w_0(\theta)}})$ generates the~$\R[\M/\M^{\circ}]$-module of~$\mathfrak{w}$-Whittaker functionals on~$i_{\M,\P}^{\G}(\rho \otimes \chi_{\M,\R})$. 
	\end{proposition}
	
	The second ingredient in the construction of local coefficients is intertwining operators.  In \cite{Datnu}, following the work of Waldspurger \cite{W}, Dat explained intertwining operators exist under certain conditions: ``the existence of~$(\sigma,\P,\Q)$-regular elements''.  In the doubling method, for the intertwining operators of classical groups which appear there, intertwining operators were studied in the PhD thesis of the first named author, and one can follow the argument in \cite[Proposition 9.1]{Girsch} to establish these conditions.  These approaches have been recently extended and improved by Moss--Trias \cite{MT}, and we use their general results to short-cut the approach of \cite{Girsch}, and establish the results on intertwining operators adapted for the construction of gamma factors which we need.
	
	Let~$\mathcal{S}_{\M,\R}$ denote the multiplicative subset of~$\R[\M/\M^{\circ}]$ generated by elements $a_0+a_1m+\cdots+a_r m^r$,~$m$ non-trivial, with leading and trailing coefficients a unit, i.e.,~$a_0,a_r\in\R^\times$. Suppose~$\theta=\Delta\backslash\{\alpha\}$, and~$\rho_{\theta}$ denotes half the sum of absolute roots  in~$\N_{\theta}$.  We write~$X^c$ for $X^{\langle\rho_\theta,\alpha^\vee\rangle^{-1}/2}$,~$\mathcal{S}_{X^c,\R}$ for the multiplicative subset of~$\R[X^{\pm c}]$ of elements with leading and trailing coefficient a unit, and define a ring homomorphism~$\Delta_{X,\R}\colon \R[\M/\M^\circ]\rightarrow \R[X^{\pm c}]$ in Section \ref{int}.  Our variable~$X$ is an algebraic substitute for~$q^{-s}$ where~$s\in\mathbb{C}$ in Shahidi's work.
	
		\begin{thm}[Theorem \ref{intertwiningops}]
		Suppose that~$\R$ is a Noetherian~$\mathbb{Z}[1/\sqrt{q}]$-algebra, and
		let $\rho$ be a finitely generated smooth~$\R[\M]$-module which satisfies Schur's lemma. Then there exists an~$\R[\M/\M^{\circ}][\G]$-morphism
		\[\mathbb{M}(\rho,\nu_{\N_{w_0(\theta)}})\colon i_{\M,\P_\theta}^{\G}(\rho\otimes \chi_{\M,\R})\rightarrow i_{w_0\M w_0^{-1},\P_{w_0(\theta)}}^{\G}((\rho\otimes \chi_{\M,\R})^{\widetilde{w}_0^{-1}})\otimes_{\R[\M/\M^{\circ}]} \mathcal{S}_{\M,\R}^{-1}\R[\M/\M^{\circ}]\]
		such that:
		\begin{enumerate}
			\item for all $f\in i_{\M,\P_\theta}^{\G}(\rho\otimes \chi_{\M,\R})$ that are supported on $\P_\theta w_0^{-1}\P_{w_0(\theta)}w_0$, 
			\[\mathbb{M}(\rho,\nu_{\N_{w_0(\theta)}})(f)(\widetilde{w}_0)=\int_{\N_{w_0(\theta)}} f(\widetilde{w}_0^{-1}n\widetilde{w}_0) dn;\]
			\item and there exists a polynomial~$\mathcal Q\in \mathcal{S}_{\M,\R}$ such that
			\[\mathcal Q\cdot \mathbb{M}(\rho,\nu_{\N_{w_0(\theta)}})\colon i_{\M,\P_\theta}^{\G}(\rho\otimes \chi_{\M,\R})\rightarrow i_{w_0\M w_0^{-1},\P_{w_0(\theta)}}^{\G}((\rho\otimes \chi_{\M,\R})^{\widetilde{w}_0^{-1}})\]
			and the image of $\mathcal Q$ under $\Delta_{X,\R}$ lies in $\mathcal S_{X,\R}$. 
		\end{enumerate}
	\end{thm}

By using the intertwining operator and Proposition \ref{11intro} we can then construct an element that interpolates Shahidi's local coefficient.

		\begin{proposition}
	Let $\rho$ be a co-Whittaker type~$\R$-representation of~$\M$ with respect to $\mathfrak w_\M$. Then there exists a unique element
	\[\mathcal{D}_\R(\rho,\P_\theta,\Xi,\psi)\in \mathcal{S}_{\M,\R}^{-1}\R[\M/\M^\circ],\] 
	satisfying the functional equation
	\[\mathcal{D}_\R(\rho,\P_\theta,\Xi,\psi)\cdot\Omega_{\mathfrak w}(\rho,\vartheta,\nu_{\N_{w_0(\theta)}})(f)= \Omega_{\mathfrak w}({}^{\widetilde{w}_0}\rho,{}^{\widetilde{w}_0}\vartheta,\nu_{\N_\theta})( \mathbb{M}(\rho,\nu_{\N_{w_0(\theta)}})(f))\]
	for all $f\in i_{\M,\P_\theta}^\G(\rho\otimes\chi_{\M,\R})$. Moreover, there is an element $\mathcal Q\in\mathcal S_{\M,\R}$ such that $\mathcal Q\cdot \mathcal{D}_\R(\rho,\P_\theta,\Xi,\psi)\in \R[\M/\M^\circ]$ and the image of $\mathcal Q$ under $\Delta_{X,\R}$ lies in $\mathcal S_{X,\R}$.
	
\end{proposition}

This allows us to define
$$\mathcal D_\R(X,\rho,\P_\theta,\Xi,\psi)\coloneqq\Delta_{X,\R}(\mathcal Q)^{-1}\Delta_{X,\R}(\mathcal Q\cdot \mathcal D_\R(\rho,\P_\theta,\Xi,\psi))\in \mathcal S_{X,\R}^{-1}\R[X^{\pm c}].$$
We can then prove that only integral powers of $X$ appear and show that it is a unit, i.e. $$\mathcal D_\R(X,\rho,\P_\theta,\Xi,\psi)\in(\mathcal S_{X,\R}^{-1}\R[X^{\pm 1}])^\times.$$
By using these results we can then use the inductive procedure of Shahidi to define gamma-factors associated to co-Whittaker~$\R$-representations of~$\M$.
		\subsection{Langlands-Shahidi gamma factors in families}
		Our main theorem is the following construction of gamma factors:
\begin{thm}[{Theorem \ref{mainthm}}]
		Let~$\R$ be a~$\R_0$-algebra. Suppose that~$\P=\M\N$ is a maximal standard parabolic subgroup of~$\G$, and~$\rho$ is a co-Whittaker type~$\R$-representation of~$\M$ with respect to~$\mathfrak w_\M$. Moreover, we denote by~$r=\bigoplus_{i=1}^{m_r}r_i$ the decomposition of the adjoint representation~$r\colon{^L}\M(\mathbb{C})\rightarrow \GL({}^L\mathfrak{n})$ of~\eqref{equationr}.  Then for each~$i\in\{1,\ldots,m_r\}$, there exists
			 $$\gamma(X,\rho,r_i,\psi)\in (\mathcal{S}_{X,\R}^{-1}\R[X^{\pm 1}])^{\times},$$ depending only on the isomorphism class of~$\rho$, satisfying:
			\begin{enumerate}
				\item (Compatibility with homomorphisms)  Suppose~$\alpha\colon \R\rightarrow \R'$ is a morphism of~$\R_0$-algebras, and $\alpha_X\colon\mathcal{S}_{X,\R}^{-1}\R[X^{\pm 1}]\rightarrow \mathcal{S}_{X,\R'}^{-1}\R'[X^{\pm 1}]$ is the unique~$\R_0$-algebra morphism extending $\alpha$ which is the identity on~$X$. Then, for each~$i\in\{1,\ldots,m_r\}$, we have 
				\[\gamma(X,\rho\otimes_\alpha \R',r_i,\psi)=\alpha_X(\gamma(X,\rho,r_i,\psi)).\]
				\item When~$\R=\mathbb{C}$, and~$\rho$ is a co-Whittaker~$\mathbb{C}$-representation with respect to $\mathfrak w_\M$ with $\mathfrak w_\M$-generic cosocle~$\rho_{gen}$, then
				\[\gamma(X,\rho,r_i,\psi)=\gamma_{Sh}(X,\rho_{gen},r_i,\psi),\]
				where~$\gamma_{Sh}(X,\rho_{gen},r_i,\psi)$ is the gamma factor constructed by Shahidi and Lomeli.
				\item (Dominance) Let~$\rho'$ be a co-Whittaker type~$\R$-representation of~$\M$ with respect to $\mathfrak w_\M$ such that~$\rho$ dominates~$\rho'$. Then, for all~$i\in\{1,\ldots,m_r\}$, we have that
				$$\gamma(X,\rho,r_i,\psi)=\gamma(X,\rho',r_i,\psi).$$
			\item(Local Functional Equation) For each~$i\in\{1,\ldots,m_r\}$, we have
			\[\gamma(X,\rho,r_i,\psi)\gamma(q^{-1}X^{-1},\rho,\overline{r_i},\psi^{-1})=1\]
						\end{enumerate}

	\end{thm}

%
%

	\subsection{Local transfer in families}
			We now suppose that~$\F$ is a $p$-adic field, and let~$\G$ be one of the quasi-split groups
		\[\Sp_{2n}(\F),\SO_r(\F),\SO_{2n,\E'/\F}(\F),\U_n(\E/\F),\mathrm{GSpin}_r(\F),\mathrm{GSpin}_{2n,\E'/\F}(\F).\]  
		Let~$\Pi_{\mathfrak{w},\mathbb{C}}(\G)$ denote the set of isomorphism classes of~$\mathfrak{w}$-generic~$\mathbb{C}$-representations of~$\G$.  In the split case setting~$\E=\F$, and letting~$\phi_{\LG}:\LG(\mathbb{C})\rightarrow \presuper{L}\GL_N(\mathbb{C})$ denote the embedding associated with local transfer (see Section \ref{notationclassical}), then we let~$\Pi_{gen,\mathbb{C}}(\GL_N(\E))$ denote the set of isomorphism classes of generic~$\mathbb{C}$-representations of~$\GL_N(\E)$.
		
Then work of Asgari--Cogdell--Heiermann--Kim--Krishnamurthy--Shahidi--Piatetski-Shapiro \cite{CKPSS1,MR2127169,MR2149370,MR2219256,CKPSS2,AsgariShahidiFunct,ACS,Heiermanngeneric}  defines a transfer of~$\mathfrak{w}$-generic~$\mathbb{C}$-representations:
\[\mathrm{t}_{\G,\mathfrak{w}}\colon\Pi_{\mathfrak{w},\mathbb{C}}(\G)\rightarrow \Pi_{gen,\mathbb{C}}(\GL_N(\E)).\]
Let~$\mathfrak{W}$ be a Whittaker datum for~$\GL_N(\E)$ over~$\R_0$.  Our main application of Langlands--Shahidi gamma factors in families is the following theorem:
	
	\begin{thm}[Local transfer in families]
	\begin{enumerate}
\item	 There exists a unique morphism of~$\R_0$-algebras
			\[\mathbf{t}_{\G,\mathfrak{W},\mathfrak{w}}\colon\E^{\GL_{N}(\E)}_{\mathfrak{W}}\rightarrow \E^{\G}_{\mathfrak{w}}\]
			such that if~$\pi$ is a generic irreducible~$\mathbb{C}$-representation of~$\G$ then the unique generic subquotient of~$\Gamma_{\mathfrak{W}}^{\GL_N(\E)}\otimes_{ f_{\pi,0}\circ \mathbf{t}_{\G,\mathfrak{W},\mathfrak{w}}}\mathbb{C}$ is~$\mathrm{t}_{\G,\mathfrak{w}}(\pi)$.  
			\item Moreover,~$\mathbf{t}_{\G,\mathfrak{W},\mathfrak{w}}$ is compatible with Langlands--Shahidi gamma factors for~$\G\times\GL_m(\E)$, and the equality of Langlands--Shahidi gamma factors characterizes the morphism uniquely (cf., Theorem \ref{mainthm2} for a precise statement).
			\end{enumerate}
	\end{thm}
	The most interesting primes for modular representation theory of a~$p$-adic group are called \emph{non-banal}, these are the primes which divide the pro-order of a compact open subgroup.    If we invert all non-banal primes, the situation becomes much simpler.  Indeed, letting~$\M_{\G}$ denote the product of the non-banal primes for~$\G$ and~$\GL_N(\E)$, then by \cite{DHKMconj} the rings~$\E^{\GL_{N}(\E)}_{\mathfrak{W}}\otimes \overline{\mathbb{Z}}[1/\M_{\G}]$ and~$\E^{\G}_{\mathfrak{w}}\otimes \overline{\mathbb{Z}}[1/\M_{\G}]$ admit a simple description, and we establish the existence of a morphism after inverting~$\M_{\G}$ in Theorem \ref{banalmainthm} using the approach of \emph{op. cit.} The difficulty of proving this theorem as stated is thus not to invert the primes~$\ell\neq p$ which are non-banal (and this is where we use gamma factors invoking the strategy  Helm and Moss developed to prove local Langlands in families for~$\GL_n(\F)$).

	The amazing work of Fargues--Scholze \cite{fargues2024geometrizationlocallanglandscorrespondence}, and Scholze \cite{Scholze}, together with Helm--Moss' isomorphism \cite{HM}, also defines an integral transfer morphism
	\[\mathbf{t}_{\G,\mathfrak{W},\mathfrak{w}}^{\mathrm{FS}}\colon\E^{\GL_{N}(\E)}_{\mathfrak{W}}\rightarrow \E^{\G}_{\mathfrak{w}},\]
	and we explain the connections in Section \ref{FSconsequences}, and for which cases one can deduce (from recent results in the literature) that it agrees with~$\mathbf{t}_{\G,\mathfrak{W},\mathfrak{w}}$.
	
		 We expect gamma factors will prove useful in understanding the finer points of this morphism -- for example in determining its image.  For symplectic groups, unitary groups, and odd special orthogonal groups, we expect that this morphism is surjective (Conjecture \ref{conjecturesurjectivity}).  Approaching this conjecture via a descent theorem involving gamma factors for classical groups, and normalization of intertwining operators in families using the results of this paper are current projects of the authors. 
	
	
	\subsection{Congruences under local transfer}
	As a corollary of the integral transfer morphism, we can also deduce pleasing properties of the local transfer of~$\mathfrak{w}$-generic~$\ell$-adic representations and congruences modulo primes~$\ell\neq p$, in the style of M\'inguez--S\'echerre \cite{MS}, and Vign\'eras \cite{VigInv}.  
	
	Let~$\G$ be a quasi-split~$p$-adic group, and~$\mathfrak{w}$ be a~$\Zl$-Whittaker datum in~$\G$.
Following Vign\'eras, for an integral~$\mathfrak{w}$-generic~$\Ql$-representation~$\pi$ of $\G$, we write~$J_{\ell}(\pi)$ for the unique~$\mathfrak{w}$-generic subquotient of the reduction modulo~$\ell$ of~$\pi$.
	
		\begin{corollary}[{cf.,~M\'inguez--S\'echerre \cite{MS}}]
		Let~$\G$ be one of the quasi-split groups
		\[\Sp_{2n}(\F),\SO_r(\F),\SO_{2n,\E'/\F}(\F),\U_n(\E/\F),\mathrm{GSpin}_r(\F),\mathrm{GSpin}_{2n,\E'/\F}(\F).\]  
		\begin{enumerate}
			\item  The generic local transfer\[\mathrm{t}_{\G,\mathfrak{w},\ell}\colon\Pi_{\mathfrak{w},\Ql}(\G)\rightarrow \Pi_{gen,\Ql}(\GL_N(\E))\] preserves integrality of~$\Ql$-representations.  
			\item If~$\pi$ and~$\pi'$ are integral~$\mathfrak{w}$-generic~$\Ql$-representations of~$\G$ then
			\[J_{\ell}(\pi)\simeq J_{\ell}(\pi')\Rightarrow J_{\ell}(\mathrm{t}_{\G,\mathfrak{w},\ell}(\pi))\simeq J_{\ell}(\mathrm{t}_{\G,\mathfrak{w},\ell}(\pi')).\]
			\item Suppose~$\pi$ and~$\pi'$ are integral~$\mathfrak{w}$-generic~$\Ql$-representations of~$\G$ with~$J_{\ell}(\pi)\simeq J_{\ell}(\pi')$.  The irreducible components of~$r_{\ell}(\mathrm{t}_{\G,\mathfrak{w},\ell}(\pi))\oplus r_{\ell}(\mathrm{t}_{\G,\mathfrak{w},\ell}(\pi'))$ all have the same supercuspidal support.
		\end{enumerate}
	\end{corollary}
This extends the generic supercuspidal case of \cite{MS} to further groups and all~$\mathfrak{w}$-generic representations.

	\subsection{Acknowledgements} It is a pleasure to thank David Helm for introducing us to co-Whittaker type representations and his continued support.  In particular, we learnt most of the basic properties of co-Whittaker representations from him (cf., Lemma \ref{cwlemma} - any mistakes, of course, remain the authors').  We thank Jean-Fran\c cois Dat, Luis Lomeli, Nadir Matringe, Gil Moss, Vincent S\'echerre, and Jack Shotton for useful conversations, and Tasho Kaletha for explaining a key choice in Shahidi's construction to us.   Considering the Langlands--Shahidi method in families was first suggested to the second author by Colin Bushnell, who suggested (for~$\GL_n(\F)$ at least) the method should extend to a general setting.
	
	Both authors were supported by EPSRC grant EP/V001930/1. J.\ G.\ was moreover supported by UKRI1020.
	
	\subsection*{AI Acknowledgements}
	After finishing a complete first draft of the article without AI assistance, we asked ChatGPT model 6 Pro to review the manuscript. We then fixed ourselves many typos and minor inaccuracies pointed out by this model.	
	
	\section{Preliminaries}
	
	\subsection{Notation}
	Let~$F$ be a non-archimedean local field with ring of integers~$\mathfrak{o}_\F$.  We let~$\mathfrak{p}_\F$ denote the unique maximal ideal of~$\mathfrak{o}_\F$ and denote by~$k_\F=\mathfrak{o}_\F/\mathfrak{p}_\F$ the residue field.  We set~$q=| k_\F|$ to be the cardinality of the residue field, which is a power of a prime number~$p$.  We fix a square root~$\sqrt{q}\in \overline{\mathbb{Q}}$ of~$q$ in (a fixed algebraic closure)~$\overline{\mathbb{Q}}$ of~$\mathbb{Q}$.
	
	Let~$\mathbb{G}$ be a quasi-split connected reductive group defined over~$\F$.  Fix a maximal~$\F$-split torus~$\mathbb{S}$ in~$\mathbb{G}$, with centralizer~$\mathbb{T}$ a maximal torus in~$\mathbb{G}$.  Fix a Borel subgroup~$\mathbb{B}$ of $\mathbb{G}$ with unipotent radical~$\mathbb{U}$ and Levi decomposition~$\mathbb{B}=\mathbb{T}\mathbb{U}$.  
	
	We write~$\X^*=\Hom(\mathbb{S},\mathbb{G}_m)$ and~$\X_*=\Hom(\mathbb{G}_m,\mathbb{S})$. Let~$\Phi\subset\X^*$ (resp.~$\Phi^+\subset \Phi$) be the subset of non-trivial characters appearing in the adjoint representation on the Lie algebra of~$\mathbb{G}$ (resp.~$\mathbb{B}$), and~$\Delta\subset \Phi^+$ the subset of simple roots.  For~$\alpha\in\Phi$, we write~$\alpha>0$ if~$\alpha\in\Phi^+$, and for a subset~$\theta$ of~$\Phi$ write~$\theta>0$ if $\theta\subseteq \Phi^+$.
	
	We fix a pinning~$\Xi$ extending~$(\mathbb{B},\mathbb{S})$.  Equivalently we fix an isomorphism, for each absolute simple root $\alpha$ of $\mathbf{U}_{\overline{\F}}$, 
		\[x_{\alpha}\colon(\mathbb{G}_{a})_{\overline{\F}}\rightarrow \mathbb{U}_{\overline{\F},\alpha}.\]
	
	For~$\mathbb{M}$ a reductive subgroup of~$\mathbb{G}$, we write~$\X^*(\mathbb{M})=\Hom(\mathbb{M},\mathbb{G}_m)$ for the abelian group of~$\F$-rational characters of~$\mathbb{M}$.  
	
	We denote by roman letters the topological group of~$\F$-points of the algebraic group, for example~$\G=\mathbb{G}(\F)$,~$\B=\mathbb{B}(\F)$,~$\T=\mathbb{T}(\F)$,~$\U=\mathbb{U}(\F)$.
	
	Let $\R$ be a commutative $\mathbb Z[1/p]$-algebra. We denote by~$\Rep_{\R}(\G)$ the abelian category of smooth~$\R$-representations of~$\G$, i.e. smooth $\R[\G]$-modules. Let~$\mathfrak{Z}_{\R}(\G)$ be the centre of~$\Rep_\R(\G)$.
	
	For~$\H$ a closed subgroup of~$\G$, and~$\theta$ an~$\R$-character of~$\H$, we let~$(-)_{\H,\theta}\colon\Rep_{\R}(\G)\rightarrow \R\text{-Mod}$ denote the functor of~$\theta$-coinvariants, where for a smooth $\R$-representation $(\pi,\mathscr{V})$ of $\G$ we set
	\[\mathscr{V}_{\H,\theta}=\mathscr{V}/\langle \pi(h)v-\theta(h)v\mid h\in\H, v\in\mathscr{V}\rangle.\]
	If $H$ is a filtered union of pro-$p$ compact open subgroups, this functor is exact cf.\ \cite[I.4.11]{Vig96}. Moreover, we set $\mathscr{V}^{\H,\theta}=\{v\in\mathscr{V}\mid\pi(h)v=\theta(h)v\text{ for all }h\in\H\}$. Note that if $\J\leq\operatorname{N}_\G(\H)$ also normalizes $\theta$, i.e.\ $\theta(h)=\theta(jhj^{-1})$ for all $h\in\H,j\in\J$, then both $\mathscr V_{\H,\theta}$ and $\mathscr V^{\H,\theta}$ are naturally smooth $R[\J]$-modules.
	
	Moreover, for an element $g\in\G$ we set $^{g}\H=g\H g^{-1}$ and $\H^g=g^{-1}\H g$. If $(\pi,\mathscr V)$ is a smooth representation of $\H$ we set $^{g}\pi$ (respectively $\pi^g$) to be the smooth representation of $^{g}\H$ (respectively $\H^g$) defined by $^{g}\pi(h)v=\pi(g^{-1}hg)v$ (respectively $\pi^g(h)v=\pi(ghg^{-1})v$) for $v\in\mathscr V$.

	\subsection{Parabolic subgroups and the Bruhat decomposition}\label{parabolics}
	Recall that a parabolic subgroup~$\mathbb{P}$ of~$\mathbb{G}$ is called \emph{standard} if it contains~$\mathbb{B}$, and (as in \cite[Theorem 15.4.6]{Springer}) we have a bijection
	\begin{align*}
		\text{Subsets of }\Delta &\rightarrow \text{Standard parabolic subgroups}\\
		\theta&\mapsto \mathbb{P}_\theta,
	\end{align*}
	and we write~$\P_\theta=\mathbb{P}_\theta(\F)$ for the~$\F$-points as above. So~$\P_{\emptyset}=\B$, and a maximal standard parabolic subgroup is of the form~$\P_{\Delta\backslash \{\alpha\}}$ for some simple root~$\alpha\in\Delta$.  
	
	A standard parabolic subgroup~$\mathbb{P}_\theta$ has a standard Levi decomposition~$\mathbb{P}_\theta=\mathbb{M}_{\theta}\mathbb{N}_{\theta}$ where~$\mathbb{N}_\theta\subseteq\mathbb{U}$ denotes the unipotent radical of~$\mathbb{P}_{\theta}$ and the Levi factor~$\mathbb{M}_{\theta}$ contains~$\mathbb{T}$.   
	
	Let~$\W$ denote the Weyl group of~$\Phi$, and for~$\alpha\in\Delta$, write~$w_\alpha$ for its corresponding reflection.  For a subset~$\theta$ of~$\Delta$ we write~$\W_\theta=\langle w_{\alpha}\mid\alpha\in\theta\rangle$ {(with the understanding that~$\W_\emptyset=1$ the trivial group)}. 
	For $\theta,\eta\subseteq\Delta$ we set
	\begin{align*}
		[\W_\theta\backslash\W]&=\{w\in \W\mid w^{-1}\theta >0\}\\
		[\W/\W_{\eta}]&=\{w\in \W\mid w\eta >0\}\\
		[\W_\theta\backslash\W/\W_{\eta}]&=[\W_\theta\backslash\W]\cap [\W/\W_{\eta}].
	\end{align*}
	
	Any $w\in \W$ can be represented by an element of~$\N_{\mathbb{G}}(\mathbb{S})(F)$ and we write~$\P_{\theta} w \P_{\eta}$ for the double coset associated to~$w\in [\W_{\theta}\backslash \W/\W_{\eta}]$ (as the double coset is independent of the choice of representative).
	\begin{proposition}[The Bruhat Decomposition, {\cite[1.3.1]{Casselman}}]
		For subsets~$\theta,\eta$ of~$\Delta$, we have a decomposition
		\[\G=\bigsqcup_{w\in[\W_{\theta}\backslash \W/\W_{\eta}]} \P_{\theta} w \P_{\eta}.\]
	\end{proposition}
	We have the Bruhat (partial) order~$\leqslant$ on double cosets in the above decomposition given by~$ \P_{\theta} w\P_{\eta}\leqslant  \P_{\theta} w'\P_{\eta}$ if $ \P_{\theta} w\P_{\eta}\subseteq  \overline{\P_{\theta} w'\P_{\eta}}$ (contained in the closure for the~$p$-adic topology on~$\G$).  The identity coset is minimal with respect to this ordering.
	
	\subsection{The universal unramified twist and parabolic inductions}
	Let~$\R$ be a~$\mathbb{Z}[\sqrt{q}^{-1}]$-algebra.  Let~$\H$ be a closed subgroup of~$\G$ which contains a compact open subgroup~$\K_\H$ of pro-order invertible in~$\R$. Then there is an~$\R$-Haar measure on~$\H$ normalized to give unit volume to~$\K_\H$, and we write ~$\delta_{\H}$ for the modulus character of~$\H$ (cf.~ \cite{Vig96}). Whenever, we choose an $\R$-Haar measure on a closed subgroup $\H$, we assume that the volume of a pro-$p$ compact open subgroup is a unit in $\R$. Note that,~if~$\H$ is a parabolic subgroup of~$\G$, then $\delta_{\H}$ is valued in $q^\mathbb Z$. 
	
	We set
	\[\G^\circ=\{g\in\G\mid\chi(g)\in\mathfrak{o}_\F^\times\text{ for all }\chi\in\X^*(\mathbb{G})\}.\]
	Then~$\G^{\circ}$ is an open normal subgroup of~$\G$, and~$\G/\G^{\circ}$ is a free abelian group of rank equal to the rank of a maximal split torus in the centre of~$\mathbb{G}$. We define $\chi_{\G,\R}$ to be the universal unramified character of $\G$ over $\R$, i.e.,
	\begin{align*}
		\chi_{\G,\R}\colon \G&\to \R[\G/\G^\circ]^\times\\
		g&\mapsto g\G^\circ.
	\end{align*}

	Let~$\P=\M\N$ be a standard Levi decomposition of a standard parabolic subgroup~$\P$ of~$\G$.  We denote by ~$i_{\M,\P}^{\G}\colon\Rep_{\R}(\M)\rightarrow \Rep_{\R}(\G)$ the functor of normalized parabolic induction:
	\[i_{\M,\P}^{\G}(-)=\operatorname{Ind}_{\P}^{\G}(\delta_{\P}^{1/2}\otimes_\R(-)),\]
	and write~$r_{\M,\P}^{\G}\colon\Rep_{\R}(\G)\rightarrow\Rep_{\R}(\M)$ for its left adjoint, the (normalized) Jacquet functor.


	
	\subsection{Coefficients of gamma factors}
	Classically, as in the Rankin--Selberg method, the objects involved in the construction of $\gamma$-factors are often first defined to be convergent integrals (i.e.\ they are Laurent series that converge in certain regions) and afterwards shown to be rational functions. However, in this article following Shahidi we take a different approach, in particular the $\gamma$-factors we construct are rational functions by definition. For our applications, we will expand $\gamma$-factors as Laurent series, for which the following result will be needed. For a commutative ring $\R$ we set 
	\[\mathcal{S}_{X,\R}=\{a_{r} X^{r}+a_{(r+1)} X^{r+1}+\cdots +a_{s-1}X^{s-1}+a_s X^{s}\mid r,s\in\mathbb Z, r\leq s\text{ and }a_{r},a_s\in \R^\times\},\]
	i.e.,~$\mathcal{S}_{X,\R}$ is the multiplicative subset of~$\R[X^{\pm 1}]$ consisting of elements, whose leading and trailing coefficient are units.
	\begin{proposition}\label{powerseriesexpansion}
		Let $\R$ be any commutative ring and $Q\in\mathcal{S}_{X,\R}^{-1}\R[X^{\pm 1}]$. Then there is a unique Laurent series $L_Q\in \R[[X]][X^{-1}]$ such that for any representative $r/s$ of $Q$ we have $s\cdot L_Q=r$.
	\end{proposition}
	\begin{proof}
		We first prove that for any $r\in\R[X^{\pm1}]$ and $s\in\mathcal{S}_{X,\R}$ there is an $L\in \R[[X]][X^{-1}]$ such that $s\cdot L=r$. It is straightforward to see that we may assume that $r=1$. Let
		$$s=\sum_{i=m}^n\alpha_iX^{i},$$
		where $\alpha_m,\alpha_n\in \R^\times$. Set $t_{j}=0$ if $j<-m$ and $t_{-m}=\alpha_m^{-1}$. Then for $L_0=t_{-m}X^{-m}$ we clearly have that $s\cdot L_0-1\in X\R[X]$. 
		Assume now that for some $k\geq0$ we have found elements $t_{-m},\dotsc, t_{-m+k}\in \R$  such that $L_k=\sum_{j=0}^{k}t_{-m+j}X^{-m+j}\in \R[X^{\pm1}]$ satisfies $s\cdot L_k-1\in X^{k+1}\R[X]$. Then we set 
		$$t_{-m+k+1}=-\alpha_m^{-1}(\alpha_{m+1}t_{-m+k}+\dotsc+\alpha_nt_{-n+k+1})$$
		and a quick computation shows that $L_{k+1}=\sum_{j=0}^{k+1}t_{-m+j}X^{-m+j}$ satisfies that $s\cdot L_{k+1}-1\in X^{k+2}\R[X]$. Hence by induction we obtain a Laurent series $L$ such that $s\cdot L-1=0$.
		
		Now assume that $Q=r_1/s_1=r_2/s_2$ and there are two Laurent series $L_1,L_2\in \R[[X]][X^{-1}]$ such that $s_i\cdot L_i=r_i$ for $i=1,2$.  Since $\mathcal{S}_{X,\R}$ contains no zero-divisors we must have that $r_1s_2-r_2s_1=0$. This implies that $s_1s_2\cdot(L_1-L_2)=0$ and by considering the trailing coefficients of this equation we obtain that $L_1=L_2$.
		
	\end{proof}
	
	\begin{definition}
		For~$L\in\R[[X]][X^{-1}]$ and~$t\in\mathbb{Z}$, we write~$L_t$ for the coefficient of~$X^t$ in $L$.  
	\end{definition}
	
	\section{Co-Whittaker representations}\label{seccowhit}
	For this section, we let
	\begin{equation*}
		\R_0=\begin{cases}\mathbb{Z}[1/p,\mu_{p^\infty}]&\text{if char}(F)=0;\\\\\mathbb{Z}[1/p,\mu_{p}]&\text{if char}(F)=p.
		\end{cases}
	\end{equation*}
	The inclusion of~$p$-power roots of unity (or~$p$-th roots in the positive characteristic setting) in~$\R_0$, allows us to fix a nontrivial character~$\psi\colon \F\rightarrow \R_0^\times$.  
	
	Let~$\R$ be a commutative~$\R_0$-algebra and we set~$\psi_\R\coloneqq\psi\otimes\R$. Then $\psi_\R$ is also nontrivial, which can be seen as follows. Let $\iota\colon\R_0\to\R$ be the ring homomorphism that endows $\R$ with an $\R_0$-algebra structure and suppose that for $x\in\F$ the element $\zeta=\psi(x)\in\R_0$ is not equal to $1$. We claim then that $\psi_\R(x)=\iota(\zeta)$ does not equal $1_\R$, the multiplicative unit in $\R$. Note that there is some smallest $p$-power $p^\alpha$ such that $\zeta^{p^\alpha}=1$. However, then
	$$1-\zeta^{p^{\alpha}-1}+1-\zeta^{p^{\alpha}-2}+\dotsc+1-\zeta=p^\alpha,$$
	which implies that $1-\zeta\in\R_0^\times$. This implies that $\iota(1-\zeta)\in\R$ is nonzero and the claim follows.

	\subsection{Blocks}
	Let~$e_{\mathfrak{t}}\in\mathfrak{Z}_{\R_0}(\G)$ be a finite sum of primitive idempotents; for example~$e_{\mathfrak{t}}$ could be the idempotent associated to the depth~$n$ subcategory of~$\Rep_{\R_0}(\G)$ of \cite[Appendix A]{Datfinitude}, or for a classical group~$\G$ with~$p$ odd a primitive idempotent associated to an ``endo-parameter'' in \cite{HKSS}, or when~$\G$ is tamely ramified and~$p$ does not divide the order of its Weyl group, an idempotent associated to a Dat--Fintzen block in \cite{datfintzen}. 
	
	Via the morphism of centres of categories $\mathfrak{Z}_{\R_0}(\G)\rightarrow \mathfrak{Z}_{\R}(\G)$, as in \cite[3.1]{DHKMconj}, this defines an idempotent, which we still denote by~$ e_{\mathfrak{t}}$, of the centre of~$\Rep_{\R}(\G)$.  For an~$\R$-representation~$\Pi$ of~$\G$ we write
	\[\Pi_{\mathfrak{t}}\coloneqq e_{\mathfrak{t}}\Pi,\quad \Rep_{\R}(\mathfrak{t})\coloneqq e_{\mathfrak{t}}\Rep_{\R}(\G),\]
	i.e.,~$\Pi_{\mathfrak{t}}$ is the direct summand of~$\Pi$ in~$\Rep_{\R}(\mathfrak{t})=e_{\mathfrak{t}}\Rep_{\R}(\G)$.  If~$\Pi=\Pi_{\mathfrak{t}}$, we say that~$\Pi$ has \emph{type} $\mathfrak{t}$.
	
	\subsection{Whittaker data and representations of Whittaker type}\label{whittdata}

	\begin{definition} 
 A \emph{Whittaker datum for $\G$} is a pair~$(\U,\Psi)$ where~$\U$ is the group of $\F$-points of the unipotent radical of a Borel subgroup of~$\mathbb{G}$, and~$\Psi\colon\U\rightarrow\R_0^\times$ is an~$\R_0$-character of~$\U$ whose restriction to all simple root subgroups of~$\U$ is non-trivial (it is \emph{non-degenerate}), i.e. the restriction of~$\Psi$ to~$\U_{\alpha}\coloneqq\mathbb{U}_{\alpha}(\F)$ is non-trivial for all~$\alpha\in\Delta$.
		
	\end{definition}
	Let~$\B=\T\U$ be a Borel subgroup of~$\G$.  Choosing a character~$\psi\colon\F\rightarrow \R_0^\times$, gives rise to a well-known bijective correspondence between Whittaker data of the form~$(\U,\Psi)$ over~$\R_0$ and pinnings for~$\G$ compatible with~$(\B,\T)$, see for example \cite[§ 5]{DHKMconj}. 
	
	For a Whittaker datum $(\U,\Psi)$ of $\G$ and a commutative $\R_0$-algebra $\R$, we set $\Psi_\R=\Psi\otimes\R$.
	 
	Let~$\mathfrak w=(\U,\Psi)$ be a Whittaker datum for~$\G$. Moreover, let~$\P=\M\N$ be a standard parabolic subgroup of~$\G$. Then, as in \cite{BushnellHenniart}, by setting~$\U_{\M}=\U\cap\M$ and~$\Psi_{\M}=\Psi_{|\U_{\M}}$, the pair~$\mathfrak w_{\M}=(\U_{\M},\Psi_\M)$ is a Whittaker datum for~$\M$.

	We continue with~$e_{\mathfrak{t}}\in\mathfrak{Z}_{\R_0}(\G)$ denoting a finite sum of primitive idempotents, as in the last section.    We will need the following result:

	\begin{proposition}\label{propcoinv}
		Let~$(\pi,V)$ be a smooth~$\R$-representation of~$\G$ of type~$\mathfrak{t}$ and let $(\U,\Psi)$ be a Whittaker datum for $\G$. Then we have an isomorphism of~$\R$-modules
		\[V_{\U,\Psi_\R}\simeq \Hom_{\R[\G]}(\operatorname{c-Ind}_{\U}^{\G}(\Psi_\R),V).\]
	\end{proposition}
	\begin{proof}
		By Rodier approximation (cf.,~\cite[Appendix A]{Hansen}), there is a compact open pro-$p$ subgroup $\K$ of $\G$ and a character $\omega\colon\K\to\R^\times$ such that $\operatorname{c-Ind}_{\U}^\G(\Psi_\R)_{\mathfrak t}\simeq\operatorname{c-Ind}_{\K}^\G(\omega)_{\mathfrak t}$. The construction of Rodier also implies that $\operatorname{c-Ind}_{\U}^\G(\Psi_\R^{-1})_{\mathfrak t^\vee}\simeq\operatorname{c-Ind}_{\K}^\G(\omega^{-1})_{\mathfrak t^\vee}$.
		Let $\mathcal{M}$ be any $\R$-module and a straightforward computation yields that
		\begin{align*}
			\Ind_{\U}^{\G}(\Psi_\R\otimes \mathcal{M})_{\mathfrak t}&\cong\Hom_\R(\operatorname{c-Ind}_{\U}^\G(\Psi_\R^{-1})_{\mathfrak t^\vee},\mathcal M)^{sm}\\
			&\cong\Hom_\R(\operatorname{c-Ind}_{\K}^\G(\omega^{-1})_{\mathfrak t^\vee},\mathcal M)^{sm}\\
			&\cong\Ind_{\K}^\G(\omega\otimes \mathcal{M})_{\mathfrak t},
		\end{align*}
		where $\Hom_\R(\pi,\mathcal M)^{sm}$ denotes the smooth vectors in $\Hom_\R(\pi,\mathcal M)$ for any smooth $\R$-representation $\pi$ of $\G$.
		
		As $\K$ is compact and pro-$p$ we have an idempotent operator $e_{\K,\omega}\colon V\to V^{\K,\omega}$, given by 
		\[e_{\K,\omega}(v)=\int_{\K}\omega^{-1}(k)\pi(k)vdk,\]
		where we choose a nontrivial~$\R$-Haar measure on $\K$ such that $\K$ has volume one. This idempotent induces a direct sum decomposition of $\R[\K]$-modules
		$V=V^{\K,\omega}\oplus(1-e_{\K,\omega})(V)$, which in particular implies that
		$$\Hom_{\R[\K]}(V,\omega\otimes \mathcal{M})\simeq \Hom_{\R[\K]}(V^{\K,\omega},\omega\otimes \mathcal{M}).$$
		
		These facts allow us to compute
		\begin{align*}
			\Hom_{\R}(V_{\U,\Psi_\R},\mathcal{M})&\simeq\Hom_{\R[\G]}(V,\Ind_{\U}^{\G}(\Psi_\R\otimes \mathcal{M})_{\mathfrak t})\\
			&\simeq\Hom_{\R[\G]}(V,\Ind_{\K}^G(\omega\otimes \mathcal{M}))\\
			&\simeq\Hom_{\R[\K]}(V,\omega\otimes \mathcal{M})\\
			&\simeq\Hom_{\R[\K]}(V^{\K,\omega},\omega\otimes \mathcal{M})\\
			&\simeq \Hom_{\R}(\Hom_{\R[\K]}(\omega,V),\mathcal{M})\\
			&\simeq \Hom_{\R}(\Hom_{\R[\G]}(\operatorname{c-Ind}_\K^\G(\omega)_{\mathfrak t},V),\mathcal{M})\\
			&\simeq\Hom_{\R}(\Hom_{\R[\G]}(\operatorname{c-Ind}_\U^\G(\Psi_\R),V),\mathcal{M}).
		\end{align*}
		Overall, we obtain an isomorphism $\Hom_{\R}(V_{\U,\Psi_\R},\mathcal{M})\simeq\Hom_{\R}(\Hom_{\R[\G]}(\operatorname{c-Ind}_\U^\G(\Psi_\R),V),\mathcal{M})$ for any $\R$-module $\mathcal{M}$, which is moreover functorial in $\mathcal{M}$. Hence Yoneda's lemma implies that we have an isomorphism of $\R$-modules~$V_{\U,\Psi_\R}\simeq\Hom_{\R[\G]}(\operatorname{c-Ind}_\U^\G(\Psi_\R),V)$.
	\end{proof}

	To a Whittaker datum~$\mathfrak w=(\U,\Psi)$ for $\G$ and a finite sum of primitive idempotents $e_{\mathfrak t}$ in $\mathfrak Z_{\R_0}(\G)$, we associate the~$\R_0$-representations
	\[\Gamma^{\G}_{\mathfrak w}=\operatorname{c-Ind}_{\U}^{\G}(\Psi),\quad \Gamma^{\G}_{\mathfrak w,\mathfrak{t}}=e_{\mathfrak{t}}\operatorname{c-Ind}_{\U}^{\G}(\Psi),\] 
	of $\G$ and their endomorphism algebras:
	\[\E^{\G}_{\mathfrak w}=\End_{\R_0[\G]}(\Gamma^{\G}_{\mathfrak w}),\quad \E^{\G}_{\mathfrak w,\mathfrak{t}}=\End_{\R_0[\G]}(\Gamma^{\G}_{\mathfrak w,\mathfrak{t}}).\]
	Moreover, for a commutative $\R_0$-algebra $\R$ we set
	\[\Gamma^{\G}_{\mathfrak w,\R}=\operatorname{c-Ind}_{\U}^{\G}(\Psi_\R),\quad \Gamma^{\G}_{\mathfrak w,\R,\mathfrak{t}}=e_{\mathfrak{t}}\operatorname{c-Ind}_{\U}^{\G}(\Psi_\R),\quad \E^{\G}_{\mathfrak w,\R,\mathfrak{t}}=\End_{\R[\G]}(\Gamma^{\G}_{\mathfrak w,\R,\mathfrak{t}}),\quad \E^{\G}_{\mathfrak w,\R}=\End_{\R[\G]}(\Gamma^{\G}_{\mathfrak w,\R}).\]
	
	Let us introduce a further Noetherian base ring:  we say~$\G$ satisfies~$(*)$ if one half of the sum of positive coroots of~$\mathbb{G}$ considered as a cocharacter of~$\mathbb{T}/\mathbb{Z}_{\mathbb{G}}$ lifts to an integral cocharacter of~$\mathbb{T}$, where~$\mathbb{Z}_{\mathbb{G}}$ denotes the centre of~$\mathbb{G}$.  Then set
\begin{equation*}
		\mathrm{S}_0=\begin{cases}\mathbb{Z}[1/p]&\text{if~$\G$ satisfies (*)};\\
		\mathbb{Z}[1/p,\mu_{p}]&\text{if~$\G$ does not satisfy~$(*)$, but either~$p$ is odd or char}(F)=p;\\
		\mathbb{Z}[1/p,\mu_{8}]&\text{otherwise.}
		\end{cases}
	\end{equation*}
	By \cite[Proposition 5.2]{DHKMconj}, there exists a~$\G$-stable~$\mathrm{S}_0$-submodule~$\mathcal{W}_{\mathfrak{w}}^{\G}$ of~$\Gamma^{\G}_{\mathfrak{w}}$
	such that the natural map
	\[\mathcal{W}_{\mathfrak{w}}^{\G}\otimes \R_0\rightarrow \Gamma^{\G}_{\mathfrak{w}}\]
	is an isomorphism.   If~$e_{\mathfrak{t}}\in\mathfrak Z_{\mathrm{S}_0}(\G)$, then  we let~$\mathcal{W}_{\mathfrak{w},\mathfrak{t}}^{\G}=e_{\mathfrak{t}}\mathcal{W}_{\mathfrak{w}}^{\G}$.

We then have the following results.
	\begin{lemma}\label{Ebasicfacts}
		\begin{enumerate}
			\item The~$\mathrm{S}_0$-representation~$\mathcal{W}_{\mathfrak{w},\mathfrak{t}}^{\G}$ of~$\G$ is finitely generated and projective.
			\item The~$\R$-representation~$\Gamma^{\G}_{\mathfrak w,\R,\mathfrak{t}}$ of~$\G$ is finitely generated and projective.  
			\item We have~$\Gamma^{\G}_{\mathfrak w,\R,\mathfrak{t}}\simeq \Gamma^{\G}_{\mathfrak w,\mathfrak{t}}\otimes_{\R_0} \R$ and~$\E^{\G}_{\mathfrak w,\R,\mathfrak{t}}\simeq \E^{\G}_{\mathfrak w,\mathfrak{t}}\otimes_{\R_0}\R$.
			\item The ring~$\E^{\G}_{\mathfrak w,\R,\mathfrak{t}}$ is~$\R$-torsion-free, commutative, and is a finitely generated $\R$-algebra. 
			\item If~$\R$ is a characteristic zero domain, then~$\E^\G_{\mathfrak w,\R,\mathfrak{t}}$ is reduced.
			\item Let~$\P$ be a standard parabolic of~$\G$, with~Levi factor~$\M$, and~$\M$-opposite~$\overline{\P}$, and suppose that~$\R$ is an~$\R_0[\sqrt{q}^{-1}]$-algebra.  Then, we have an isomorphism~$r^{\G}_{\M,\overline{\P}}(\Gamma^{\G}_{\mathfrak w,\R})\simeq \Gamma^{\M}_{\mathfrak w_{\M},\R}$, which defines (by functoriality of parabolic restriction) a natural map
			\[r^{\G}_{\M,\overline{\P}}\colon\E^{\G}_{\mathfrak w,\R}\rightarrow \E^{\M}_{\mathfrak w_{\M},\R}.\]
		\end{enumerate}
	\end{lemma}
	\begin{proof}
		The first and second parts follow from \cite[Corollary 5.12]{DHKMconj}.  As~$\operatorname{c-Ind}_{\U}^{\G}(\Psi_\R)_{\mathfrak{t}}$ is projective it is torsion-free and hence so is its endomorphism algebra.  The third part follows as compact induction is compatible with scalar extension and \cite[Lemma A.1]{DHKMconj}.
		
		For~$e_{\mathfrak{s}}\in\mathfrak Z_{\mathrm{S}_0}(\G)$, the~$\mathrm{S}_0$-algebra~$\End_{\mathrm{S}_0[\G]}(\mathcal{W}_{\mathfrak{w},\mathfrak{s}}^{\G})$ is finitely generated by \cite{DHKMfiniteness}, \cite[Theorem 1.2]{DHKMconj}.  By \cite[Lemma A.1]{DHKMconj}, we deduce for any~$e_{\mathfrak t}\in\mathfrak Z_{\R_0}(\G)$ the~$\R$-algebra~$\E^{\G}_{\mathfrak w,\R,\mathfrak{t}}$ is finitely generated; indeed, this shows it holds for~$e_{\mathfrak s}\in\mathfrak Z_{\R_0}(\G)$ and~$e_{\mathfrak t}\in\mathfrak Z_{\R_0}(\G)$ is a summand of~$e_{\mathfrak s}\in\mathfrak Z_{\R_0}(\G)$.
		
		Now suppose~$\R$ is a characteristic zero domain, and let~$\K$ be the field of fractions of~$\R$, and~$\overline{\K}$ an algebraic closure of~$\K$.  By torsion-freeness and \cite[Lemma A.1]{DHKMconj} we have
		\[\End_{\R[\G]}(\operatorname{c-Ind}_{\U}^{\G}(\Psi_\R)_{\mathfrak{t}})\hookrightarrow \End_{\R[\G]}(\operatorname{c-Ind}_{\U}^{\G}(\Psi_\R)_{\mathfrak{t}})\otimes \overline{\K}\simeq \End_{\overline{\K}[\G]}(\operatorname{c-Ind}_{\U}^{\G}(\Psi_\R\otimes\overline{\K})_{\mathfrak{t}})\]
		and the algebra on the right is commutative and reduced as it is isomorphic to the~``$\Psi$-generic Bernstein centre'' by \cite[\S 9]{BushnellHenniart}, a subring of the Bernstein centre~$\mathfrak{Z}_{\G,\overline{\K}}$ of~$\Rep_{\overline{\K}}(\G)$, which is reduced for example from Bernstein's description \cite[Theorem 4.1]{DHKMconj}.  Commutativity in general follows from \cite[Lemma A.1]{DHKMconj}. The final part is \cite[2.2 Theorem]{BushnellHenniart}.
	\end{proof}
	
	\begin{definition}
		An~$\R$-representation~$\rho$ of~$\G$ is of \emph{Whittaker type} with respect to a Whittaker datum~$(\U,\Psi)$ for $\G$, if~$\Hom_{\R[\U]}(\rho,\Psi_\R)$ is free of rank one over~$\R$ and a generator of this space is surjective.
	\end{definition}
	
	We have the following basic lemma:
	\begin{lemma}\label{lemwhituniv} 
		Let $\mathfrak w=(\U,\Psi)$ be a Whittaker datum. If~$\pi$ is a Whittaker type $\R$-representation of~$\G$ with respect to $\mathfrak w$, then~$\pi\otimes\chi_{\G,\R}$ is a Whittaker type $ \R[\G/\G^\circ]$-representation of~$\G$ with respect to $\mathfrak w$.
	\end{lemma}
	\begin{proof}
		Let $\alpha$ be a surjective generator of $\Hom_{\R[\U]}(\pi,\Psi_\R)$.
		Suppose that $\phi$ is an element of 
		$$\Hom_{\R[\G/\G^\circ][\U]}(\pi\otimes\chi_{\G,\R},\Psi_{\R[\G/\G^\circ]}),$$
		we will prove that
		$\phi(v\otimes r)=cr\alpha(v)$ for some $c\in \R[\G/\G^\circ]$ and all $v\in\pi,r\in \R[\G/\G^\circ]$. Note that $\R[\G/\G^\circ]$ is isomorphic to a Laurent polynomial ring in $r$ variables over $\R$, and we fix an isomorphism
		$$\R[\G/\G^\circ]\simeq\R[X_1^{\pm1},\dotsc,X_r^{\pm1}].$$
		Now for each tuple $\underline{m}=(m_1,\dotsc,m_r)\in\mathbb Z^r$ we obtain a map $\phi_{\underline{m}}\in\Hom_{\R[\U]}(\pi,\Psi_\R)$ by setting $\phi_{\underline{m}}(v)$ to be the coefficient of $X_1^{m_1}\cdots X_r^{m_r}$ in $\phi(v\otimes 1)$. However, this implies that there is a $c_{\underline{m}}\in \R$ such that $\phi_{\underline{m}}=c_{\underline{m}}\cdot\alpha$. We obtain that
		$$\phi(v\otimes 1)=\left(\sum_{\underline{m}\in\mathbb Z^r}c_{\underline{m}}X_1^{m_1}\cdots X_r^{m_r}\right)\cdot\alpha(v)$$
		for all $v\in\pi$. Now since we assume $\alpha$ to be surjective, the $c_{\underline{m}}$ are nonzero for only finitely many tuples $\underline{m}\in\mathbb Z^r$ and we set $$c=\left(\sum_{\underline{m}\in\mathbb Z^r}c_{\underline{m}}X_1^{m_1}\cdots X_r^{m_r}\right)\in \R[X_1^{\pm1},\dotsc,X_r^{\pm1}].$$
		We obtain that 
		$$\phi(v\otimes a)=ca\cdot\alpha(v),$$
		for all $v\in\pi$ and $a\in \R[X_1^{\pm1},\dotsc,X_r^{\pm1}]$, which implies the result.
	\end{proof}
	
	\subsection{Helm's co-Whittaker machinery}\label{cwrepssect}
	Throughout this section we fix a Whittaker datum $\mathfrak w=(\U,\Psi)$ for $\G$ and a commutative $\R_0$-algebra $\R$.

	\begin{definition} Let $e_{\mathfrak t}$ be a finite sum of primitive idempotents of~$\mathfrak{Z}_{\R_0}(\G)$.
		\begin{enumerate}
			\item We call an~$\R$-representation~$\rho$ of~$\G$ \emph{co-Whittaker} with respect to~$\mathfrak w$ of type~$\mathfrak{t}$, if
			
			\begin{itemize}
				\item $\Hom_{\R[{\G}]}(\Gamma_{\mathfrak w,\R}^{\G},\rho)$ is free of rank one over~$\R$,
				\item $\rho=\rho_{\mathfrak{t}}$,
				\item the natural map
				\[\Gamma_{\mathfrak w,\R}^{\G}\otimes \Hom_{\R[{\G}]}(\Gamma_{\mathfrak w,\R}^{\G},\rho)\rightarrow \rho\]
				is surjective.
			\end{itemize}
			\item Given a co-Whittaker~$\R$-representation~$\rho$ of~$\G$ with respect to~$\mathfrak w$ of type~$\mathfrak{t}$ we have a natural morphism of algebras~
			\[f_{\rho}\colon\E^{\G}_{\mathfrak w,\R,\mathfrak{t}}\rightarrow \R\]
			given by the action of~$\E^{\G}_{\mathfrak w,\R,\mathfrak{t}}$ on~$ \Hom_{\R[\G]}(\Gamma^{\G}_{\mathfrak w,\R,\mathfrak{t}},\rho)= \Hom_{\R[\G]}(\Gamma^{\G}_{\mathfrak w,\R},\rho)$ by pre-composition. 
		\end{enumerate}
	\end{definition}
	
	Because~$\E^{\G}_{\mathfrak w,\R,\mathfrak{t}}\simeq \E^{\G}_{\mathfrak w,\mathfrak{t}}\otimes \R$, we have the canonical bijection, between the~$\R$-points of~$\E^{\G}_{\mathfrak w,\R,\mathfrak{t}}$ and the~$\R$-points of~$\E^{\G}_{\mathfrak w,\mathfrak{t}}$.  (Explicitly given by
	\begin{align*}
		\Hom_{\R\text{-alg}}(\E^{\G}_{\mathfrak w,\R,\mathfrak{t}},\R)&\simeq \Hom_{\R_0\text{-alg}}(\E^{\G}_{\mathfrak w,\mathfrak{t}},\R)\\
		f&\mapsto f_0:e_0\mapsto f(e_0\otimes 1), \end{align*}
	with its inverse given by~$f_0\mapsto f:e_0\otimes r\mapsto r f_0(e_0)$.)  Given a co-Whittaker~$\R$-representation~$\rho$, the~$\R$-point~$f_{\rho}\colon\E^{\G}_{\mathfrak w,\R,\mathfrak{t}}\rightarrow \R$ corresponds to the~$\R$-point~$f_{\rho,0}\colon\E^{\G}_{\mathfrak w,\mathfrak{t}}\rightarrow \R$.
		
	It is straightforward to see that the~$\E^{\G}_{\mathfrak w,\mathfrak{t}}$-representation $\Gamma^{\G}_{\mathfrak w,\mathfrak{t}}$ of $\G$ is co-Whittaker with respect to $\mathfrak w$ of type~$\mathfrak{t}$. Following~Helm, we call this representation the \emph{universal co-Whittaker representation} with respect to $\mathfrak w$ of type~$\mathfrak{t}$.   
	
	Given a morphism of~$\R_0$-algebras
	\[\phi_0\colon\E^{\G}_{\mathfrak w,\mathfrak{t}}\rightarrow \R\]
	we can consider the~$\R$-representation
	\[(\Gamma^{\G}_{\mathfrak w,\mathfrak{t}})_{\phi_0}\coloneqq\Gamma^{\G}_{\mathfrak w,\mathfrak{t}}\otimes_{\phi_0} \R.\]
	We have~$(\Gamma^{\G}_{\mathfrak w,\mathfrak{t}})_{\phi_0}\simeq (\Gamma^{\G}_{\mathfrak w,\R,\mathfrak{t}})_{\phi}$ by the natural map, and hence 
	\[\Hom_{\R[\G]}(\Gamma^{\G}_{\mathfrak w,\R,\mathfrak{t}},(\Gamma^{\G}_{\mathfrak w,\R,\mathfrak{t}})_{\phi})\simeq \Hom_{\R[\G]}(\Gamma^{\G}_{\mathfrak w,\R,\mathfrak{t}},\Gamma^{\G}_{\mathfrak w,\R,\mathfrak{t}}\otimes_{\E^{\G}_{\mathfrak w,\R,\mathfrak{t}},\phi} \R)\simeq \E^{\G}_{\mathfrak w,\R,\mathfrak{t}} \otimes_{\E^{\G}_{\mathfrak w,\R,\mathfrak{t}},\phi} \R\simeq \R;\]
	and it follows that~$(\Gamma^{\G}_{\mathfrak w,\mathfrak{t}})_{\phi_0}$ is a co-Whittaker~$\R$-representation of~$\G$ of type~$\mathfrak{t}$.
	
	Conversely, given a co-Whittaker~$\R$-representation~$\rho$ of~$\G$ of type~$\mathfrak{t}$, we can consider
	\[(\Gamma^{\G}_{\mathfrak w,\mathfrak{t}})_{f_{\rho,0}}\simeq \Gamma^{\G}_{\mathfrak w,\R,\mathfrak{t}}\otimes_{f_{\rho}} \R.\]
	In general,~$(\Gamma^{\G}_{\mathfrak w,\mathfrak{t}})_{f_{\rho,0}}$ is not isomorphic to $\rho$, however the defining properties of co-Whittaker representations imply that~$(\Gamma^{\G}_{\mathfrak w,\mathfrak{t}})_{f_{\rho,0}}$ surjects onto~$\rho$, which gives rise to Helm's notion of \emph{dominance}.
	\begin{definition}
		Let $\rho$ and $\rho'$ be two co-Whittaker~$\R$-representations of~$\G$ with respect to $\mathfrak w$ of type~$\mathfrak{t}$. We say that~$\rho$ \emph{dominates}~$\rho'$ if there is a surjection~$\rho\twoheadrightarrow \rho'$ of~$\R$-representations of~$\G$.
	\end{definition}
	 Suppose now that $\rho$ dominates $\rho'$. Then the surjection $\rho\twoheadrightarrow\rho'$ induces an isomorphism
	\[\Hom_{\R[\G]}(\Gamma_{\mathfrak w,\R,\mathfrak t}^{\G},\rho)\xrightarrow{\simeq} \Hom_{\R[\G]}(\Gamma_{\mathfrak w,\R,\mathfrak t}^{\G},\rho')\]
	and hence~$f_{\rho}=f_{\rho'}$.  
	
	We recall some properties of co-Whittaker representations:
	
	\begin{lemma}\label{cwlemma}
		Let~$\rho$ be a co-Whittaker~$\R$-representation of~$\G$ with respect to~$\mathfrak w=(\U,\Psi)$ of type~$\mathfrak{t}$. Then
		\begin{enumerate}
			\item\label{cwlemma1}
			$\rho$ is of Whittaker type with respect to $\mathfrak w$. 
			\item \label{cwlemma2} $\rho$ is finitely generated as an~$\R[\G]$-module. Moreover if~$\R$ is Noetherian then~$\rho$ is admissible as an~$\R[\G]$-module.
			
			\item \label{cwlemma3}$\rho$ satisfies Schur's lemma (in other words, the natural map~$\R\rightarrow \End_{\R[\G]}(\rho)$ is an isomorphism), and~$\rho$ has a central character.
			\item \label{cwlemma4}for an $\R_0$-algebra morphism $f\colon\R\rightarrow \R'$, the $\R'$-representation $\rho\otimes_f \R'$ is co-Whittaker with respect to~$\mathfrak w$ of type~$\mathfrak{t}$.
		\end{enumerate}
	\end{lemma}
	\begin{proof}
		\begin{enumerate}
			\item By Proposition \ref{propcoinv} we obtain that the coinvariants $\rho_{\U,\Psi_\R}$ are free of rank one as an $\R$-module. Hence, 
			$\Hom_{\R[\U]}(\rho_{\U,\Psi_\R},\Psi_\R)\simeq\Hom_{\R[\U]}(\rho,\Psi_\R)$
			is also free of rank one as an $\R$-module. Moreover, the map which sends a generator of $\rho_{\U,\Psi_\R}$ to one is clearly surjective.
			\item Since $\rho$ is a quotient of the finitely generated representation~$\Gamma^{\G}_{\mathfrak w,\R,\mathfrak{t}}$, it is finitely generated.  Now assume that~$\R$ is Noetherian.  By \cite[Lemma 3.2]{DHKMfiniteness} and \cite[Theorem 1.2]{DHKMconj}, it follows that~$\rho$ is~$\mathfrak{Z}_{\R}(\G)$-finite.  In particular,~$\rho$ is admissible over~$e_{\mathfrak{t}}\mathfrak{Z}_{\R}(\G)$; i.e., for any compact open subgroup~$\H$ of~$\G$,~$\rho^{\H}$ is a finitely generated~$e_{\mathfrak{t}}\mathfrak{Z}_{\R}(\G)$-module.  Finally, as~$\rho$ is co-Whittaker, the action of~$e_{\mathfrak{t}}\mathfrak{Z}_\R(\G)$ on~$\rho$ factors through the natural map~$\mathfrak{Z}_\R(\G)\rightarrow \E^{\G}_{\mathfrak w,\R,\mathfrak{t}}\xrightarrow{f_{\rho}} \R$ which implies the result.
			\item The action of~$\End_{\R[\G]}(\rho)$ on~$\Hom_{\R[\G]}(\Gamma_{\mathfrak w,\R}^{\G},\rho)$ defines the morphism~$\Phi\colon\End_{\R[\G]}(\rho)\rightarrow \R$ satisfying, for all~$f\in \Hom_{\R[\G]}(\Gamma_{\mathfrak w,\R}^{\G},\rho)$,~$\alpha\in \End_{\R[\G]}(\rho)$,
			\[\alpha\circ f=\Phi(\alpha) f.\]
			We claim this map is injective.  Indeed, if~$\Phi(\alpha) f=0$ for all~$f\in \Hom_{\R[\G]}(\Gamma_{\mathfrak w,\R,\mathfrak t}^{\G},\rho)$ then~$\alpha \circ f=0$ for all such~$f$, and hence as~$\rho$ is a quotient of~$\Gamma_{\mathfrak w,\R,\mathfrak t}^{\G}$ we have~$\alpha=0$.  Hence we have defined an injection~$\End_{\R[\G]}(\rho)\rightarrow \R$ inverse to the natural map~$\R\rightarrow \End_{\R[\G]}(\rho)$.
			
			For an element $z$ of the centre~$\Z(\G)$ of~$\G$, then~$\rho(z)$ is an element of $\End_{\R[\G]}(\rho)= \R$, and hence~$\rho$ has a central character~$\omega_\rho\colon\Z(\G)\rightarrow \R^\times$ such that~$\rho(z)v=\omega_{\rho}(z)v$ for all~$z\in\Z(\G),v\in \rho$; (note the image of~$\omega_{\rho}$ lies in~$\R^\times$ as~$\omega_{\rho}(z)\omega_{\rho}(z^{-1})=1$).
			\item Because~$\Gamma^{\G}_{\mathfrak w,\R,\mathfrak{t}}$ is finitely generated and projective, by \cite[Lemma A.1]{DHKMconj}, the natural map~$f\otimes r'\mapsto r'(f\otimes 1)$ defines an isomorphism
			\[\Hom_{\R[\G]}(\Gamma^{\G}_{\mathfrak w,\R,\mathfrak{t}},\rho)\otimes \R'\simeq \Hom_{\R'[\G]}(\Gamma^{\G}_{\mathfrak w,\R,\mathfrak{t}}\otimes \R',\rho\otimes \R')\]
			As~$\Gamma^{\G}_{\mathfrak w,\R,\mathfrak{t}}\otimes \R'=\Gamma^{\G}_{\mathfrak w,\R',\mathfrak{t}}$, it follows that~$\Hom_{\R'[{\G}]}(\Gamma^{\G}_{\mathfrak w,\R',\mathfrak{t}},\rho\otimes \R')$ is free of rank one over~$\R'$, and as the tensor product is right exact, we find~$\rho\otimes \R'$ is co-Whittaker.\end{enumerate}\end{proof}
	
	
	Note that when $\R$ is a field an irreducible $\R$-representation~$\rho$ of~$\G$ being of Whittaker type with respect to $\mathfrak w$, is equivalent to $\rho$ being \emph{$\mathfrak w$-generic}, where we call an~$\R$-representation~$\rho$ to be \emph{~$\mathfrak{w}$-generic} if it is irreducible and $\rho_{\U,\Psi_\R}$ is one dimensional as an $\R$-vector space.
	
	\begin{lemma}\label{cowhittakerfields}
		Let~$\R$ be a field and~$\R_0$-algebra, and~$\rho,\rho'$ be co-Whittaker~$\R$-representations of~$\G$ with respect to~$\mathfrak w$ of type~$\mathfrak{t}$.
		\begin{enumerate}
			\item\label{cowhittakerfields1} Then~$\rho$ has finite length, and the cosocle of~$\rho$ is absolutely irreducible and~$\mathfrak w$-generic, and is the unique~$\mathfrak w$-generic irreducible subquotient of~$\rho$.  
			\item\label{cowhittakerfields2} Suppose~$\R$ is algebraically closed.  Write~$\rho_{gen}$ (resp.~$\rho'_{gen}$) for the unique generic (sub)quotient of~$\rho$ (resp.~$\rho'$).  Then~$f_{\rho,0}=f_{\rho',0}$ if and only if~$\rho_{gen}\simeq\rho'_{gen}$.
		\end{enumerate}
	\end{lemma}
	\begin{proof}
		\begin{enumerate}
			\item 

		The~$R$-representation~$\Pi=(\Gamma^{\G}_{\mathfrak w,\mathfrak{t}})_{f_{\rho,0}}$ dominates~$\rho$.  By Lemma \ref{cwlemma} (\ref{cwlemma2}),~$\Pi$ is admissible and finitely generated, hence of finite length (cf.\ II.5.10 in \cite{Vig96}), and hence~$\rho$ too is of finite length.  It follows that the head of~$\Pi$ is irreducible (as it is a semisimple co-Whittaker quotient of~$\Gamma_{\mathfrak w,\R,\mathfrak{t}}^{\G}$) and~$\mathfrak w$-generic, hence the head of~$\rho$ is also (irreducible and)~$\mathfrak w$-generic.  
		Now let~$\overline{\R}$ denote an algebraic closure of~$\R$.   Moreover, from Lemma \ref{cwlemma} \ref{cwlemma3}, the cosocle of~$\rho$ satisfies Schur's lemma and hence is absolutely irreducible.
		The functor of~$(\U,\Psi_\R)$-coinvariants is exact (which follows as~$\U$ is a colimit of pro-$p$-groups, cf., \cite[I 4.11]{Vig96}), and as~$\rho_{\U,\Psi_\R}\simeq \R$ by Proposition \ref{propcoinv}, we obtain that~$\rho$ has a unique~$\mathfrak w$-generic subquotient.

		\item The~$\R$-representation~$(\Gamma^{\G}_{\mathfrak w,\mathfrak{t}})_{f_{\rho,0}}$ is a co-Whittaker~$\R$-representation of~$\G$ of type~$\mathfrak{t}$ and hence of finite length by the first part. If~$f_{\rho,0}=f_{\rho',0}$, then~$(\Gamma^{\G}_{\mathfrak w,\mathfrak{t}})_{f_{\rho,0}}$ dominates both~$\rho$ and~$\rho'$. Thus~$\rho_{gen},\rho'_{gen}$ are (non-zero) summands of the cosocle of~$(\Gamma^{\G}_{\mathfrak w,\mathfrak{t}})_{f_{\rho,0}}$. By the first point the cosocle of~$(\Gamma^{\G}_{\mathfrak w,\mathfrak{t}})_{f_{\rho,0}}$ is irreducible and hence~$\rho_{gen}\simeq\rho'_{gen}$. 
		
		 Conversely, the quotient~$\rho\rightarrow \rho_{gen}$, defines an~$\E_{\mathfrak w,\R,\mathfrak{t}}^{\G}$-equivariant isomorphism~$\Hom_{\R[\G]}(\Gamma^{\G}_{\mathfrak w,\R,\mathfrak{t}},\rho)\simeq \Hom_{\R[\G]}(\Gamma^{\G}_{\mathfrak w,\R,\mathfrak{t}},\rho_{gen})$ and hence~$f_{\rho,0}=f_{ \rho_{gen},0}$.		\end{enumerate}\end{proof}
	
		For a fixed~$\R_0$-algebra structure on an algebraically closed field we have the following bijection:
	\begin{corollary}\label{corbijection}
	Let~$\R$ be an algebraically closed field and~$\R_0$-algebra.  Then the map from the set of isomorphism classes of generic~$\R$-representations of~$\G$ of type~$\mathfrak{t}$ to the set of the~$\R$-points of~$\E^{\G}_{\mathfrak{w},\mathfrak{t}}$, defined by~$\pi\mapsto f_{\pi,0}$, defines a bijection.
	\end{corollary}
	
	\begin{proof}
	From Lemma \ref{cowhittakerfields} \eqref{cowhittakerfields2}, the map is injective.   Conversely, given~$\phi\in\Hom_{\R_0\text{-alg}}(\E^{\G}_{\mathfrak{w},\mathfrak{t}},\R)$, the co-Whittaker representation~$\Gamma^{\G}_{\mathfrak{w},\mathfrak{t}}\otimes_{\E^{\G}_{\mathfrak{w},\mathfrak{t}},\phi}\R$ has $\mathfrak w$-generic cosocle~$\rho$ and~$f_{\rho,0}=\phi$.
	\end{proof}

	Let~$\R$ be a discrete valuation ring with field of fractions~$K$, and~$(\rho,V)$ be a~$K$-representation of~$\G$. If there exists a~$\G$-stable~$\R$-submodule~$L$ of~$V$, which is free as an~$\R$-module, and contains a~$K$-basis of~$V$, then we say that~$\rho$ is \emph{integral} and call~$L$ an~$\R$-lattice in~$\rho$.
	
	
	\begin{lemma}\label{cowhittakerDVR}
		Let~$\R$ be a discrete valuation ring and~$\R_0$-algebra, with field of fractions~$K$, maximal ideal~$\mathfrak{m}$, and residue field~$\kappa(\mathfrak{m})$.  Let~$\rho$ be a co-Whittaker~$K$-representation of~$\G$ with respect to~$\mathfrak w$ of type~$\mathfrak{t}$.
		\begin{enumerate}
			\item \label{cowhittakerDVR1} Then~$\rho$ is~$\R$-integral if and only if~$f_{\rho,0}$ takes values in~$\R$.  Moreover, in this case, there exists a co-Whittaker~$\R$-lattice in~$\rho$.
			\item\label{cowhittakerDVR2} If~$\rho$ is~$\R$-integral with co-Whittaker~$\R$-lattice~$L$, then under~$\kappa(\mathfrak{m})\simeq \R\otimes \kappa(\mathfrak{m})$ (applied to the values of the morphisms), we have	\[f_{L\otimes \kappa(\mathfrak{m}),0}=f_{\rho,0}\otimes 1.\]
			Furthermore, let~$J(\rho)$ denote the cosocle of~$L\otimes\kappa(\mathfrak{m})$.  Then
			\[f_{J(\rho),0}=f_{\rho,0}\otimes 1.\]
		\end{enumerate}
	\end{lemma}
	
	\begin{proof}
		\begin{enumerate}
			\item If~$f_{\rho,0}$ takes values in~$\R$, then 
			the image of~$\Gamma^{\G}_{\mathfrak w,\mathfrak{t}}\otimes_{f_{\rho,0}} \R$ under the quotient map defines an~$\R$-lattice in~$\rho$, by \cite[I 9.3]{Vig96}, and hence~$\rho$ is integral.
			
			Conversely, suppose~$\rho$ is integral, and let~$L$ be an~$\R$-lattice in~$\rho$.  Then we claim that~$\Phi\coloneqq\Hom_{\R[\G]}(\Gamma^{\G}_{\mathfrak w,\R,\mathfrak{t}},L)\simeq \R$.  Indeed,~$\Phi$ is a free~$\R$-module because~$\Gamma^{\G}_{\mathfrak w,\R,\mathfrak{t}}$ is finitely generated as an~$\R[\G]$-module and $L$ is free as an~$\R$-module and~$\R$ is a principal ideal domain. Hence, as~$\Phi\otimes K\simeq K$ by \cite[Lemma A.1]{DHKMconj}, we deduce~$\Phi\simeq \R$.  As the action of~$\E_{\mathfrak w,\mathfrak{t}}^{\G}$ on~$\Hom_{K[\G]}(\Gamma^{\G}_{\mathfrak w,\mathfrak{t}}\otimes K,\rho)$ stabilizes~$\Hom_{\R[\G]}(\Gamma^{\G}_{\mathfrak w,\mathfrak{t}}\otimes \R,L)$ it must be through a morphism~$\E^{\G}_{\mathfrak w,\mathfrak t}\rightarrow \R$.
			\item 		
			Note that
			\[\Hom_{\R[\G]}(\Gamma^{\G}_{\mathfrak w,\R,\mathfrak{t}},L)\otimes\kappa(\mathfrak{m})\simeq \Hom_{\R[\G]}(\Gamma^{\G}_{\mathfrak w,\R,\mathfrak{t}}\otimes \kappa(\mathfrak{m}),L\otimes \kappa(\mathfrak{m})).\]
			This final part now follows from Lemma \ref{cowhittakerfields}.
		\end{enumerate}
		
	\end{proof}

	\section{Shahidi's local coefficients and Langlands-Shahidi gamma factors in families}
	For this section, we let
	\begin{equation*}
		\R_0=\begin{cases}\mathbb{Z}[\sqrt{q}^{-1}, \mu_{p^\infty}]&\text{if char}(\F)=0\text{ and }~\G\text{ is split;}\\\mathbb{Z}[\sqrt{q}^{-1}, \mu_{4p^\infty}]&\text{if char}(\F)=0\text{ and }~\G\text{ is quasi-split non-split;}\\\mathbb{Z}[\sqrt{q}^{-1},\mu_{p}]&\text{if char}(\F)=p\text{ and }~\G\text{ is split;}\\
		\mathbb{Z}[\sqrt{q}^{-1},\mu_{4p}]&\text{if char}(\F)=p\text{ and }~\G\text{ is quasi-split non-split.}
		\end{cases}
	\end{equation*}
	We let~$\R$ be a~$\R_0$-algebra.
	
	We adjoin~$\sqrt{q}^{-1}$ for the convenience of being able to normalize parabolic induction.  The fourth roots of unity in the quasi-split case allow a normalization in the construction of the gamma factor by a certain product of Langlands $\lambda$-functions (cf.\ Remark \ref{remlambda}).

	\subsection{Whittaker functionals on parabolic inductions}\label{Whittfunctionalssect}
	We fix an~$\F$-pinning~$\Xi$ of $\G$, and a character~$\psi\colon\F\rightarrow \R_0^\times$. Following Shahidi (see also \cite[A.5]{AGIKMS}) these data define a Whittaker datum~$\mathfrak w=(\U,\Psi)$ for $\G$, as in Section \ref{whittdata}.
	
	Let $\theta$ be a subset of $\Delta$ such that there is a simple root $\alpha$ such that $\theta=\Delta\backslash\{\alpha\}$. Then we can associate to $\theta$ a standard maximal parabolic subgroup $\P_\theta$ of $\G$ with a standard Levi decomposition $\P_\theta=\M_\theta\N_\theta$.
	
	Let $w_\ell=w_\ell^\Delta$ be the longest element in $\W$ and $w_\ell^{\theta}$ be the longest element in $\W_\theta$. We set $w_0=w_\ell w_\ell^{\theta}$. Note that then $w_0(\theta)\subset\Delta$ and $w_0(\alpha)<0$. We choose the Langlands--Shelstad Weyl group representative $\widetilde{w}_0\in\N_{\G}(\mathrm{S})$ determined by~$\Xi$ (see also \cite[Section 1.7]{AGIKMS}). Let $\P_{w_0(\theta)}$ be the standard maximal parabolic subgroup associated to $w_0(\theta)$ and let $\N_{w_0(\theta)}$ be its unipotent radical. We fix a nontrivial Haar measure $\nu_{\N_{w_0(\theta)}}$ on $\N_{w_0(\theta)}$.

	\begin{remark}\label{remlambda}[Langlands $\lambda$-functions]
		Following~\cite[Section 4.1]{KeysShahidi} (see also \cite[Section 1.7]{AGIKMS}, and \cite[p.\ 18]{lomelilemma}), we let~$\lambda(w_0,\psi)$ denote a particular finite product of Langlands $\lambda$-factors (each of which is associated to a particular finite extension of~$\F$ and~$\psi$).  By \cite[Remark p195]{BHBook}, for a finite extension~$\E/\F$, the Langlands~$\lambda$-factor $\lambda_{\E/\F}(\psi)$ is a fourth root of unity.  Hence~$\lambda(w_0,\psi)$ is a fourth root of unity.  As~$\lambda(w_0,\psi)$ will appear later in decompositions of Shahidi we use, we need to ensure our coefficient ring contains these values. 
		
	\end{remark}

	As mentioned in Section \ref{whittdata}, by restriction, $\mathfrak w$ gives rise to a Whittaker datum $\mathfrak w_{\M_\theta}=(\U_{\M_{\theta}},\Psi_{\M_\theta})$ for $\M_\theta$, where recall that we set $\U_{\M_\theta}=\U\cap\M_\theta$. Note that $\Psi$ and $\Psi_{\M_\theta}$ are $\widetilde{w}_0$-compatible, i.e., the equation \[\Psi_{\M_\theta}(u)=\Psi(\widetilde{w}_0u\widetilde{w}_0^{-1})\] holds for $u\in\U_{\M_{\theta}}$.
	
	For an $\R_0$-algebra $\R$ we set $\Psi_{\M_\theta,\R}=\Psi_{\M_\theta}\otimes\R$. Let~$\rho$ be a Whittaker type~$\R$-representation of~$\M_\theta$ with respect to $\mathfrak w_{\M_\theta}$, and fix a surjective~$\R$-module generator
	\[\vartheta\in\Hom_{\R[\U\cap\M_{\theta}]}(\rho,\Psi_{\M_\theta,\R}).\]  
	
	The goal of this section is to prove that the $\R[\M_\theta/\M_\theta^{\circ}]$-representation~$i_{\M_\theta,\P_\theta}^{\G}(\rho\otimes\chi_{\M_\theta,\R})$ of $\G$ is of Whittaker type with respect to $\mathfrak w$. To prove this we use ideas similar to the geometric lemma. Note that $\P_\theta w_0^{-1}\B$ is an open double coset in~$\P_\theta\backslash \G/\B$. By \cite[Proposition 1.15]{bernstein1976representations} we can find an ordering $w_0\leqslant w_1\leqslant\dotsc \leqslant w_r$ on a set of Weyl representatives $[\W/\W_{\theta}]$
	such that $\bigcup_{i\leqslant j}\P_\theta w_i^{-1} \B$ is open in $\G$ for all $0\leqslant j\leqslant r$ (note that this is not the Bruhat ordering of Section \ref{parabolics}).
	Let $J_j$ be the $\R[\M_{\theta}/\M_{\theta}^\circ][\B]$-submodule of $i_{\M_{\theta},\P_\theta}^{\G}(\rho\otimes\chi_{\M_{\theta},\R})$ consisting of those functions whose support lies in $\bigcup_{i\leq j}\P_\theta w_i^{-1} \B$. For $J_0$ we have the following result.

	\begin{lemma}\label{lemmaconstfunction}
		For~$f\in  J_0$,  the integral
		\begin{align}
			\label{EQ1} 
			\Omega_{\mathfrak w}(\rho,\vartheta,\nu_{\N_{w_0(\theta)}})(f)=\int_{\N_{w_0(\theta)}} \vartheta(f(\widetilde{w}_{0}^{-1}n))\Psi_\R^{-1}(n)dn
		\end{align}
		is well-defined, and~$\Omega_{\mathfrak w}(\rho,\vartheta,\nu_{\N_{w_0(\theta)}})(u\cdot f)=\Psi_\R(u)\Omega_{\mathfrak w}(\rho,\vartheta,\nu_{\N_{w_0(\theta)}})(f)$ for $u\in\U$. Moreover, there is a function $f_0\in J_0$ such that $\Omega_{\mathfrak w}(\rho,\vartheta,\nu_{\N_{w_0(\theta)}})(f_0)=1$.
	\end{lemma}

	\begin{proof}
		For the first part, it suffices to see that for~$f\in J_0$, the map~$n\mapsto f(\widetilde{w}_{0}^{-1}n)$ has compact support on $\N_{w_0(\theta)}$. The map $b\mapsto\widetilde{w}_0^{-1}b$ yields a homeomorphism $(w_0\P_\theta w_0^{-1}\cap \B)\backslash \B\simeq\P_\theta\backslash \P_\theta w_0^{-1}\B$. The latter is isomorphic to $(w_0\P_\theta w_0^{-1}\cap \U)\backslash \U$ and since $\U=\N_{w_0(\theta)}(\M_{w_0(\theta)}\cap \U)$ and $\M_{w_0(\theta)}\cap \U=w_0\P_\theta w_0^{-1}\cap \U$ we obtain that the map 
		\begin{align*}
			\N_{w_0(\theta)}&\to \P_\theta\backslash \P_\theta w_0^{-1}\B\\
			n&\mapsto\widetilde{w}_0^{-1}n
		\end{align*}
		is a homeomorphism. Now as $\operatorname{supp}(f)$ is open and closed in $\G$ and the map $\G\to \P_\theta\backslash \G$ is open, it is straightforward to see that $\P_\theta\backslash\operatorname{supp}(f)$ is closed in $\P_\theta\backslash\G$. As $\P_\theta\backslash\G$ is compact we obtain that $\P_\theta\backslash\operatorname{supp}(f)$ is compact. The result now follows since by assumption $\P_\theta\backslash\operatorname{supp}(f)$ is contained in $\P_\theta\backslash \P_\theta w_0^{-1}\B$ and by the above homeomorphism.
		
		The second statement is a straightforward computation -- see Proposition 8.9 in Kim's article in \cite{cogdell2004lectures}.
		
		For the third statement let $\N_0$ be a compact open subgroup of $\N_{w_0(\theta)}$ such that $\Psi$ is trivial on $\N_0$ and $\operatorname{vol}(\N_0)\in\R_0^\times$. By assumption there is a $v_0\in\rho$ such that $\vartheta(v_0)=1$ and let $f_0$ be the function in $J_0$ that under the above homeomorphism corresponds to the function on $\N_{w_0(\theta)}$ that is supported on $\N_0$ with constant value $\operatorname{vol}(\N_0)^{-1}v_0$ there. Then clearly
		$$\Omega_{\mathfrak w}(\rho,\vartheta,\nu_{\N_{w_0(\theta)}})(f_0)=\int_{\N_{w_0(\theta)}}\vartheta(f_0(\widetilde{w}_{0}^{-1}n))\Psi_\R^{-1}(n)dn=\frac{1}{\operatorname{vol}(\N_0)}\int_{\N_0}\vartheta(v_0)\Psi_\R^{-1}(n)dn=1.$$
	\end{proof}
	
	We briefly discuss the dependence of $\Omega_{\mathfrak w}(\rho,\vartheta,\nu_{\N_{w_0(\theta)}})$ on the choices of $\vartheta$ and $\nu_{\N_{w_0(\theta)}}$. By assumption~$\Hom_{\R[\U_{\M_\theta}]}(\rho,\Psi_{\M_\theta,\R})$ is free of rank one, so when choosing another generator $\vartheta'$, we obtain that $\vartheta'=c\cdot\vartheta$ for some unit $c\in \R^\times$. Similarly, if $\nu_{\N_{w_0(\theta)}}'$ is another nontrivial Haar measure on $\N_{w_0(\theta)}$ then $\nu_{\N_{w_0(\theta)}}'=d\nu_{\N_{w_0(\theta)}}$ for some $d\in \R^\times$. Then $\Omega_{\mathfrak w}(\rho,\vartheta',\nu_{\N_{w_0(\theta)}}')=cd\cdot\Omega_{\mathfrak w}(\rho,\vartheta,\nu_{\N_{w_0(\theta)}})$. In the following we will often fix choices of $\vartheta$ and $\nu_{\N_{w_0(\theta)}}$ and write $\Omega_{\mathfrak w}(\rho)$ for $\Omega_{\mathfrak w}(\rho,\vartheta,\nu_{\N_{w_0(\theta)}})$.
	
	The above lemma in particular shows that $\Omega_{\mathfrak w}(\rho)$ gives rise to a non-zero element in 
	$$\Hom_{\R[\M_\theta/\M_\theta^\circ]}((J_0)_{\U,\Psi_\R},\R[\M_\theta/\M_\theta^{\circ}])\simeq \Hom_{\R[\M_\theta/\M_\theta^\circ][\U]}(J_0,\Psi_{\R[\M_\theta/\M_\theta^{\circ}]}),$$
	for which we have the following result.
	\begin{lemma}
		The space~$ \Hom_{\R[\M_\theta/\M_\theta^\circ]}((J_0)_{\U,\Psi_{\R}},\R[\M_\theta/\M_\theta^{\circ}])$ is free of rank one over~$\R[\M_\theta/\M_\theta^{\circ}]$ and generated by~$\Omega_{\mathfrak w}(\rho)$.
	\end{lemma}

	\begin{proof}
		Note that $f\mapsto (x\mapsto f(\widetilde{w}_0^{-1}x))$ yields an isomorphism of $\R[\M_\theta/\M_\theta^\circ][\U]$-modules between $J_0$ and $\operatorname{c-Ind}_{w_0\M_\theta w_0^{-1}\cap \U}^\U({}^{\widetilde{w}_0}(\rho\otimes\chi_{\M_\theta,\R}))$, where we note that $w_0\P_\theta w_0^{-1}\cap \U=w_0\M_\theta w_0^{-1}\cap \U$. We can write
		$$(J_0)_{\U,\Psi_\R}\simeq\Psi_\R^{-1}\otimes_{\R[\U]}\operatorname{c-Ind}_{w_0\P_\theta w_0^{-1}\cap \U}^\U({}^{\widetilde{w}_0}(\rho\otimes\chi_{\M_\theta,\R})).$$
		and use Propositions 3.1 and 3.3 of \cite{Girsch} to see that
		$$\Psi_\R^{-1}\otimes_{\R[\U]}\operatorname{c-Ind}_{w_0\P_\theta w_0^{-1}\cap \U}^\U({}^{\widetilde{w}_0}(\rho\otimes\chi_{\M_\theta,\R}))\simeq\Psi_\R^{-1}\otimes_{\R[w_0\M_\theta w_0^{-1}\cap\U]}({}^{\widetilde{w}_0}(\rho\otimes\chi_{\M_\theta,\R}))$$
		As $\M_\theta\cap w_0^{-1}\U w_0=\M_\theta\cap \U$
		and $\Psi$ and $\Psi_{\M_\theta}$ are $\widetilde{w}_0$-compatible it is straightforward to see that
		$$\Psi_\R^{-1}\otimes_{\R[w_0\M_\theta w_0^{-1}\cap\U]}({}^{\widetilde{w}_0}(\rho\otimes\chi_{\M_\theta,\R}))=(\rho\otimes\chi_{\M_\theta,\R})_{\M_\theta\cap\U,\Psi_{\M_\theta,\R}}$$
		and the first result now follows from Lemma \ref{lemwhituniv} since
		$$\Hom_{\R[\M_\theta/\M_\theta^\circ]}((\rho\otimes\chi_{\M_\theta,\R})_{\U\cap\M_\theta,\Psi_{\M_\theta,\R}},\R[\M_\theta/\M_\theta^\circ])\cong\Hom_{\R[\M_\theta/\M_\theta^\circ][\U\cap\M_\theta]}(\rho\otimes\chi_{\M_\theta,\R},\Psi_{\R[\M_\theta/\M_\theta^{\circ}]}).$$
		
		By Lemma \ref{lemmaconstfunction} there is an $f_0$ such that $\Omega_{\mathfrak w}(\rho)(f_0)=1$ which immediately implies that $\Omega_{\mathfrak w}(\rho)$ is a generator of $\Hom_{\R[\M_\theta/\M_\theta^\circ]}((J_0)_{\U,\Psi_\R},\R[\M_\theta/\M_\theta^{\circ}])$.
	\end{proof}

	We will now show $\Omega_{\mathfrak w}(\rho)$ can be extended to a functional on $i_{\M_\theta,\P_\theta}^{\G}(\rho\otimes \chi_{\M_\theta,\R})$. 
	\begin{lemma}
		For any~$j\geq 1$, we have~$(J_j/J_{j-1})_{\U,\Psi_\R}=0.$
	\end{lemma}

	\begin{proof}
		In the special case that $\R=\mathbb C$, Lomeli (cf.\ \cite[Lemma 2.4]{Lomeli}) checks that the corresponding $\mathbb C$-vector space is zero by computing its Hom-space into the complex numbers. That argument does not directly extend to our setting of $\R$-representations (an~$\R$-module can be non-zero, but have zero Hom-space~(as an~$\R$-module) into~$\R$), so we give a slightly different argument:
		Let $\widetilde{w}_j$ be any lift of $w_j$.
		Note that $(J_j/J_{j-1})_{\U,\Psi_\R}$ can be written as
		$\Psi_\R^{-1}\otimes_{\R[\U]}J_j/J_{j-1}$. 
		We have an $\R[\M_\theta/\M_\theta^\circ][\U]$-module isomorphism
		$$J_j/J_{j-1}\simeq\operatorname{c-Ind}_{w_j\P_\theta w_j^{-1}\cap \U}^\U\left((\rho\otimes\chi_{\M_\theta,\R})^{\widetilde{w}_j^{-1}}\right).$$
		By using Propositions 3.1 and 3.3 of \cite{Girsch} we can compute that 
		$$\Psi_\R^{-1}\otimes_{\R[\U]}\operatorname{c-Ind}_{w_j\P_\theta w_j^{-1}\cap \U}^\U\left((\rho\otimes\chi_{\M_\theta,\R})^{\widetilde{w}_j^{-1}}\right)\simeq \Psi_\R^{-1}\otimes_{\R[w_j\P_\theta w_j^{-1}\cap \U]}(\rho\otimes\chi_{\M_\theta,\R})^{\widetilde{w}_j^{-1}}.$$
		However, as in the proof of \cite[Lemma 2.4]{Lomeli} there is an element $x\in w_j\N_\theta w_j^{-1}\cap \U$ such that $\Psi_\R(x)\not=1$ and hence as we saw at the beginning of Section \ref{seccowhit} that $1-\Psi_\R(x)\in \R^\times$,  which implies that the space is zero.
	\end{proof}

	By exactness of $\Psi_\R$-coinvariance we obtain that $\left(i_{\M_\theta,\P_\theta}^{\G}(\rho\otimes\chi_{\M_\theta,\R})\right)_{\U,\Psi_\R}\simeq (J_0)_{\U,\Psi_\R}$ and we deduce:
	\begin{corollary}\label{generates}
		The~$\R[\M_\theta/\M_\theta^{\circ}]$-representation~$i_{\M_\theta,\P_\theta}^{\G}(\rho\otimes\chi_{\M_\theta,\R})$ is of Whittaker type with respect to $\mathfrak w$. Moreover, the free~$\R[\M_\theta/\M_\theta^{\circ}]$-module of rank one
		\[\Hom_{\R[\M_\theta/\M_\theta^{\circ}][\U]}(i_{\M_\theta,\P_\theta}^{\G}(\rho\otimes\chi_{\M_\theta,\R}),\Psi_{\R[\M_\theta/\M_\theta^{\circ}]})\]
		is generated by~$\Omega_{\mathfrak w}(\rho)$.
	\end{corollary}

	\subsection{Intertwining operators}\label{int}	
	Recall that $\P_\theta$ is the maximal standard parabolic subgroup of $\G$ that is associated to $\theta\subset\Delta$ and there is an $\alpha\in\Delta$ with $\theta=\Delta\backslash\{\alpha\}$.
	
	We denote by~$\rho_{\theta}$ half the sum of absolute roots of $\T$ in~$\N_\theta$.  We set~$a$ to be an absolute root which restricts to~$\alpha$, and~$a^{\vee}$ is the coroot associated to~$a$. Let $X$ be an indeterminate and we define a ring homomorphism
	\begin{align*}
		\Delta_{X,\R}\colon \R[\M_\theta/\M_\theta^\circ]&\to \R[X^{\pm\langle\rho_\theta,a^\vee\rangle^{-1}/2}]\\
		m&\mapsto X^{\mathrm{val}(\delta_{\P_\theta}(m))\langle\rho_\theta,a^\vee\rangle^{-1}/2},
	\end{align*}
	where we set for $r\in\mathbb Q$ that $\mathrm{val}(q^r)=-r$.
	We will often write $X^c$ for $X^{\langle\rho_\theta,a^\vee\rangle^{-1}/2}$.

	\begin{definition}
		\begin{enumerate}
			\item Let~$\mathcal{S}_{\M_\theta,\R}$ denote the multiplicative subset of~$\R[\M_\theta/\M_\theta^{\circ}]$ generated by elements $$\{a_0+a_1m+\cdots+a_r m^r\mid m\in\M_\theta/\M_\theta^{\circ}\text{ non-trivial},~a_i\in\R,~a_0,a_r\in \R^\times\}.$$
			\item We set,
			\[\mathcal{S}_{X^c,\R}=\{a_{r} X^{cr}+a_{(r+1)} X^{c(r+1)}+\cdots +a_{s-1}X^{c(s-1)}+a_s X^{cs}\mid r,s\in\mathbb Z, r\leq s\text{ and }a_{r},a_s\in \R^\times\},\]
			i.e.,~$\mathcal{S}_{X^c,\R}$ is the multiplicative subset of~$\R[X^{\pm c}]$ of elements with leading and trailing coefficient a unit.
		\end{enumerate}
	\end{definition}
	As before, we write~$\chi_{\M_\theta ,\R}\colon\M\rightarrow \R[\M_\theta/\M_\theta^{\circ}]$ for the \emph{universal unramified character}.
	
	Recall that $\P_{w_0(\theta)}$ is the maximal standard parabolic subgroup associated to $w_0(\theta)$ and $\nu_{\N_{w_0(\theta)}}$ is a nontrivial Haar measure on $\N_{w_0(\theta)}$. In the following we will often for convenience write $\M$ for $\M_\theta$. Note that $\P_\theta w_0^{-1}\P_{w_0(\theta)}w_0$ is open in $\G$.
	
	\begin{thm}
		\label{intertwiningops}
		Suppose that~$\R$ is a Noetherian $\mathbb{Z}[1/\sqrt{q}]$-algebra, and
		let $\rho$ be a finitely generated smooth~$\R[\M]$-module which satisfies Schur's lemma. Then there exists an~$\R[\M/\M^{\circ}][\G]$-morphism
		\[\mathbb{M}(\rho,\nu_{\N_{w_0(\theta)}})\colon i_{\M,\P_\theta}^{\G}(\rho\otimes \chi_{\M,\R})\rightarrow i_{w_0\M w_0^{-1},\P_{w_0(\theta)}}^{\G}((\rho\otimes \chi_{\M,\R})^{\widetilde{w}_0^{-1}})\otimes_{\R[\M/\M^{\circ}]} \mathcal{S}_{\M,\R}^{-1}\R[\M/\M^{\circ}]\]
		such that:
		\begin{enumerate}
			\item for all $f\in i_{\M,\P_\theta}^{\G}(\rho\otimes \chi_{\M,\R})$ that are supported on $\P_\theta w_0^{-1}\P_{w_0(\theta)}w_0$, the equality
			\[\mathbb{M}(\rho,\nu_{\N_{w_0(\theta)}})(f)(\widetilde{w}_0)=\int_{\N_{w_0(\theta)}} f(\widetilde{w}_0^{-1}n\widetilde{w}_0) dn\]
			holds,  
			\item and there exists a polynomial~$\mathcal Q\in \mathcal{S}_{\M,\R}$ such that
			\[\mathcal Q\cdot \mathbb{M}(\rho,\nu_{\N_{w_0(\theta)}})\colon i_{\M,\P_\theta}^{\G}(\rho\otimes \chi_{\M,\R})\rightarrow i_{w_0\M w_0^{-1},\P_{w_0(\theta)}}^{\G}((\rho\otimes \chi_{\M,\R})^{\widetilde{w}_0^{-1}})\]
			and the image of $\mathcal Q$ under $\Delta_{X,\R}$ lies in $\mathcal S_{X,\R}$.
		\end{enumerate}
	\end{thm}
	
	\begin{proof}
		To prove this we use the construction in Section 3.4 of \cite{MT}, which we first review:

		Note that $\M$ is a Levi subgroup of both $\P_\theta$ and $w_0^{-1}\P_{w_0(\theta)}w_0$.
		We write $\overline{f}$ for the image of a function~$f\in i_{\M,\P_\theta}^{\G}(\rho\otimes \chi_{\M,\R})$ in $r_{\M,w_0^{-1}\P_{w_0(\theta)}w_0}^\G(i_{\M,\P_\theta}^{\G}(\rho\otimes \chi_{\M,\R}))$. Moreover, let $\mathcal I$ be the $\R[\M/\M^\circ] [w_0^{-1}\P_{w_0(\theta)}w_0]$-submodule of $i_{\M,\P_\theta}^{\G}(\rho\otimes \chi_{\M,\R})$ consisting of those functions whose support is contained in $\P_\theta w_0^{-1}\P_{w_0(\theta)}w_0$. We then have that $r_{\M,w_0^{-1}\P_{w_0(\theta)}w_0}^\G(\mathcal I)\cong\rho\otimes\chi_{\M,\R}$.
		
		Let $\A_\M$ be the split component in the centre of $\M$. By \cite[Corollary 5.4]{MT} there is an element~$\beta\in\R[\M/\M^\circ][\A_\M]$, such that~$\beta$ annihilates
		$$r_{\M,w_0^{-1}\P_{w_0(\theta)}w_0}^\G(i_{\M,\P_\theta}^{\G}(\rho\otimes \chi_{\M,\R})/\mathcal I).$$
		As one can see in the proof of \cite[Theorem 3.8]{MT},~$\beta$ acts on $r_{\M,w_0^{-1}\P_{w_0(\theta)}w_0}^\G(\mathcal I)\cong\rho\otimes\chi_{\M,\R}$ via multiplication with an element $b\in \mathcal S_{\M,\R}$. 
		
		Note that conjugation by $\widetilde{w}_0^{-1}$ induces a Haar measure on $w_0^{-1}\N_{w_0(\theta)}w_0$.
		This can be used to construct a $\R[\M/\M^\circ][\G]$-linear intertwining operator 
		$$J_{w_0^{-1}\P_{w_0(\theta)}w_0\mid\P_\theta}(\rho)\colon i_{\M,\P_\theta}^{\G}(\rho\otimes \chi_{\M,\R})\to i_{\M,w_0^{-1}\P_{w_0(\theta)}w_0}^{\G}(\rho\otimes \chi_{\M,\R})\otimes_{\R[\M/\M^\circ]}\mathcal S^{-1}_{\M,\R}\R[\M/\M^\circ]$$ 
		by setting
		$$J_{w_0^{-1}\P_{w_0(\theta)}w_0\mid\P_\theta}(\rho)(f)(g)=\frac{1}{b}\int_{w_0^{-1}\N_{w_0(\theta)}w_0\cap\overline{\N}_\theta}(\beta (\overline{g\cdot f}))(n)dn$$
		for all $f\in i_{\M,\P_\theta }^{\G}(\rho\otimes \chi_{\M,\R})$ and $g\in\G$. Here, we lift $\beta (\overline{g\cdot f})$ to a function in $\mathcal I$ for which the above integral is well-defined and independent of the chosen lift. Note that $b\cdot J_{w_0^{-1}\P_{w_0(\theta)}w_0\mid\P_\theta}(\rho)$ has image in $i_{\M,w_0^{-1}\P_{w_0(\theta)}w_0}^{\G}(\rho\otimes \chi_{\M,\R})$.


		(1):
		If $f_0\in\mathcal I$ then $\beta$ acts on $\overline{f}_0$ via multiplication with $b$ and since $w_0^{-1}\N_{w_0(\theta)}w_0=\overline{\N}_\theta$ we obtain that 
		$$J_{w_0^{-1}\P_{w_0(\theta)}w_0\mid\P_\theta}(\rho)(f_0)(1)=\int_{w_0^{-1}\N_{w_0(\theta)}w_0}f_0(n)dn.$$

		We have a $\R[\M/\M^\circ][\G]$-linear isomorphism
		\begin{align*}
			i_{\M,w_0^{-1}\P_{w_0(\theta)}w_0}^\G(\rho\otimes \chi_{\M,\R})&\to i_{w_0\M w_0^{-1},\P_{w_0(\theta)}}^\G((\rho\otimes \chi_{\M,\R})^{\widetilde{w}^{-1}_0})\\
			f&\mapsto (g\mapsto f(\widetilde{w}_0^{-1}g)).
		\end{align*}
		Hence by composition with $J_{w_0^{-1}\P_{w_0(\theta)}w_0\mid\P_\theta}(\rho)$ we obtain an intertwining operator 
		\[\mathbb{M}(\rho,\nu_{\N_{w_0(\theta)}})\colon i_{\M,\P_\theta}^{\G}(\rho\otimes \chi_{\M,\R})\rightarrow i_{w_0\M w_0^{-1},\P_{w_0(\theta)}}^{\G}((\rho\otimes \chi_{\M,\R})^{\widetilde{w}_0^{-1}})\otimes_{\R[\M/\M^{\circ}]} \mathcal{S}_{\M,\R}^{-1}\R[\M/\M^{\circ}]\]
		which for functions $f_0\in\mathcal I$ satisfies
		$$\mathbb{M}(\rho,\nu_{\N_{w_0(\theta)}})(f_0)(\widetilde{w}_0)=\int_{w_0^{-1}\N_{w_0(\theta)}w_0}f_0(n)dn.$$
		The result follows since the measure on $w_0^{-1}\N_{w_0(\theta)}w_0$ was induced from $\nu_{\N_{w_0(\theta)}}$ by conjugation with $\widetilde{w}_0^{-1}$ and hence
		$$\int_{w_0^{-1}\N_{w_0(\theta)}w_0}f_0(n)dn=\int_{\N_{w_0(\theta)}}f_0(\widetilde{w}_0^{-1}n\widetilde{w}_0)dn.$$

		(2):
		The element $b$ constructed in \cite{MT} is a product of elements of the form \begin{equation}\label{eq:intpol}
			\sum_{i=0}^kb_i(a_ww^{-1}a_w^{-1}w)^i,
		\end{equation}
		where $w\in \W_\M\backslash \W/\W_\M$ and $w\not=1$ and $a_w$ is an element in the split component of the centre of $\M$ such that $a_ww^{-1}a_w^{-1}w\notin\M^\circ$ and $b_0,b_k\in \R^\times$. 
		
		We first show that for such an $a_w$ we have that $\Delta_{X,\R}(a_ww^{-1}a_w^{-1}w)\not=1$, or equivalently that $\delta_{\P_\theta}(a_ww^{-1}a_w^{-1}w)\not=1$. Note that $\Z_\G\M^\circ$ is contained in $\ker(\delta_{\P_\theta})$ and $\ker(\delta_{\P_\theta})/\Z_\G\M^\circ$ is finite. Hence if $a_ww^{-1}a_w^{-1}w\in\ker(\delta_{\P_\theta})$ there is some positive integer $k$ such that $(a_ww^{-1}a_w^{-1}w)^k\in\Z_\G\M^\circ$. Then we can write $(a_ww^{-1}a_w^{-1}w)^k=zm$ for $z\in\Z_\G$ and $m\in\M^\circ$. Then $(a_ww^{-1}a_w^{-1}w)^km^{-1}$ lies in the kernel of $|\chi|_\F$ for all algebraic characters $\chi\colon \G\to \F^\times$ and we obtain that $z\in\G^\circ$. Hence $z\in\Z_\G\cap\G^\circ$ which is a compact group by Proposition V.2.6 of \cite{renard2010representations} and implies since $z\in\M$ that $z\in\M^\circ$. We obtain that $(a_ww^{-1}a_w^{-1}w)^k\in\M^\circ$ and hence also that $a_ww^{-1}a_w^{-1}w\in\M^\circ$, which yields a contradiction. 
		
		In particular, this implies that the image of an element as in Equation (\ref{eq:intpol}) under $\Delta_{X,\R}$ lies in $\mathcal S_{X^c,\R}$. However, when constructing $b$ the only restriction on the choice of $a_w$ for $w\not=1$ was that $a_ww^{-1}a_w^{-1}w\not\in\M^\circ$. Since $a_w$ and $w^{-1}a_w^{-1}w$ are in the fixed maximal torus of $\G$ we see that $(a_ww^{-1}a_w^{-1}w)^k=a_w^kw^{-1}a_w^{-k}w$ which is also not an element of $\M^\circ$. As $\Delta_{X,\R}(a_w^kw^{-1}a_w^{-k}w)=\Delta_{X,\R}(a_ww^{-1}a_w^{-1}w)^k$ we obtain that we may choose $a_w$ such that $\Delta_{X,\R}(a_w^kw^{-1}a_w^{-k}w)$ is an integral power of $X$, which yields the result.\end{proof}
	
	Note that if $\nu_{\N_{w_0(\theta)}}'$ is another nontrivial Haar measure on $\N_{w_0(\theta)}$, then $\nu_{\N_{w_0(\theta)}}'=c\cdot\nu_{\N_{w_0(\theta)}}$ for some $c\in \R^\times$. It is then straightforward to see that $\mathbb{M}(\rho,\nu_{\N_{w_0(\theta)}}')=c\cdot\mathbb{M}(\rho,\nu_{\N_{w_0(\theta)}})$ and we will often abbreviate $\mathbb{M}(\rho,\nu_{\N_{w_0(\theta)}})$ with $\mathbb{M}(\rho)$.

	Moreover, if $\R$ is an $\R_0$-algebra (not necessarily Noetherian), we can extend the above construction of intertwining operators to co-Whittaker $\R$-representations of $M$ of type $\mathfrak t$ with respect to a Whittaker datum $\mathfrak w_\M$ for $\M$. We set $\Gamma=\Gamma_{\mathfrak w_\M,\mathfrak t}^\M$ and $\E=\E_{\mathfrak w_\M,\mathfrak t}^\M$.  \rob{We have the~$\mathrm{S}_0$-representation $\mathcal W_{\mathfrak w_\M,\mathfrak t}^\M$ of Section \ref{cwrepssect} (where we arrange for~$\mathfrak{t}$ to be defined over~$\mathrm{S}_0$, e.g., an idempotent corresponding to a depth)} with Noetherian endomorphism ring $\mathcal E_{\mathfrak w_\M,\mathfrak t}^\M=\End_{\mathrm{S}_0[\M]}(\mathcal W_{\mathfrak w_\M,\mathfrak t}^\M)$. Hence\rob{, after extending scalars to~$\mathrm{S}_0[\sqrt{q}]$,} we can apply Theorem \ref{intertwiningops} to obtain an intertwining operator $\mathbb M(\mathcal W_{\mathfrak w_\M,\mathfrak t}^\M\otimes \mathrm{S}_0[\sqrt{q}])$. By tensoring this map with $\operatorname{id}_{\R_0}$ we obtain by Lemmas 5.14 and 5.15 of \cite{DHKMconj} a $\E[\M/\M^\circ][\G]$-morphism
	$$\mathbb M(\Gamma)\colon i_{\M,\P_\theta}^{\G}(\Gamma\otimes \chi_{\M,\E})\rightarrow i_{w_0\M w_0^{-1},\P_{w_0(\theta)}}^{\G}((\Gamma\otimes \chi_{\M,\E})^{\widetilde{w}_0^{-1}})\otimes_{\E[\M/\M^{\circ}]} \mathcal{S}_{\M,\E}^{-1}\E[\M/\M^{\circ}].$$
	Now if $\rho$ is a co-Whittaker $\R$-representation of $\M$ of type $\mathfrak t$ with respect to the Whittaker datum $\mathfrak w_\M$, we obtain an $\R_0$-algebra morphism $f_{\rho,0}\colon\E\to\R$ and a surjection $\Gamma\otimes_{f_{\rho,0}}\R\twoheadrightarrow\rho$. Now we define $\mathbb M(\rho)$ of an element in $i_{\M,\P_\theta}^{\G}(\rho\otimes\chi_{\M,\R})$ by
	first choosing a preimage in  
	$$i_{\M,\P_\theta}^{\G}(\Gamma\otimes \chi_{\M,\E})\otimes\R\cong 	i_{\M,\P_\theta}^{\G}((\Gamma\otimes\R)\otimes \chi_{\M,\R})$$
	then applying $\mathbb M(\Gamma)\otimes\operatorname{id}_\R$ and then applying the induced surjection again to obtain an element in
	\[i_{w_0\M w_0^{-1},\P_{w_0(\theta)}}^{\G}((\rho\otimes \chi_{\M,\R})^{\widetilde{w}_0^{-1}})\otimes\mathcal S_{\M,\R}^{-1}\R[\M/\M^\circ],\] i.e.\ the diagram 
		$$\begin{tikzcd}
		i_{\M,\P_\theta}^{\G}((\Gamma\otimes\R)\otimes \chi_{\M,\R})\arrow{r}{\mathbb M(\Gamma)\otimes\operatorname{id}_\R}\arrow{d}
		&i_{w_0\M w_0^{-1},\P_{w_0(\theta)}}^{\G}(((\Gamma\otimes\R)\otimes \chi_{\M,\R})^{\widetilde{w}_0^{-1}})\otimes\mathcal S_{\M,\R}^{-1}\R[\M/\M^\circ]\arrow{d}\\
		i_{\M,\P_\theta}^{\G}(\rho\otimes \chi_{\M,\R})\arrow{r}{\mathbb M(\rho)}
		&i_{w_0\M w_0^{-1},\P_{w_0(\theta)}}^{\G}((\rho\otimes \chi_{\M,\R})^{\widetilde{w}_0^{-1}})\otimes\mathcal S_{\M,\R}^{-1}\R[\M/\M^\circ],
	\end{tikzcd}$$
	commutes. It is straightforward to see from the properties of Theorem \ref{intertwiningops} that this is well-defined. It also follows readily that this definition retains the properties of Theorem \ref{intertwiningops} and of Proposition 3.13 of \cite{MT}.

	\subsection{(The inverse of) Shahidi's local coefficients in families}\label{localcoeffsect}
	For this section, we fix as before a Borel subgroup $\B=\T\U$ and a standard maximal parabolic subgroup~$\P=\M\N$ of~$\G$, i.e.\ ~$\P=\P_{\theta}$. Moreover, let $\R$ be an $\R_0$-algebra.
	
	\subsubsection{Choices}\label{secchoices}
	Having fixed our~$\F$-pinning~$\Xi$, and our character~$\psi\colon\F\rightarrow \R_0^\times$, following Shahidi  (see also \cite[A.5]{AGIKMS}) these data define:
	\begin{enumerate}
		\item the Whittaker datum~$\mathfrak w=(\U,\Psi)$ for $\G$ determined by~$(\Xi,\psi)$, as in Section \ref{whittdata};
		\item for $w_0$, we choose the Langlands--Shelstad Weyl group representative $\widetilde{w}_0\in\N_{\G}(\mathrm{S})$ determined by~$\Xi$ (see also \cite[Section 1.7]{AGIKMS}); 
		\item the~$\R_0$-Haar measure~$\nu_{\N_\theta}$ on $\N_\theta$ determined by~$(\Xi,\psi)$.
	\end{enumerate}

	\subsubsection{Definition of the inverse of Shahidi's local coefficients in families}

	At this point we can already interpolate the inverse of Shahidi's local coefficients in families:
	\begin{proposition}\label{Inverseshahidis}
		Let $\rho$ be a co-Whittaker type~$\R$-representation of~$\M$ with respect to $\mathfrak w_\M$. Then there exists a unique element
		\[\mathcal{D}_\R(\rho,\P_\theta,\Xi,\psi)\in \mathcal{S}_{\M,\R}^{-1}\R[\M/\M^\circ],\] 
		satisfying the functional equation
		\[\mathcal{D}_\R(\rho,\P_\theta,\Xi,\psi)\cdot\Omega_{\mathfrak w}(\rho,\vartheta,\nu_{\N_{w_0(\theta)}})(f)= \Omega_{\mathfrak w}({}^{\widetilde{w}_0}\rho,{}^{\widetilde{w}_0}\vartheta,\nu_{\N_\theta})( \mathbb{M}(\rho,\nu_{\N_{w_0(\theta)}})(f))\]
		for all $f\in i_{\M,\P_\theta}^\G(\rho\otimes\chi_{\M,\R})$. Moreover, there is an element $\mathcal Q\in\mathcal S_{\M,\R}$ such that $\mathcal Q\cdot \mathcal{D}_\R(\rho,\P_\theta,\Xi,\psi)\in \R[\M/\M^\circ]$ and the image of $\mathcal Q$ under $\Delta_{X,\R}$ lies in $\mathcal S_{X,\R}$.
		
	\end{proposition}

	\begin{proof}
		From Corollary \ref{generates}, we obtain that $\Omega_{\mathfrak w}(\rho,\vartheta,\nu_{\N_{w_0(\theta)}})$ generates 
		\[\Hom_{\R[\M/\M^{\circ}][\U]}(i_{\M,\P_\theta}^{\G}(\rho\otimes \chi_{\M,\R}),\Psi_{\R[\M/\M^\circ]}).\] 
		Moreover, by using the construction of intertwining operators in Theorem \ref{intertwiningops}, there is an element $\mathcal Q\in \mathcal{S}_{\M,\R}$ such that 
		\[\mathcal Q\cdot \mathbb{M}(\rho,\nu_{\N_{w_0(\theta)}})\colon i_{\M,\P_\theta}^{\G}(\rho\otimes \chi_{\M,\R})\rightarrow i_{w_0\M w_0^{-1},\P_{w_0(\theta)}}^{\G}({}^{\widetilde{w}_0}(\rho\otimes \chi_{\M,\R})).\]
		Clearly, ${}^{\widetilde{w}_0}(\rho\otimes \chi_{\M,\R})\simeq{}^{\widetilde{w}_0}\rho\otimes{}^{\widetilde{w}_0}(\chi_{\M,\R})$ and the conjugation map $\R[\M/\M^\circ]\to \R[w_0\M w_0^{-1}/w_0\M^\circ w_0^{-1}], m\mapsto w_0mw_0^{-1}$ yields an isomorphism between ${}^{\widetilde{w}_0}(\chi_{\M,\R})$ and $\chi_{w_0\M w_0^{-1},\R}$ (note here that $(w_0\M w_0^{-1})^\circ =w_0\M^\circ w_0^{-1}$).
		We write ${}^{\widetilde{w}_0}\vartheta$ for the surjective map in the free rank one $\R$-module $\Hom_{\R[w_0(\M\cap \U)w_0^{-1}]}({}^{\widetilde{w}_0}\rho,{}^{\widetilde{w}_0}\Psi_{\M})$ that is induced by $\vartheta$. Note that $w_0(\M\cap \U)w_0^{-1}=w_0\M w_0^{-1}\cap \U$ and since by assumption $\Psi_{\M}(u)=\Psi(w_0uw_0^{-1})$ for $u\in\M\cap\U$ we see that 
		${}^{\widetilde{w}_0}\vartheta$ is a surjective morphism in
		$$\Hom_{\R[w_0\M w_0^{-1}\cap \U]}({}^{\widetilde{w}_0}\rho,\Psi_\R).$$
		We may apply Corollary \ref{generates} and obtain a nontrivial element $\Omega_{\mathfrak w}({}^{\widetilde{w}_0}\rho,{}^{\widetilde{w}_0}\vartheta,\nu_{\N_\theta})$
		of 
		$$\Hom_{\R[w_0\M w_0^{-1}/w_0\M^\circ w_0^{-1}][\U]}\left(i_{w_0\M w_0^{-1},\P_{w_0(\theta)}}^\G({}^{\widetilde{w}_0}\rho\otimes\chi_{w_0\M w_0^{-1},\R}),\Psi_{\R[w_0\M w_0^{-1}/w_0\M^\circ w_0^{-1}]}\right).$$ Note here that if $w_\ell^{w_0(\theta)}$ is the longest Weyl group element in $\W_{w_0(\theta)}$ then $w_0^{-1}=w_\ell w_{\ell}^{w_0(\theta)}$.
		
		However, note that conjugation by $\widetilde{w}_0^{-1}$ yields an isomorphism $\Psi_{\R[w_0\M w_0^{-1}/w_0\M^\circ w_0^{-1}]}\simeq\Psi_{\R[\M/\M^\circ]}$ of $\R[\U]$-modules and it is straightforward to see that composing $\Omega_{\mathfrak w}({}^{\widetilde{w}_0}\rho,{}^{\widetilde{w}_0}\vartheta,\nu_{\N_{\theta}})$ and $\mathcal Q\cdot \mathbb{M}(\rho,\nu_{\N_{w_0(\theta)}})$ (and using the isomorphism $\R[w_0\M w_0^{-1}/w_0\M^\circ w_0^{-1}]\simeq\R[\M/\M^\circ]$ induced by conjugation with $\widetilde{w}_0^{-1}$) yields an element of
		\[\Hom_{\R[\M/\M^{\circ}][\U]}(i_{\M,\P_\theta}^{\G}(\rho\otimes \chi_{\M,\R}),\Psi_{\R[\M/\M^\circ]}).\]

		Since $\Omega_{\mathfrak w}(\rho,\vartheta,\nu_{\N_{w_0(\theta)}})$ is a generator of this space there is a unique~$\mathcal C\in \R[\M/\M^\circ]$, which satisfies
		\[\mathcal C\cdot \Omega_{\mathfrak w}(\rho,\vartheta,\nu_{\N_{w_0(\theta)}})(f)=\Omega_{\mathfrak w}({}^{\widetilde{w}_0}\rho,{}^{\widetilde{w}_0}\vartheta,\nu_{\N_\theta})(\mathcal Q\cdot \mathbb{M}(\rho,\nu_{\N_{w_0(\theta)}})(f)).\]
		By multiplying this equation with $\mathcal Q^{-1}$ we obtain 
		\begin{equation}\label{eq:funceq}
			\mathcal C\cdot\mathcal Q^{-1}\cdot \Omega_{\mathfrak w}(\rho,\vartheta,\nu_{\N_{w_0(\theta)}})(f)=\Omega_{\mathfrak w}({}^{\widetilde{w}_0}\rho,{}^{\widetilde{w}_0}\vartheta,\nu_{\N_\theta})( \mathbb{M}(\rho,\nu_{\N_{w_0(\theta)}})(f))
		\end{equation}
		in $\mathcal{S}_{\M,\R}^{-1}\R[\M/\M^\circ]$. We set $\mathcal D_\R(\rho,\P_\theta,\Xi,\psi)=\mathcal C\cdot\mathcal Q^{-1}\in\mathcal S_{\M,\R}^{-1}\R[\M/\M^\circ]$. 
		
		Note that if $\vartheta'$ is another generator of $\Hom_{\R[\M\cap \U]}(\rho,\Psi_{\M,\R})$, then by assumption $\vartheta'=c\vartheta$ for some unit $c\in \R^\times$ which also implies that ${}^{\widetilde{w}_0}\vartheta'=c({}^{\widetilde{w}_0}\vartheta)$. Since $\Omega_{\mathfrak w}(\rho,c\vartheta,\nu_{\N_{w_0(\theta)}})=c\Omega_{\mathfrak w}(\rho,\vartheta,\nu_{\N_{w_0(\theta)}})$ it follows by Equation \eqref{eq:funceq} that $\mathcal D_\R(\rho,\P_\theta,\Xi,\psi)$ does not depend on the choice of a generator $\vartheta\in\Hom_{\R[\M\cap \U]}(\rho,\Psi_{\M,\R})$. Similarly, if $\nu_{\N_{w_0(\theta)}}'$ is another Haar measure on $\N_{w_0(\theta)}$ then $\nu_{\N_{w_0(\theta)}}'=d\nu_{\N_{w_0(\theta)}}$ for some $d\in \R^\times$ and we have that $\Omega_{\mathfrak w}(\rho,\vartheta,\nu_{\N_{w_0(\theta)}}')=d\Omega_{\mathfrak w}(\rho,\vartheta,\nu_{\N_{w_0(\theta)}})$ and $\mathbb{M}(\rho,\nu_{\N_{w_0(\theta)}}')=d\mathbb{M}(\rho,\nu_{\N_{w_0(\theta)}})$. By Equation \eqref{eq:funceq} we obtain that $\mathcal D_\R(\rho,\P_\theta,\Xi,\psi)$ does not depend on the chosen Haar measure $\nu_{\N_{w_0(\theta)}}$ on $\N_{w_0(\theta)}$. Theorem \ref{intertwiningops} (2) implies that the image of $\mathcal Q$ under $\Delta_{X,\R}$ lies in $\mathcal S_{X,\R}$ and we obtain the result.
	\end{proof}
		\begin{remark}
		If one assumes that $\R$ is a Noetherian $\R_0$-algebra then the above Proposition \ref{Inverseshahidis} also holds for finitely generated Whittaker type $\R$-representations of~$\M$ satisfying Schur's lemma, for which the same proof applies. Unfortunately, we can not extend the result in this setting to more general rings $\R$, since for Whittaker type~$\R$-representations Theorem \ref{intertwiningops} only yields an intertwining operator for Noetherian rings $\R$ and we cannot extend the construction as in the co-Whittaker case. 
	\end{remark}
	
	\subsubsection{Properties of the inverse of Shahidi's local coefficients in families}
	\begin{lemma}[Compatibility with homomorphisms]\label{lem:comphom}
		Let $\rho$ be a co-Whittaker $\R$-representation of $\M$ with respect to $\mathfrak w_\M$. Suppose that~$\alpha\colon \R\rightarrow \R'$ is a morphism of~$\R_0$-algebras, and~$\alpha_\M\colon\mathcal{S}_{\M,\R}^{-1}\R[\M/\M^\circ]\rightarrow \mathcal{S}_{\M,\R'}^{-1}\R'[\M/\M^\circ]$ is the unique~$\R_0$-algebra morphism extending $\alpha$ which is the identity on elements of $\M/\M^\circ$. Then we have that 
		\[\mathcal{D}_{\R'}(\rho\otimes_\alpha \R',\P_\theta,\Xi,\psi)=\alpha_\M(\mathcal{D}_\R(\rho,\P_\theta,\Xi,\psi)).\]
	\end{lemma}
	\begin{proof} Let $\nu_{\N_{w_0(\theta)}}$ be a nontrivial Haar-measure on $\N_{w_0(\theta)}$ which we use to define Whittaker functionals. However, for ease of notation we drop the dependencies on this Haar measure (and on $\nu_{\N_\theta})$ in $\Omega_{\mathfrak w}$ in this proof.

		Let $\vartheta$ be a generator of $\Hom_{\R[\U_\M]}(\rho,\Psi_{\M,\R})$. Lemma \ref{cwlemma} (4) implies that $\rho\otimes_\alpha \R'$ is a co-Whittaker $\R'$-representation of $\M$ with respect to $\mathfrak w_\M$ of type $\mathfrak t$. Note that $\vartheta\otimes_\alpha\operatorname{id}_{\R'}$ is a generator of $\Hom_{\R'[\U_\M]}(\rho\otimes_\alpha\R',\Psi_{\M,\R'})$. By Proposition \ref{Inverseshahidis} there is a unique element $\mathcal{D}_{\R'}(\rho\otimes_\alpha \R',\P_\theta,\Xi,\psi)\in\mathcal S_{\M,\R'}^{-1}\R'[\M/\M^\circ]$ such that 
		$$\mathcal{D}_{\R'}(\rho\otimes_\alpha \R',\P_\theta,\Xi,\psi)\cdot\Omega_{\mathfrak w}(\rho\otimes_\alpha \R',\vartheta\otimes_\alpha\operatorname{id}_{\R'})(f')=\Omega_{\mathfrak w}({}^{\widetilde{w}_0}\rho\otimes_\alpha \R',{}^{\widetilde{w}_0}\vartheta\otimes_\alpha\operatorname{id}_{\R'})(\mathbb M(\rho\otimes_\alpha \R')(f'))$$
		for all $f'\in i_{\M,\P_\theta}^\G((\rho\otimes_\alpha \R')\otimes_{\R'}\chi_{\M,\R'})$.

We have a natural isomorphism 
		\begin{equation}\label{eq:parindbase}
			i_{\M,\P_\theta}^\G(\rho\otimes_\R\chi_{\M,\R})\otimes_{\alpha}\R'\simeq i_{\M,\P_\theta}^\G((\rho\otimes_\alpha \R')\otimes_{\R'}\chi_{\M,\R'}).
		\end{equation} 
		By Equation \eqref{eq:parindbase} both $\Omega_{\mathfrak w}(\rho,\vartheta)\otimes\operatorname{id}_{\R'}$ and $\Omega_{\mathfrak w}(\rho\otimes_\alpha \R',\vartheta\otimes_\alpha\operatorname{id}_{\R'})$ are elements of 
		\[\Hom_{\R'[\M/\M^{\circ}][\U]}(i_{\M,\P_\theta}^\G(\rho\otimes_\R\chi_{\M,\R})\otimes_{\alpha}\R',\Psi_{\R'[\M/\M^{\circ}]})\cong\R'[\M/\M^\circ],\]
		and hence are linearly related. Now a straightforward calculation using Lemma \ref{lemmaconstfunction} shows that these elements agree on $f\otimes x$ where $f\in i_{\M,\P_\theta}^{\G}(\rho\otimes \chi_{\M,\R})$  is supported in $\P_\theta w_0^{-1}\B$ and $x\in\R'$. Hence we obtain that 
		$$\Omega_{\mathfrak w}(\rho,\vartheta)\otimes\operatorname{id}_{\R'}=\Omega_{\mathfrak w}(\rho\otimes_\alpha \R',\vartheta\otimes_\alpha\operatorname{id}_{\R'})$$
		and analogously
		$$\Omega_{\mathfrak w}({}^{\widetilde{w}_0}\rho,{}^{\widetilde{w}_0}\vartheta)\otimes\operatorname{id}_{\R'}=\Omega_{\mathfrak w}({}^{\widetilde{w}_0}\rho\otimes_{\alpha}\R',{}^{\widetilde{w}_0}\vartheta\otimes\operatorname{id}_{\R'}).$$
		Proposition 3.13 of \cite{MT} implies for $f\in i_{\M,\P_\theta}^\G(\rho\otimes_\R\chi_{\M,\R})$ and $x\in \R'$ that
		$$\mathbb M(\rho\otimes_\alpha \R')(f\otimes x)=\mathbb M(\rho)(f)\otimes x$$
		and we can compute
		\begin{align*}
		\Omega_{\mathfrak w}({}^{\widetilde{w}_0}\rho\otimes_{\alpha}\R',{}^{\widetilde{w}_0}\vartheta\otimes\operatorname{id}_{\R'})(\mathbb M(\rho\otimes_\alpha \R')(f\otimes x))&=	\Omega_{\mathfrak w}({}^{\widetilde{w}_0}\rho,{}^{\widetilde{w}_0}\vartheta)(\mathbb M(\rho)(f))\otimes x\\	&=(\mathcal{D}_\R(\rho,\P_\theta,\Xi,\psi)\cdot\Omega_{\mathfrak w}(\rho,\vartheta)(f))\otimes x\\
			&=\alpha_\M(\mathcal{D}_\R(\rho,\P_\theta,\Xi,\psi))\cdot(\Omega_{\mathfrak w}(\rho,\vartheta)(f)\otimes x)\\
			&=\alpha_\M(\mathcal{D}_\R(\rho,\P_\theta,\Xi,\psi))\cdot\Omega_{\mathfrak w}(\rho\otimes_\alpha \R',\vartheta\otimes_\alpha\operatorname{id}_{\R'})(f\otimes x).
		\end{align*}
		Equation \eqref{eq:parindbase} then yields that $$\alpha_\M(\mathcal{D}_\R(\rho,\P_\theta,\Xi,\psi))\cdot\Omega_{\mathfrak w}(\rho\otimes_\alpha \R',\vartheta\otimes_\alpha\operatorname{id}_{\R'})(f')=
		\Omega_{\mathfrak w}({}^{\widetilde{w}_0}\rho\otimes_{\alpha}\R',{}^{\widetilde{w}_0}\vartheta\otimes\operatorname{id}_{\R'})(\mathbb M(\rho\otimes_\alpha \R')f')$$
		for all $f'\in i_{\M,\P_\theta}^\G((\rho\otimes_\alpha \R')\otimes_{\R'}\chi_{\M,\R'})$, which implies by the uniqueness in Proposition \ref{Inverseshahidis} that
		\[\mathcal{D}_{\R'}(\rho\otimes_\alpha \R',\P_\theta,\Xi,\psi)=\alpha_\M(\mathcal{D}_\R(\rho,\P_\theta,\Xi,\psi)).\]\end{proof}
	
	\begin{lemma}\label{coefficientscowhittC}
		Let $\rho$ and $\rho'$ be co-Whittaker $\R$-representations of $\M$ with respect to $\mathfrak w_\M$ and of type $\mathfrak t$ and suppose that $\rho$ dominates $\rho'$. Then 
		$$\mathcal D_\R(\rho,\P_\theta,\Xi,\psi)=\mathcal D_\R(\rho',\P_\theta,\Xi,\psi).$$

	\end{lemma}

	\begin{proof}
		Since $\rho$ dominates $\rho'$ we have a surjective $\R[\M]$-morphism $\Upsilon\colon\rho\to\rho'$. By exactness of $(\U_{\M},\Psi_{\M})$-coinvariants, $\Upsilon$ induces a surjective $\R[\U_\M]$-linear morphism $\overline{\Upsilon}\colon\rho_{\U_\M,\Psi_{\M,\R}}\to\rho'_{\U_\M,\Psi_{\M,\R}}$. Since by assumption $\rho$ and $\rho'$ are co-Whittaker, we obtain by Proposition \ref{propcoinv} that $\rho_{\U_\M,\Psi_{\M,\R}}$ and $\rho'_{\U_\M,\Psi_{\M,\R}}$ are free $\R$-modules of rank one, which implies that $\overline{\Upsilon}$ is an isomorphism. Hence fixing an $\R[\U_\M]$-linear isomorphism $\alpha_\rho$ between $\rho_{\U_\M,\Psi_{\M,\R}}$ and $\Psi_{\M,\R}$, induces an $\R[\U_\M]$-linear isomorphism $\alpha_{\rho'}$ between $\rho'_{\U_\M,\Psi_{\M},\R}$ and $\Psi_{\M,\R}$. We have the following commutative diagram of $\R[\U_\M]$-modules

		$$\begin{tikzcd}
			\rho \arrow{rr}{\Upsilon} \arrow{d}[swap]{\operatorname{pr}_\rho}  &   &\rho' \arrow{d}{\operatorname{pr}_{\rho'}}  \\
			\rho_{\U_\M,\Psi_{\M,\R}} \arrow{rr}{\overline{\Upsilon}} \arrow{rd}[swap]{\alpha_\rho} &   & \rho'_{\U_\M,\Psi_{\M,\R}} \arrow{ld}{\alpha_{\rho'}} \\
			& \Psi_{\M,\R}. &             
		\end{tikzcd}$$
		Hence the map $\vartheta=\alpha_\rho\circ\operatorname{pr}_\rho\in\Hom_{\R[\U_\M]}(\rho,\Psi_{\M,\R})$ factors through $\rho'$ and $\vartheta=\omega\circ\Upsilon$ for $\omega=\alpha_{\rho'}\circ\operatorname{pr}_{\rho'}\in\Hom_{\R[\U_\M]}(\rho',\Psi_{\M,\R})$.

		Proposition 3.13 of \cite{MT} implies the commutative diagram 
		$$\begin{tikzcd}
			i_{\M,\P_\theta}^{\G}(\rho\otimes \chi_{\M,\R})\arrow{r}{\mathbb M(\rho)}\arrow{d}[swap]{\Upsilon_1}
			&i_{w_0\M w_0^{-1},\P_{w_0(\theta)}}^{\G}((\rho\otimes \chi_{\M,\R})^{\widetilde{w}_0^{-1}})\otimes\mathcal S_{\M,\R}^{-1}\R[\M/\M^\circ]\arrow{d}{\Upsilon_2\otimes\operatorname{id}}\\
			i_{\M,\P_\theta}^{\G}(\rho'\otimes \chi_{\M,\R})\arrow{r}{\mathbb M(\rho')}
			&i_{w_0\M w_0^{-1},\P_{w_0(\theta)}}^{\G}((\rho'\otimes \chi_{\M,\R})^{\widetilde{w}_0^{-1}})\otimes\mathcal S_{\M,\R}^{-1}\R[\M/\M^\circ],
		\end{tikzcd}$$
		where we write $\Upsilon_1$ and $\Upsilon_2$ for the obvious morphisms induced by $\Upsilon$.

		Note that both $\Omega_{\mathfrak w}(\rho,\vartheta,\nu_{\N_{w_0(\theta)}})$ and $\Omega_{\mathfrak w}(\rho',\omega,\nu_{\N_{w_0(\theta)}})\circ\Upsilon_1$ are elements of 
		\[\Hom_{\R[\M/\M^{\circ}][\U]}(i_{\M,\P_\theta}^{\G}(\rho\otimes\chi_{\M,\R}),\Psi_{\R[\M/\M^{\circ}]})\cong\R[\M/\M^\circ],\]
		and hence are linearly related.
		However, for any $f\in i_{\M,\P_\theta}^{\G}(\rho\otimes \chi_{\M,\R})$ that is supported in $\P_\theta w_0^{-1}\B$ we have that
		\begin{align*}
			\Omega_{\mathfrak w}(\rho,\vartheta,\nu_{\N_{w_0(\theta)}})(f)&=\int_{\N_{w_0(\theta)}} \vartheta(f(\widetilde{w}_{0}^{-1}n))\Psi_{\R}^{-1}(n)dn\\
			&=\int_{\N_{w_0(\theta)}} (\omega\circ\Upsilon)(f(\widetilde{w}_{0}^{-1}n))\Psi_{\R}^{-1}(n)dn\\
			&=\Omega_{\mathfrak w}(\rho',\omega,\nu_{\N_{w_0(\theta)}})(\Upsilon_1f),
		\end{align*}
		where we use that $\Upsilon(f(g))=(\Upsilon_1f)(g)$ for all $g\in\G$. This implies that 
		$$\Omega_{\mathfrak w}(\rho,\vartheta,\nu_{\N_{w_0(\theta)}})=\Omega_{\mathfrak w}(\rho',\omega,\nu_{\N_{w_0(\theta)}})\circ\Upsilon_1.$$
		Analogously, we obtain that
		$$\Omega_{\mathfrak w}({}^{\widetilde{w}_0}\rho,{}^{\widetilde{w}_0}\vartheta,\nu_{\N_{\theta}})=\Omega_{\mathfrak w}({}^{\widetilde{w}_0}\rho',{}^{\widetilde{w}_0}\omega,\nu_{\N_{\theta}})\circ\Upsilon_2.$$
	
		Hence we can compute that
		\begin{align*}
			\mathcal{D}_{\R}(\rho,\P_\theta,\Xi,\psi)\cdot\Omega_{\mathfrak w}(\rho',\omega,\nu_{\N_{w_0(\theta)}})(\Upsilon_1f)&=\mathcal{D}_{\R}(\rho,\P_\theta,\Xi,\psi)\cdot\Omega_{\mathfrak w}(\rho,\vartheta,\nu_{\N_{w_0(\theta)}})(f)\\
			&=\Omega_{\mathfrak w}({}^{\widetilde{w}_0}\rho,{}^{\widetilde{w}_0}\vartheta,\nu_{\N_{\theta}})(\mathbb M(\rho,\nu_{\N_{w_0(\theta)}})f)\\
			&=\Omega_{\mathfrak w}({}^{\widetilde{w}_0}\rho',{}^{\widetilde{w}_0}\omega,\nu_{\N_{\theta}})((\Upsilon_2\otimes\operatorname{id})\circ\mathbb M(\rho,\nu_{\N_{w_0(\theta)}})(f))\\
			&=\Omega_{\mathfrak w}({}^{\widetilde{w}_0}\rho',{}^{\widetilde{w}_0}\omega,\nu_{\N_{\theta}})(\mathbb M(\rho',\nu_{\N_{w_0(\theta)}})(\Upsilon_1f))
		\end{align*}
		for all $f\in i_{\M,\P_\theta}^{\G}(\rho\otimes \chi_{\M,\R})$. Note that $\Upsilon_1$ is surjective and since $\mathcal{D}_{\R}(\rho',\P_\theta,\Xi,\psi)$ is the unique element in $\mathcal{S}_{\M,\R}^{-1}\R[\M/\M^\circ]$ such that 
		$$\mathcal{D}_{\R}(\rho',\P_\theta,\Xi,\psi)\cdot\Omega_{\mathfrak w}(\rho',\omega,\nu_{\N_{w_0(\theta)}})(f')=\Omega_{\mathfrak w}({}^{\widetilde{w}_0}\rho',{}^{\widetilde{w}_0}\omega,\nu_{\N_{\theta}})(\mathbb M(\rho',\nu_{\N_{w_0(\theta)}})(f'))$$
		for all $f'\in i_{\M,\P_\theta}^{\G}(\rho'\otimes \chi_{\M,\R})$, we obtain that
		$$\mathcal{D}_{\R}(\rho,\P_\theta,\Xi,\psi)=\mathcal{D}_{\R}(\rho',\P_\theta,\Xi,\psi).$$
	\end{proof}
	
	If $\rho$ and $\rho'$ are two isomorphic co-Whittaker modules, then clearly $\rho$ dominates $\rho'$ and we immediately obtain by the above that $\mathcal D_\R(\rho,\P,\Xi,\psi)$ depends only on the isomorphism class of $\rho$.

	Let $\rho$ be a co-Whittaker $\R$ representation of $\M$ with respect to $\mathfrak w_\M$ of type $\mathfrak t$. By Proposition \ref{Inverseshahidis} there is an element $\mathcal Q\in\mathcal S_{\M,\R}$ such that $\mathcal Q\cdot\mathcal D_\R(\rho,\P_\theta,\Xi,\psi)\in \R[\M/\M^\circ]$ and the image of $\mathcal Q$ under $\Delta_{X,\R}$ lies in $\mathcal S_{X,\R}$. We then define
	$$\mathcal D_\R(X,\rho,\P_\theta,\Xi,\psi)\coloneqq\Delta_{X,\R}(\mathcal Q)^{-1}\Delta_{X,\R}(\mathcal Q\cdot \mathcal D_\R(\rho,\P_\theta,\Xi,\psi))\in \mathcal S_{X,\R}^{-1}\R[X^{\pm c}],$$
	and it is straightforward to see that $\mathcal D_\R(X,\rho,\P_\theta,\Xi,\psi)$ is independent of the chosen $\mathcal Q$. The following result is the analogue of Lemma \ref{lem:comphom}, i.e.\ $\mathcal D_\R(X,\rho,\P_\theta,\Xi,\psi)$ behaves well under base-change. 
	
	\begin{proposition}\label{Xcoeffbasechange}Let $\rho$ be a co-Whittaker $\R$-representation of $\M$ with respect to $\mathfrak w_\M$ of type $\mathfrak t$ and suppose that $\alpha\colon \R\to \R'$ is an $\R_0$-algebra homomorphism. Let $\alpha_X\colon\mathcal S_{X,\R}^{-1}\R[X^{\pm c}]\to\mathcal S_{X,\R'}^{-1}\R'[X^{\pm c}]$ be the unique extension of $\alpha$ that is the identity on $X^{c}$. Then
		$$\mathcal D_{\R'}(X,\rho\otimes_\alpha \R',\P_\theta,\Xi,\psi)=\alpha_X(\mathcal D_\R(X,\rho,\P_\theta,\Xi,\psi)).$$
		Moreover, if $L_{\rho}\in \R[[X^c]][X^{- c}]$ (respectively $L_{\rho\otimes_{\alpha} \R'}\in \R'[[X^c]][X^{-c}]$) are the associated Laurent series (cf.\ Proposition \ref{powerseriesexpansion}) and we denote by $\overline{\alpha}_X\colon \R[[X^c]][X^{-c}]\to \R'[[X^c]][X^{-c}]$ the induced morphism that is the identity on $X^c$, we have that 
		$$L_{\rho\otimes_{\alpha} \R'}=\overline{\alpha}_X(L_{\rho}).$$
	\end{proposition}

	\begin{proof}
		By Lemma \ref{lem:comphom} we have that~$\mathcal D_{\R'}(\rho\otimes_{\alpha}\R',\P_\theta,\Xi,\psi)=\alpha_\M(\mathcal D_{\R}(\rho,\P_\theta,\Xi,\psi))$. Let $\mathcal Q\in\mathcal S_{\M,\R}$ be as in Proposition \ref{Inverseshahidis} and hence 
		$$\alpha_\M(\mathcal Q)\cdot\mathcal D_{\R'}(\rho\otimes_{\alpha}\R',\P_\theta,\Xi,\psi)=\alpha_\M(\mathcal Q\cdot\mathcal D_{\R}(\rho,\P_\theta,\Xi,\psi))\in \R'[\M/\M^\circ],$$
		and we see that
		$$\mathcal D_{\R'}(X,\rho\otimes_\alpha \R',\P_\theta,\Xi,\psi)=\Delta_{X,\R'}(\alpha_\M(\mathcal Q))^{-1}\cdot\Delta_{X,\R'}(\alpha_\M(\mathcal Q)\cdot\mathcal D_{\R'}(\rho\otimes_{\alpha}\R',\P_\theta,\Xi,\psi)).$$
		We clearly have that $\Delta_{X,\R'}\circ\alpha_\M=\alpha_X\circ\Delta_{X,\R}$ on $\R[\M/\M^\circ]$, and hence can compute
		\begin{align*}
			\mathcal D_{\R'}(X,\rho\otimes_\alpha \R',\P_\theta,\Xi,\psi)&=\Delta_{X,\R'}(\alpha_\M(\mathcal Q))^{-1}\cdot\Delta_{X,\R'}(\alpha_\M(\mathcal Q)\cdot\mathcal D_{\R'}(\rho\otimes_{\alpha}\R',\P_\theta,\Xi,\psi))\\
			&=\Delta_{X,\R'}(\alpha_\M(\mathcal Q))^{-1}\cdot\Delta_{X,\R'}(\alpha_\M(\mathcal Q\cdot\mathcal D_{\R}(\rho,\P_\theta,\Xi,\psi)))\\
			&=\alpha_X(\Delta_{X,\R}(\mathcal Q))^{-1}\cdot\alpha_X(\Delta_{X,\R}(\mathcal Q\cdot\mathcal D_{\R}(\rho,\P_\theta,\Xi,\psi)))\\
			&=\alpha_X(\mathcal D_\R(X,\rho,\P_\theta,\Xi,\psi)).
		\end{align*}
		
		Suppose that $\mathcal D_{\R}(X,\rho,\P_\theta,\Xi,\psi)=B^{-1}\cdot A$, where $B\in\mathcal S_{X,\R}$ and $A\in \R[X^{\pm c}]$. We then have that $B\cdot L_{\rho}=A$, which implies that $\overline{\alpha}_X(B)\cdot\overline{\alpha}_X(L_{\rho})=\overline{\alpha}_X(A)$. By the first part of this proposition we obtain that $\mathcal D_{\R'}(X,\rho\otimes_\alpha \R',\P_\theta,\Xi,\psi)=\alpha_X(B)^{-1}\cdot\alpha_X(A)$ and hence $\alpha_X(B)\cdot L_{\rho\otimes_{\alpha} \R'}=\alpha_X(A)$. However, we clearly have that $\overline{\alpha}_X(A)=\alpha_X(A)$ and $\overline{\alpha}_X(B)=\alpha_X(B)$ and by uniqueness of Proposition \ref{powerseriesexpansion} we obtain the result.
	\end{proof}

	The analogue of Lemma \ref{coefficientscowhittC} also holds and we can moreover prove that only integral powers of $X$ appear in $\mathcal D_\R(X,\rho,\P_\theta,\Xi,\psi)$, i.e.\ we have the following results.
	\begin{proposition}\label{propositiondowntointX} Let $\rho$ and $\rho'$ be co-Whittaker $\R$-representations of $\M$ with respect to $\mathfrak w_\M$ of type $\mathfrak t$. 
		\begin{enumerate}
			\item If $\rho$ dominates $\rho'$, then
			$$\mathcal D_\R(X,\rho,\P_\theta,\Xi,\psi)=\mathcal D_\R(X,\rho',\P_\theta,\Xi,\psi).$$
			
			\item We have that
			$$\mathcal D_\R(X,\rho,\P_\theta,\Xi,\psi)\in \mathcal{S}_{X,\R}^{-1}\R[X^{\pm 1}].$$
		\end{enumerate}
	\end{proposition}
	
	\begin{proof}
		For the first result let $\mathcal Q\in\mathcal S_{\M,\R}$ be as in Proposition \ref{Inverseshahidis}, in particular we have
		$\mathcal Q\cdot\mathcal D_{\R}(\rho,\P_\theta,\Xi,\psi)\in \R[\M/\M^\circ]$. By Lemma \ref{coefficientscowhittC} we obtain
		$$\mathcal Q\cdot \mathcal D_\R(\rho',\P_\theta,\Xi,\psi)=\mathcal Q\cdot\mathcal D_\R(\rho,\P_\theta,\Xi,\psi)\in \R[\M/\M^\circ].$$
		However, we then have
		\begin{align*}
					\mathcal D_\R(X,\rho',\P_\theta,\Xi,\psi)&=\Delta_{X,\R}(\mathcal Q)^{-1}\Delta_{X,\R}(\mathcal Q\cdot\mathcal D_\R(\rho',\P_\theta,\Xi,\psi))\\&=\Delta_{X,\R}(\mathcal Q)^{-1}\Delta_{X,\R}(\mathcal Q\cdot\mathcal D_\R(\rho,\P_\theta,\Xi,\psi))\\&=\mathcal D_\R(X,\rho,\P_\theta,\Xi,\psi).
		\end{align*}

		We first prove the second result in the universal case, i.e.\ $\rho=\Gamma_{\mathfrak w_\M,\mathfrak t}^\M$ and $\R=\E_{\mathfrak w_\M,\mathfrak{t}}^\M$ for some finite sum of idempotents $\mathfrak t$ in the Bernstein centre $\mathfrak Z_{\R_0}(\M)$. For ease of notation we set $\Gamma=\Gamma_{\mathfrak w_\M,\mathfrak t}^\M$ and $\E=\E_{\mathfrak w_\M,\mathfrak{t}}^\M$. By Proposition \ref{powerseriesexpansion} we can write
		$\mathcal D_\E(X,\Gamma,\P_\theta,\Xi,\psi)$ as a unique Laurent series $L_\Gamma\in\E[[X^c]][X^{-c}]$. As~$\E$ is torsion-free, all minimal primes of~$\E$ are characteristic zero, and as~$\E$ is a finitely generated~$\R_0$-algebra by Lemma \ref{Ebasicfacts}, for a minimal prime~$\mathfrak{p}$ of~$\E$, its residue field~$\kappa(\mathfrak{p})$ embeds into~$\mathbb{C}$.  For any minimal prime~$\mathfrak{p}$ of~$\E$, we choose an embedding~$\kappa(\mathfrak{p})\hookrightarrow \mathbb{C}$, and hence obtain an $\R_0$-algebra homomorphism $\iota_{\mathfrak p}\colon \E\to\mathbb C$. Now $\varsigma=\Gamma\otimes_{\iota_{\mathfrak p}}\mathbb C$ is a co-Whittaker $\mathbb C$-representation of $\M$ with respect to $\mathfrak w_\M$ and thus has an irreducible cosocle $\varsigma_{gen}$ which is generic with respect to $\mathfrak w_\M$.
		By Proposition \ref{Xcoeffbasechange} we see that the Laurent series associated to $\mathcal D_{\mathbb C}(X,\Gamma\otimes_{\iota_{\mathfrak p}}\mathbb C,\P_\theta,\Xi,\psi)$ equals $\overline{\iota_{\mathfrak p}}(L_\Gamma)$. Now by the first part of this proposition we have 	
		$$\mathcal D_{\mathbb C}(X,\Gamma\otimes_{\iota_{\mathfrak p}}\mathbb C,\P_\theta,\Xi,\psi)
		=\mathcal D_{\mathbb C}(X,\varsigma_{gen},\P_\theta,\Xi,\psi).$$
		We write $L_{\varsigma_{gen}}$ for the Laurent series associated to $\mathcal D_{\mathbb C}(X,\varsigma_{gen},\P_\theta,\Xi,\psi)$ and by uniqueness in Proposition \ref{powerseriesexpansion} we see that 
		$$\overline{\iota_{\mathfrak p}}_X(L_\Gamma)=L_{\varsigma_{gen}}.$$
		
		However, Lemmas 11.9 and 11.10 of Kim's article in \cite{cogdell2004lectures} imply that the coefficient of $X^{kc}$ in $L_{\varsigma_{gen}}$ can only be nonzero if $kc\in\mathbb Z$. Hence if $kc\notin\mathbb Z$ the coefficient of $X^{kc}$ in $\overline{\iota_{\mathfrak p}}_X(L_\Gamma)\in\mathbb C[[X^{c}]][X^{-c}]$ is zero for all minimal prime ideals $\mathfrak p$ of $\E$. However, since $\E$ is reduced this implies that the coefficient of $X^{kc}$ in $L_\Gamma$ is zero and we obtain that $$L_\Gamma\in\E[[X]][X^{-1}].$$
		By Proposition \ref{Inverseshahidis} there is a $\mathcal Q$ in $\mathcal S_{X,\E}$ such that $\mathcal Q\cdot\mathcal D_\E(X,\Gamma,\P_\theta,\Xi,\psi)\in\E[X^{\pm c}]$. Since
		$$\mathcal Q\cdot\mathcal D_\E(X,\Gamma,\P_\theta,\Xi,\psi)=\mathcal Q\cdot L_\Gamma$$ we immediately obtain that $\mathcal Q\cdot\mathcal D_\E(X,\Gamma,\P_\theta,\Xi,\psi)$ actually lies in $\E[X^{\pm 1}]$ and hence that $$\mathcal D_\E(X,\Gamma,\P_\theta,\Xi,\psi)\in\mathcal S_{X,\E}^{-1}\E[X^{\pm1}].$$
		Now for an arbitrary $\R_0$-algebra $\R$ and a co-Whittaker $\R$-representation $\rho$ of $\M$ with respect to $\mathfrak w_\M$, we have that there is a $\mathfrak t$ and an $\R_0$-algebra homomorphism $f_\rho\colon \E_{\mathfrak w_\M,\mathfrak{t}}^\M\to \R$ such that $\Gamma_{\mathfrak w_\M,\mathfrak t}^\M\otimes_{f_\rho}\R$ dominates $\rho$. However, by Proposition \ref{Xcoeffbasechange} we obtain that 
		$$\mathcal D_\R(X,\Gamma_{\mathfrak w_\M,\mathfrak t}^\M\otimes_{f_\rho}\R,\P_\theta,\Xi,\psi)=(f_\rho)_X(\mathcal D_{\E_{\mathfrak w_\M,\mathfrak{t}}^\M}(X,\Gamma_{\mathfrak w_\M,\mathfrak t}^\M,\P_\theta,\Xi,\psi))\in\mathcal S_{X,\R}^{-1}\R[X^{\pm 1}].$$
		By the first part of this proposition we have that $\mathcal D_\R(X,\Gamma_{\mathfrak w_\M,\mathfrak t}^\M\otimes_{f_\rho}\R,\P_\theta,\Xi,\psi)=\mathcal D_{\R}(X,\rho,\P_\theta,\Xi,\psi)$ and the result follows.
	\end{proof}

	\subsection{A key inductive lemma}
	Let~$\P=\M\N$ be a Levi decomposition of a standard maximal parabolic of~$\G$.  We can write~$\P=\P_{\theta}$ where~$\theta=\Delta\backslash\{a\}$ as in Section \ref{parabolics}.  Let~$\rho=\rho_{\theta}$ denote half the sum of the absolute roots of~$\T$ in the Lie algebra of~$\N_\theta$.  
	Let~\[\widetilde{a}=\langle \rho,a^{\vee}\rangle^{-1}\rho.\]
	where~$a$ is an absolute root which restricts to~$\alpha$, and~$a^{\vee}$ is the coroot of~$a$.

	Let~$\widehat{\P}(\mathbb{C})=\widehat{\M}(\mathbb{C})\widehat{\N}(\mathbb{C})$ be the (complex points of the) dual parabolic subgroup, together with the Langlands dual~${^L}\P(\mathbb{C})={^L}\M(\mathbb{C}){^L}\N(\mathbb{C})$ in~${^L}G(\mathbb{C})$, where~${^L}\P(\mathbb{C})=\widehat{\P}(\mathbb{C})\rtimes \W_\F$, (and for a choice of Levi factor)~${^L}\M(\mathbb{C})=\widehat{\M}(\mathbb{C})\rtimes \W_\F$, and~${^L}\N(\mathbb{C})=\widehat{\N}(\mathbb{C})$.

	The Levi subgroup~${^L}\M(\mathbb{C})=\widehat{\M}(\mathbb{C})\rtimes \mathrm{W}_{F}$ acts on the Lie algebra~${^L}\mathfrak{n}$ of~${^L}\N(\mathbb{C})$ by conjugation, we write
	\[r\colon{^L}\M(\mathbb{C})\rightarrow \GL({^L}\mathfrak{n})\]
	for this representation.  This representation decomposes
	\begin{align}\label{equationr}
		r=\bigoplus_{i=1}^{m_r}r_i
	\end{align}
	where~$r_i$ is the subrepresentation of~$r$ defined on the vector space spanned by
	\[\{x_{\beta^{\vee}}\in{^L}\mathfrak{n}\mid\langle\widetilde{a},\beta^\vee\rangle=i\}.\]
	
	Let~$\widetilde{r}$ denote the contragredient of~$r$.  Let~$\overline{\P}$ denote a parabolic subgroup opposite to~$\P$ with Levi decomposition~$\overline{\P}=\M\overline{\N}$, and let~$\overline{r}\colon{^L}\M(\mathbb{C})\rightarrow \GL({^L}\overline{\mathfrak{n}})$ denote the adjoint action of~${^L}\M(\mathbb{C})$ on~${^L}\overline{\mathfrak{n}}$.   As this Lie algebra is the sum of root subspaces corresponding to the negatives of the original root spaces, one finds that the representations~$\widetilde{r_i}$ and~$\overline{r}_i$ are isomorphic, and~$m_r=m_{\widetilde{r}}=m_{\overline{r}}$.

	A strong version, due to Arthur, of the inductive lemma key to Shahidi's construction of gamma factors (cf.,~\cite[Remark 6.3.7]{MR3307916}) is:
	\begin{lemma}[{\cite{Arthur}}]\label{keylemma}
		Suppose~$\F/\mathbb{Q}_p$ is a~$p$-adic field.  Let~$(\B,\T)$ be a Borel-Torus pair for~$\G$,
		%
		and let~$\P$ be a maximal standard (with respect to~$\B$) parabolic subgroup of~$\G$, with standard Levi decomposition~$\P=\M\N$.  Let~$r\colon{^L}\M(\mathbb{C})\rightarrow \GL({^L}\mathfrak{n})$ denote the adjoint representation of~$^{L}\M(\mathbb{C})$ on~$^{L}\mathfrak{n}$.  Write~$r=\bigoplus_{i=1}^{m_r} r_i$ as above.  Then, for each~$i>1$, there exists a quasi-split reductive group~$\G_i$, a Borel-Torus pair~$(\B_i,\T_i)$ of~$\G_i$, and a maximal standard (with respect to~$\B_i$) parabolic~$\P_i$ of~$\G_i$, with Levi decomposition~$\P_i=\M\N_i$, such that the corresponding adjoint representation
		\[r'\colon{^L}\M(\mathbb{C})\rightarrow \GL({^L}\mathfrak{n}_i)\] of~$^{L}\M(\mathbb{C})$ on~$^{L}\mathfrak{n}_i$ decomposes (following the same recipe as that given for~$r$ above) as~\[r'=\bigoplus_{j=1}^{m_{r'}} r_j',\text{ with~}m_{r'}<m_r,\text{ and~}r_1'=r_i.\]  
	\end{lemma}
	
	An analogous lemma also holds for~$\F$ of positive characteristic (cf.,~\cite[§2.3]{lomelilemma}).  Note, that while we do not need it, Arthur shows in \emph{op.~cit.}~that~$\G_i$ can be chosen to be an endoscopic group for~$\G$.
	

	\subsection{Langlands-Shahidi gamma factors in families}
	We now interpolate Shahidi's local coefficients and gamma factors:
	\begin{thm}\label{mainthm}
		Let~$\R$ be a~$\R_0$-algebra.  Let~$\P=\M\N$ be a maximal standard parabolic subgroup of~$\G$, and~$\rho$ be a co-Whittaker type~$\R$-representation of~$\M$ with respect to~$\mathfrak w_\M$, and~$r=\bigoplus_{i=1}^{m_r}r_i$ denote the decomposition of~$r\colon{^L}\M(\mathbb{C})\rightarrow \GL({}^L\mathfrak{n})$ of~\eqref{equationr}.
		\begin{enumerate}
			\item \label{mainthmpart1} (Shahidi's local coefficient)
			The element~$\mathcal{D}_\R(X,\rho,\P,\Xi,\psi)$ of~Section \ref{localcoeffsect} is a unit.  In other words,
			\[\mathcal{D}_\R(X,\rho,\P,\Xi,\psi)\in (\mathcal{S}_{X,\R}^{-1}\R[X^{\pm 1}])^{\times},\]
			and we write~$\mathcal{C}_\R(X,\rho,\P,\Xi,\psi)=\mathcal{D}_\R(X,\rho,\P,\Xi,\psi)^{-1}$. 
		
			\item \label{Cintogamma} (Langlands--Shahidi gamma factors) For each~$i\in\{1,\ldots,m_r\}$, there exists a~
			 $$\gamma(X,\rho,r_i,\psi)\in (\mathcal{S}_{X,\R}^{-1}\R[X^{\pm 1}])^{\times},$$ depending only on the isomorphism class of~$\rho$, satisfying 
			\begin{enumerate}
				\item \label{maintheocomp} (Compatibility with homomorphisms)  Suppose~$\alpha\colon \R\rightarrow \R'$ is a morphism of~$\R_0$-algebras, and $\alpha_X\colon\mathcal{S}_{X,\R}^{-1}\R[X^{\pm 1}]\rightarrow \mathcal{S}_{X,\R'}^{-1}\R'[X^{\pm 1}]$ is the unique~$\R_0$-algebra morphism extending $\alpha$ which is the identity on~$X$. Then, for each~$i\in\{1,\ldots,m_r\}$, we have 
				\[\gamma(X,\rho\otimes_\alpha \R',r_i,\psi)=\alpha_X(\gamma(X,\rho,r_i,\psi)).\]
				\item\label{maintheoC} When~$\R=\mathbb{C}$, and~$\rho$ is a co-Whittaker~$\mathbb{C}$-representation with respect to $\mathfrak w_\M$ with $\mathfrak w_\M$-generic cosocle~$\rho_{gen}$, then
				\[\gamma(X,\rho,r_i,\psi)=\gamma_{Sh}(X,\rho_{gen},r_i,\psi),\]
				where~$\gamma_{Sh}(X,\rho_{gen},r_i,\psi)$ is the gamma factor constructed by Shahidi and Lomeli.
				\item \label{dominancemainthm}(Dominance) Let~$\rho'$ be a co-Whittaker type~$\R$-representation of~$\M$ with respect to $\mathfrak w_\M$ such that~$\rho$ dominates~$\rho'$. Then, for all~$i\in\{1,\ldots,m_r\}$, we have that
				
					$$\gamma(X,\rho,r_i,\psi)=\gamma(X,\rho',r_i,\psi).$$
			\end{enumerate}
			\item\label{LFE} (Local Functional Equation) For each~$i\in\{1,\ldots,m_r\}$,
			\[\gamma(X,\rho,r_i,\psi)\gamma(q^{-1}X^{-1},\rho,\overline{r_i},\psi^{-1})=1.\]
			and moreover,
			\[\mathcal{C}_{\R}(X,\rho,\P,\Xi,\psi)=\lambda(w_0,\psi)^{-1}\prod_{i=1}^{m_r}\gamma(X^i,\rho,r_i,\psi).\]
		

		\end{enumerate}
	\end{thm}
	
	\begin{proof}
	For ease of notation we will write $\Gamma$ for the universal co-Whittaker representation $\Gamma^{\M}_{\mathfrak w_\M,\mathfrak{t}}$ of type~$\mathfrak{t}$ and $\E$ for its endomorphism ring $\E^{\M}_{\mathfrak w_\M,\mathfrak{t}}$. As~$\E$ is torsion-free, all minimal primes of~$\E$ are characteristic zero, and as~$\E$ is a finitely generated~$\R_0$-algebra by Lemma \ref{Ebasicfacts}, for a minimal prime~$\mathfrak{p}$ of~$\E$, its residue field~$\kappa(\mathfrak{p})$ embeds into~$\mathbb{C}$. For any minimal prime~$\mathfrak{p}$ of~$\E$, we choose an embedding~$\kappa(\mathfrak{p})\hookrightarrow \mathbb{C}$, which induces a map $\iota_{\mathfrak p}\colon\E\to\mathbb C$. 
	
		We proceed via induction on $m_r$, so suppose first that $m_r=1$.
		  By Propositions \ref{Inverseshahidis} and \ref{propositiondowntointX},
		we obtain an element
		\[\mathcal{D}_{\E}(X,\Gamma,\P,\Xi,\psi)\in\mathcal{S}_{X,\E}^{-1}\E[X^{\pm 1}].\]
		
		Now let $\B'=(\B\cap\M)\overline{\N}_{\theta}$, which yields another Borel subgroup of $\G$ and $\overline{\P}$ is a standard maximal parabolic subgroup of $\G$ with respect to $\B'$. Let $\Delta'$ be the set of simple roots with respect to $\B'$. We choose an $\F$-pinning $\Xi'=(\B',\T,\{x_\alpha'\}_{\alpha\in\Delta'})$ on $\G$, where if $\Xi=(\B,\T,\{x_\alpha\}_{\alpha\in\Delta})$ we set for $\alpha\in\theta$ that
		$x_{\alpha}'(a)=x_{\alpha}(-a)$. Then a straightforward calculation shows that the Whittaker datum $\mathfrak w'$ associated to $(\Xi',\psi^{-1})$ satisfies that 
		$\mathfrak w_\M=\mathfrak w'_\M$. Propositions \ref{Inverseshahidis} and \ref{propositiondowntointX} again yield an element
		$$\mathcal{D}_{\E}(q^{-1}X^{-1},\Gamma,\overline{\P},\Xi',\psi^{-1})\in\mathcal{S}_{X,\E}^{-1}\E[X^{\pm 1}]$$
		
		We now consider
		$$
		\mathcal{D}_{\E}(X,\Gamma,\P,\Xi,\psi)\cdot\mathcal{D}_{\E}(q^{-1}X^{-1},\Gamma,\overline{\P},\Xi',\psi^{-1})-\lambda(w_0,\psi)\lambda(w_0',\psi^{-1}).$$
		For any minimal prime ideal $\mathfrak p$ of $\E$ we obtain by Proposition \ref{Xcoeffbasechange}, that
		\begin{align*}
		\iota_{\mathfrak p,X}(\mathcal{D}_{\E}(X,\Gamma,\P,\Xi,\psi)\cdot\mathcal{D}_{\E}(q^{-1}X^{-1},\Gamma,\overline{\P},\Xi',\psi^{-1})-\lambda(w_0,\psi)\lambda(w_0',\psi^{-1})) &=\\\notag\mathcal{D}_{\mathbb{C}}(X,\Gamma\otimes_{\iota_{\mathfrak p}}\mathbb C,\P,\Xi,\psi)\cdot\mathcal{D}_{\mathbb{C}}(q^{-1}X^{-1},\Gamma\otimes_{\iota_{\mathfrak p}}\mathbb C,\overline{\P},\Xi',\psi^{-1})-&\lambda(w_0,\psi)\lambda(w_0',\psi^{-1}).\end{align*}
		Now,~$\varsigma=\Gamma\otimes_{\iota_{\mathfrak p}}\mathbb C$ is a co-Whittaker~$\mathbb{C}$-representation of $\M$ with respect to $\mathfrak w_\M$, and by Lemma \ref{cowhittakerfields} has an irreducible $\mathfrak w_\M$-generic cosocle~$\varsigma_{gen}$.  Moreover, by Proposition \ref{propositiondowntointX}, we have that
		$$\mathcal{D}_{\mathbb{C}}(X,\varsigma,\P,\Xi,\psi)=\mathcal{D}_{\mathbb{C}}(X,\varsigma_{gen},\P,\Xi,\psi)$$
		and
		$$\mathcal{D}_{\mathbb{C}}(q^{-1}X^{-1},\varsigma,\overline{\P},\Xi',\psi^{-1})=\mathcal{D}_{\mathbb{C}}(q^{-1}X^{-1},\varsigma_{gen},\overline{\P},\Xi',\psi^{-1}).$$
	Now as $m_r=1$ we see by Theorem 3.5 in \cite{ShahidiAnnals}, and Equation (5.1) in \cite{lomelilemma} (see also Section 2.6 in \cite{AGIKMS}) that
	$$\mathcal{D}_{\mathbb{C}}(X,\varsigma_{gen},\P,\Xi,\psi)\gamma_{Sh}(X,\varsigma_{gen},r_1,\psi)=\lambda(w_0,\psi)$$
	and
	$$\mathcal{D}_{\mathbb{C}}(q^{-1}X^{-1},\varsigma_{gen},\overline{\P},\Xi',\psi^{-1})\gamma_{Sh}(q^{-1}X^{-1},\varsigma_{gen},\overline{r_1},\psi^{-1})=\lambda(w_0',\psi^{-1}).$$
	Since by Theorem 11.7 3) of Kim's article in \cite{cogdell2004lectures}, $$\gamma_{Sh}(X,\varsigma_{gen},r_1,\psi)\gamma_{Sh}(q^{-1}X^{-1},\varsigma_{gen},\overline{r_1},\psi^{-1})=1$$
	we obtain that
		\[\mathcal{D}_{\mathbb{C}}(X,\varsigma_{gen},\P,\Xi,\psi)\mathcal{D}_{\mathbb{C}}(q^{-1}X^{-1},\varsigma_{gen},\overline{\P},\Xi',\psi^{-1})=\lambda(w_0,\psi)\lambda(w_0',\psi^{-1}).\]
		
Hence, we deduce that for any minimal prime~$\mathfrak{p}$ of~$\E$,
		\begin{equation*}\iota_{\mathfrak p,X}(\mathcal{D}_{\E}(X,\Gamma,\P,\Xi,\psi)\mathcal{D}_{\E}(q^{-1}X^{-1},\Gamma,\overline{\P},\Xi',\psi^{-1})-\lambda(w_0,\psi)\lambda(w_0',\psi^{-1}))=0.\end{equation*}
		As~$\E$ is reduced, the intersection of its minimal primes is trivial, thus we find
		$$\mathcal{D}_{\E}(X,\Gamma,\P,\Xi,\psi)\mathcal{D}_{\E}(q^{-1}X^{-1},\Gamma,\overline{\P},\Xi',\psi^{-1})=\lambda(w_0,\psi)\lambda(w_0',\psi^{-1})$$
		and hence~$\mathcal{D}_{\E}(X,\Gamma,\P,\Xi,\psi)\in (\mathcal{S}_{X,\E}^{-1}\E[X^{\pm1}])^\times$. 
		
		For a co-Whittaker $\R$-representation $\rho$ of $\M$ with respect to $\mathfrak w_\M$ of type $\mathfrak t$ we obtain an $\R_0$-algebra morphism $f_{\rho,0}\colon\E\to\R$ such that $\Gamma\otimes_{f_{\rho,0}}\R$ dominates $\rho$. By Propositions \ref{Xcoeffbasechange} and \ref{propositiondowntointX} we obtain that
		\begin{align*}
			\mathcal{D}_{\R}(X,\rho,\P,\Xi,\psi)\mathcal{D}_{\R}(q^{-1}X^{-1},\rho,\overline{\P},\Xi',\psi^{-1})&=	\mathcal{D}_{\R}(X,\Gamma\otimes_{f_{\rho,0}}\R,\P,\Xi,\psi)\mathcal{D}_{\R}(q^{-1}X^{-1},\Gamma\otimes_{f_{\rho,0}}\R,\overline{\P},\Xi',\psi^{-1})\\
			&=(f_{\rho,0})_X(\mathcal{D}_{\E}(X,\Gamma,\P,\Xi,\psi)\mathcal{D}_{\E}(q^{-1}X^{-1},\Gamma,\overline{\P},\Xi',\psi^{-1}))\\
			&=\lambda(w_0,\psi)\lambda(w_0',\psi^{-1}),
		\end{align*}
		which implies that $\mathcal{D}_{\R}(X,\rho,\P,\Xi,\psi)\in(\mathcal S_{X,\R}^{-1}\R[X^{\pm1}])^{\times}$. Hence we can define 
		$$\gamma(X,\rho,r_1,\psi)\coloneqq\lambda(w_0,\psi)\mathcal{D}_{\R}(X,\rho,\P,\Xi,\psi)^{-1}\in (\mathcal S_{X,\R}^{-1}\R[X^{\pm1}])^{\times}.$$
		Part \ref{LFE} then follows immediately and Properties \ref{maintheocomp} and \ref{dominancemainthm} follow from Propositions \ref{Xcoeffbasechange} and \ref{propositiondowntointX}. Part \ref{maintheoC} follows since our construction is compatible with the one for representations over $\mathbb C$ of Shahidi and Lomeli.

		Now suppose~$m_r>1$, and that for any~$\G'$ with maximal parabolic subgroup $\P'=\M\N'$, such that the corresponding adjoint representation  $r'\colon{^L}\M(\mathbb C)\to\GL(^L\mathfrak n')$ satisfies
		$$r'=\bigoplus_{i=1}^{m_{r'}}r_i'$$
		with~$m_{r'}<m_r$, we have defined gamma factors~$\gamma(X,\rho,r'_j,\psi)$ and proved the stated properties. In particular, by Lemma \ref{keylemma}, we can thus suppose that for all~$1<i\leq m_r$, we have defined~$\gamma(X,\Gamma,r_i,\psi)\in (\mathcal{S}_{X,\E}^{-1}\E[X^{\pm1}])^{\times}$.

		Note that these gamma factors are independent of the choice of~$\G'$:  Indeed, if~$\gamma'(X,\Gamma,r'_j,\psi)$ is defined with respect to another choice~$\G''$, then we can consider
		\[\gamma(X,\Gamma,r'_j,\psi)-\gamma'(X,\Gamma,r'_j,\psi)\]
		and we find that extending scalars to~$\kappa(\mathfrak{p})$ for each minimal prime~$\mathfrak{p}$ of~$\E$ this difference is zero (using Shahidi's gamma factors do not depend on the choice of~$\G'$), hence as~$\E$ is reduced we find~$\gamma(X,\Gamma,r'_j,\psi)=\gamma'(X,\Gamma,r'_j,\psi)$.  In particular, this inductive process is well-defined.

		By Propositions \ref{Inverseshahidis} and \ref{propositiondowntointX},
		we obtain elements
		\[\mathcal{D}_{\E}(X,\Gamma,\P,\Xi,\psi),\mathcal{D}_{\E}(q^{-1}X^{-1},\Gamma,\overline{\P},\Xi',\psi^{-1})\in\mathcal{S}_{X,\E}^{-1}\E[X^{\pm 1}]\]  and consider
				\begin{equation}\label{maintheoeq}\mathcal{D}_{\E}(X,\Gamma,\P,\Xi,\psi)\mathcal{D}_{\E}(q^{-1}X^{-1},\Gamma,\overline{\P},\Xi',\psi^{-1})-\frac{\lambda(w_0,\psi)\lambda(w_0',\psi^{-1})}{\prod_{i=2}^{m_r}\gamma(X^{i},\Gamma,r_i,\psi)\gamma(q^{-i}X^{-i},\Gamma,\overline{r_i},\psi^{-1})}.\end{equation}
		
		For any minimal prime ideal $\mathfrak p$ of $\E$ we obtain by the induction hypothesis and Proposition \ref{Xcoeffbasechange}, that
		the image of \eqref{maintheoeq} under $\iota_{\mathfrak p,X}$ equals
		$$\mathcal{D}_{\mathbb C}(X,\Gamma\otimes_{\iota_{\mathfrak p}}\mathbb C,\P,\Xi,\psi)\mathcal{D}_{\mathbb C}(q^{-1}X^{-1},\Gamma\otimes_{\iota_{\mathfrak p}}\mathbb C,\overline{\P},\Xi',\psi^{-1})-\frac{\lambda(w_0,\psi)\lambda(w_0',\psi^{-1})}{\prod_{i=2}^{m_r}\gamma(X^{i},\Gamma\otimes_{\iota_{\mathfrak p}}\mathbb C,r_i,\psi)\gamma(q^{-i}X^{-i},\Gamma\otimes_{\iota_{\mathfrak p}}\mathbb C,\overline{r_i},\psi^{-1})}.$$
			Again,~$\varsigma=\Gamma\otimes_{\iota_{\mathfrak p}}\mathbb C$ is a co-Whittaker~$\mathbb{C}$-representation of $\M$ with respect to $\mathfrak w_\M$, and by Lemma \ref{cowhittakerfields} has an irreducible $\mathfrak w_\M$-generic cosocle~$\varsigma_{gen}$. By the induction hypothesis and Proposition \ref{propositiondowntointX} the image of \eqref{maintheoeq} under $\iota_{\mathfrak p,X}$ equals 
			$$\mathcal{D}_{\mathbb C}(X,\varsigma_{gen},\P,\Xi,\psi)\mathcal{D}_{\mathbb C}(q^{-1}X^{-1},\varsigma_{gen},\overline{\P},\Xi',\psi^{-1})-\frac{\lambda(w_0,\psi)\lambda(w_0',\psi^{-1})}{\prod_{i=2}^{m_r}\gamma(X^{i},\varsigma_{gen},r_i,\psi)\gamma(q^{-i}X^{-i},\varsigma_{gen},\overline{r_i},\psi^{-1})}.$$
		However, by Theorem 3.5 in \cite{ShahidiAnnals}, Equation (5.1) in \cite{lomelilemma}, and Theorem 11.7 3) of Kim's article in \cite{cogdell2004lectures} this expression equals zero. We obtain that for any minimal prime~$\mathfrak{p}$ of~$\E$, the image of \eqref{maintheoeq} under $\iota_{\mathfrak p,X}$ is zero and since the intersection of all minimal prime ideals of $\E$ is zero, we obtain that 
		$$\mathcal{D}_{\E}(X,\Gamma,\P,\Xi,\psi)\mathcal{D}_{\E}(q^{-1}X^{-1},\Gamma,\overline{\P},\Xi',\psi^{-1})=\frac{\lambda(w_0,\psi)\lambda(w_0',\psi^{-1})}{\prod_{i=2}^{m_r}\gamma(X^{i},\Gamma,r_i,\psi)\gamma(q^{-i}X^{-i},\Gamma,\overline{r_i},\psi^{-1})}.$$
		For a co-Whittaker $\R$ representation $\rho$ with respect to $\mathfrak w_\M$ of type~$\mathfrak{t}$ we obtain an $\R_0$-algebra morphism $f_{\rho,0}\colon\E\to\R$ such that $\Gamma\otimes_{f_{\rho,0}}\R$ dominates $\rho$. By using \ref{Xcoeffbasechange} and \ref{propositiondowntointX} and the induction hypothesis a quick computation yields that 
		\begin{equation}\label{eqmaintheo2}
			\mathcal{D}_{\R}(X,\rho,\P,\Xi,\psi)\mathcal{D}_{\R}(q^{-1}X^{-1},\rho,\overline{\P},\Xi',\psi^{-1})=\frac{\lambda(w_0,\psi)\lambda(w_0',\psi^{-1})}{\prod_{i=2}^{m_r}\gamma(X^{i},\rho,r_i,\psi)\gamma(q^{-i}X^{-i},\rho,\overline{r_i},\psi^{-1})}.\end{equation}
		By the induction hypothesis the right hand side lies in $(\mathcal S_{X,\R}^{-1}\R[X^{\pm1}])^{\times}$ and hence also 
		$$\mathcal{D}_{\R}(X,\rho,\P,\Xi,\psi)\in (\mathcal S_{X,\R}^{-1}\R[X^{\pm1}])^{\times}.$$
		We can thus define
		$$\gamma(X,\rho,r_1,\psi)\coloneqq\lambda(w_0,\psi)\mathcal{D}_{\R}(X,\rho,\P,\Xi,\psi)^{-1}\prod_{i=2}^{m_r}\gamma(X^{i},\rho,r_i,\psi)^{-1}.$$
		Property \ref{LFE} then immediately follows from the definition and Equation \eqref{eqmaintheo2}. Parts \ref{maintheocomp} and \ref{dominancemainthm} follow from Propositions \ref{Xcoeffbasechange} and \ref{propositiondowntointX}. Part \ref{maintheoC} follows since our construction is compatible with Shahidi's and Lomeli's constructions over $\mathbb C$.	\end{proof}

	\subsection{On the $j$-function for co-Whittaker representations}
	We fix a prime number~$\ell\neq p$, and let
	\begin{equation*}\mathbb{Z}_{\ell}':=\begin{cases}\mathbb{Z}_{\ell}^{ur}&\text{if~}\ell\neq 2;\\
			\mathbb{Z}_{\ell}^{ur}[\sqrt{q},\mu_4]&\text{if~}\ell=2;\end{cases}
	\end{equation*} where~$\mathbb{Z}_{\ell}^{ur}$ denotes the ring of integers of the maximal unramified extension of~$\mathbb{Q}_{\ell}$. Suppose for this section that $\R$ is a Noetherian $\mathbb{Z}_\ell'$-algebra. Let~$\P=\M\N$ be a standard maximal parabolic subgroup of~$\G$, and~$\rho$ be a co-Whittaker type~$\R$-representation of~$\M$ with respect to~$\mathfrak w_\M$ of type $\mathfrak t$. We fix the same choices as in Section \ref{secchoices}. Moreover, for this section we assume that $\P$ is self-associate, i.e.\ if $\P$ is associated to the subset $\theta\subset\Delta$, then $w_0(\theta)=\theta$. Note that this implies that $w_0\P w_0^{-1}=\overline{\P}$. Moss and Trias in \cite{MT} (in a more general setting) have introduced the Harish--Chandra $j$-function~$j^{\G}_\R(\rho)$ as an element of~$ \mathcal{S}_{\M,\R}^{-1}\R[\M/\M^{\circ}]$, as a composition of intertwining operators involving~$\P$ and $\overline{\P}$. Hence the $j$-function depends on the choice of measures on $\N$ and $\overline{\N}$ and we choose the same measure for $\N$ as in Section \ref{secchoices} which since $w_0\N w_0^{-1}=\overline{\N}$ also defines a measure on $\overline{\N}$.
	The $j$-function is independent of the choice of standard parabolic~$\P$.

	It follows by Theorem \ref{intertwiningops} that we can find an element $\mathcal Q\in \mathcal{S}_{\M,\R}$ such that $\mathcal Q\cdot j_\R^{\G}(\rho)\in \R[\M/\M^\circ]$ and such that $\Delta_{X,\R}(\mathcal Q)\in\mathcal S_{X,\R}$. We then define $$j_\R^{\G}(X,\rho)\coloneqq\Delta_{X,\R}(\mathcal Q)^{-1}\Delta_{X,\R}(\mathcal Q\cdot j_\R^\G(\rho))\in\mathcal S_{X,\R}^{-1}\R[X^{\pm c}],$$
	and it is straightforward to see that this does not depend on the choice of $\mathcal Q$.
	
	By Proposition 6.3 of \cite{MT} the $j$-function behaves well with respect to base-change, i.e.\ if $\alpha\colon \R\to\R'$ is an $\R_0$-algebra homomorphism and $\alpha_\M\colon\mathcal S_{\M,\R}^{-1}\R[\M/\M^\circ]\to\mathcal S_{\M,\R'}^{-1}\R'[\M/\M^\circ]$ the unique extension of $\alpha$ that is the identity on elements of $\M/\M^\circ$, then $$j^\G_{\R'}(\rho\otimes_\alpha\R')=\alpha_\M(j^\G_\R(\rho))\in\mathcal S_{\M,\R'}^{-1}\R'[\M/\M^\circ].$$
	Moreover, it is straightforward to see that this implies for $\alpha_X\colon\mathcal S_{X,\R}^{-1}\R[X^{\pm c}]\to\mathcal S_{X,\R'}^{-1}\R'[X^{\pm c}]$, the unique extension of $\alpha$ that is the identity on $X^c$, that
	\begin{equation}\label{jbasechange}
		j^\G_{\R'}(X,\rho\otimes_\alpha\R')=\alpha_X(j^\G_\R(X,\rho))\in\mathcal S_{X,\R'}^{-1}\R'[X^{\pm c}].
	\end{equation}
	Furthermore, note that by Proposition 6.3 of \cite{MT}, if $\rho'$ is another co-Whittaker $\R$-representation of $\M$ with respect to $\mathfrak w_\M$ such that $\rho$ dominates $\rho'$, then the equality
	$j^\G_\R(\rho)=j^\G_\R(\rho')\in\mathcal S_{\M,\R}^{-1}\R[\M/\M^\circ]$
	holds and in particular
	\begin{equation}\label{jdomination}
		j^\G_\R(X,\rho)=j^\G_\R(X,\rho').
	\end{equation}

	We define a new $\F$-pinning $\Xi^{-}=(\B,\T,\{x_\alpha'\}_{\alpha\in\Delta})$, where if $\Xi=(\B,\T,\{x_\alpha\}_{\alpha\in\Delta})$, we set for $\alpha\in\Delta$ that
	$x_{\alpha}'(a)=x_{\alpha}(-a)$.
	\begin{corollary}
		Let~$\R$ be a Noetherian $\mathbb{Z}_\ell'$-algebra. Moreover, let~$\P=\M\N$ be a maximal self-associate standard parabolic subgroup of~$\G$, and~$\rho$ be a co-Whittaker type~$\R$-representation of~$\M$ with respect to~$\mathfrak w_\M$ of type $\mathfrak t$. Then
		$$j^\G_\R(X,\rho)= \mathcal{D}_\R(X^{-1},{}^{\widetilde{w}_0}\rho,\P,\Xi^-,\psi^{-1})\cdot\mathcal{D}_\R(X,\rho,\P,\Xi,\psi)$$
		and $j^{\G}_\R(X,\rho)\in  (\mathcal{S}_{X,\R}^{-1}\R[X^{\pm 1}])^{\times}$.
	\end{corollary}
	\begin{proof} 
		
		The proof is similar to some of the arguments above in the sense that we first deduce the result for the universal co-Whittaker module with respect to $\mathfrak w_\M$ of type~$\mathfrak{t}$ and then show that this implies the result. For ease of notation, we set~$\E=\E^{\M}_{\mathfrak w_\M,\mathbb{Z}_\ell',\mathfrak{t}}$ and~$\Gamma=\Gamma^{\M}_{\mathfrak w_\M,\mathbb{Z}_\ell',\mathfrak{t}}$. We consider
		$$j^\G_\E(X,\Gamma)-\mathcal{D}_\E(X^{-1},{}^{\widetilde{w}_0}\Gamma,\P,\Xi^{-},\psi^{-1})\cdot\mathcal{D}_\E(X,\Gamma,\P,\Xi,\psi)$$
		
		As~$\E$ is torsion-free, all minimal primes of~$\E$ are characteristic zero, and as~$\E$ is a finitely generated~$\mathbb{Z}_\ell'$-algebra by Lemma \ref{Ebasicfacts}, for a minimal prime~$\mathfrak{p}$ of~$\E$, its residue field~$\kappa(\mathfrak{p})$ embeds into~$\mathbb{C}$.  For any minimal prime~$\mathfrak{p}$ of~$\E$, choose an embedding~$\kappa(\mathfrak{p})\hookrightarrow \mathbb{C}$ and we overall obtain an $\mathbb{Z}_\ell'$-algebra homomorphism $\iota_{\mathfrak p}\colon\E\to\mathbb C$.
		We set $\varsigma=\Gamma\otimes_{\iota_{\mathfrak p}}\mathbb C$, which is a co-Whittaker $\mathbb C$-representation.
		By Equation (\ref{jbasechange}) and Proposition \ref{Xcoeffbasechange} we see that
		\begin{multline*}
			\iota_{\mathfrak p,X}(j^\G_\E(X,\Gamma)-\mathcal{D}_\E(X^{-1},{}^{\widetilde{w}_0}\Gamma,\P,\Xi^{-},\psi^{-1})\cdot\mathcal{D}_\E(X,\Gamma,\P,\Xi,\psi))=\\
			j^{\G}_{\mathbb C}(X,\varsigma)-\mathcal{D}_{\mathbb C}(X^{-1},{}^{\widetilde{w}_0}\varsigma,\P,\Xi^{-},\psi^{-1})\cdot\mathcal{D}_{\mathbb C}(X,\varsigma,\P,\Xi,\psi).
		\end{multline*}

		By Lemma \ref{cowhittakerfields} we obtain that the cosocle $\varsigma_{gen}$ of $\varsigma$ is irreducible and $\mathfrak w_\M$-generic. By Equation (\ref{jdomination}) and Proposition \ref{propositiondowntointX} we obtain that 
	\begin{multline*}
		j^{\G}_{\mathbb C}(X,\varsigma)-\mathcal{D}_{\mathbb C}(X^{-1},{}^{\widetilde{w}_0}\varsigma,\P,\Xi^{-},\psi^{-1})\cdot\mathcal{D}_{\mathbb C}(X,\varsigma,\P,\Xi,\psi)=\\j^{\G}_{\mathbb C}(X,\varsigma_{gen})-\mathcal{D}_{\mathbb C}(X^{-1},{}^{\widetilde{w}_0}\varsigma_{gen},\P,\Xi^{-},\psi^{-1})\cdot\mathcal{D}_{\mathbb C}(X,\varsigma_{gen},\P,\Xi,\psi).
	\end{multline*}
		However, Corollary 5.3.1 in \cite{shahidi2025eisenstein} implies that 
		$$j^{\G}_{\mathbb C}(X,\varsigma_{gen})-\mathcal{D}_{\mathbb C}(X^{-1},{}^{\widetilde{w}_0}\varsigma_{gen},\P,\Xi^{-},\psi^{-1})\cdot\mathcal{D}_{\mathbb C}(X,\varsigma_{gen},\P,\Xi,\psi)=0.$$
		As $\E$ is reduced, the intersection of its minimal prime ideals is trivial which implies that 
		$$j^\G_\E(X,\Gamma)-\mathcal{D}_\E(X^{-1},{}^{\widetilde{w}_0}\Gamma,\P,\Xi^{-},\psi^{-1})\cdot\mathcal{D}_\E(X,\Gamma,\P,\Xi,\psi)=0.$$
		
		Now there is a $\mathbb{Z}_\ell'$-algebra homomorphism $f_\rho\colon\E\to\R$ such that $\Gamma\otimes_{f_\rho}\R$ dominates $\rho$, i.e.\ there is a surjective $\R[\M]$ morphism $\Gamma\otimes_{f_\rho}\R\twoheadrightarrow\rho$. We can again use Equations (\ref{jbasechange}) and (\ref{jdomination}) and Propositions \ref{Xcoeffbasechange} and \ref{propositiondowntointX} to compute
		\begin{align*}
			j_{\R}^\G(X,\rho)&=j_{\R}^\G(X,\Gamma\otimes_{f_\rho}\R)\\
			&=(f_\rho)_X(j_{\E}^\G(X,\Gamma))\\
			&=(f_\rho)_X(\mathcal{D}_\E(X^{-1},{}^{\widetilde{w}_0}\Gamma,\P,\Xi^{-},\psi^{-1}))\cdot (f_\rho)_X(\mathcal{D}_\E(X,\Gamma,\P,\Xi,\psi))\\
			&=\mathcal{D}_\R(X^{-1},{}^{\widetilde{w}_0}(\Gamma\otimes_{f_\rho}\R),\P,\Xi^{-},\psi^{-1})\cdot\mathcal{D}_\R(X,\Gamma\otimes_{f_\rho}\R,\P,\Xi,\psi)\\
			&=\mathcal{D}_\R(X^{-1},{}^{\widetilde{w}_0}\rho,\P,\Xi^{-},\psi^{-1})\cdot\mathcal{D}_\R(X,\rho,\P,\Xi,\psi),
		\end{align*}
		which proves the first part of this Corollary. Theorem \ref{mainthm} (\ref{mainthmpart1}) implies that $\mathcal{D}_\R(X^{-1},{}^{\widetilde{w}_0}\rho,\P,\Xi^{-},\psi^{-1})\cdot\mathcal{D}_\R(X,\rho,\P,\Xi,\psi)$ is an element of $(\mathcal{S}_{X,\R}^{-1}\R[X^{\pm 1}])^{\times}$ and hence the same is true for $j^{\G}(X,\rho)$.
		\end{proof}
	
	\subsection{Compatibility with the Rankin--Selberg method and Moss's work}
 For any positive integer $n$ let $\U_n$ be the standard unipotent subgroup of $\GL_n(\F)$ and let $\Psi_n$ be the standard nondegenerate character of $\U_n$ induced by $\psi$. We fix positive integers~$n>m$ and let~$\G=\GL_{m+n}(\F)$, and~$\P_{m,n}$ be the standard upper triangular parabolic in~$\G$ with Levi factor~$\M=\GL_{n}(\F)\times\GL_m(\F)$. In this case, as in Example \ref{RSexample}, the representation~$r\simeq \phi_{\GL_n}\otimes \phi_{\GL_m}^{\vee}$ is irreducible.	
	
	We fix a prime number $\ell\not=p$, and let~$\W(\Fl)'=\W(\Fl)[\sqrt{q}]$. Let $\mathfrak t$ be a finite sum of primitive idempotents in $\mathfrak{Z}_{\W(\Fl)'}(\GL_n(\F))$ and $\mathfrak t'$ a finite sum of primitive idempotents in $\mathfrak{Z}_{\W(\Fl)'}(\GL_m(\F))$.  	Let~$\R_1$ and $\R_2$ be Noetherian~$\W(\Fl)'$-algebras.  Let~$\rho_1$ be a co-Whittaker~$\R_1$-representation of~$\GL_{n}(\F)$ with respect to~$\mathfrak w_n=(\U_n,\Psi_n)$ of type $\mathfrak t$ and~$\rho_2$ be a co-Whittaker~$\R_2$-representation of~$\GL_{m}(\F)$ with respect to~$\mathfrak w_m=(\U_m,\Psi_m)$ of type $\mathfrak t'$. This defines a co-Whittaker~$\R=\R_1\otimes_{\W(\Fl)'} \R_2$-representation ~$\rho=\rho_1\otimes\rho_2$ of~$\M$ with respect to~$\mathfrak w_{m,n}=(\U_n\times\U_m,\Psi_n\otimes\Psi_m)$.   In this setting, we have the gamma factor of Theorem \ref{mainthm}
	\[\gamma(X,\rho,r,\psi)\in (\mathcal{S}_{X,\R}^{-1}\R[X^{\pm 1}])^\times.\]
	
	Let~${\rho}_2^{\iota}$ be the co-Whittaker dual of~$\rho_2$, as defined in~\cite[§2.1]{HM}, i.e.,~${\rho}_2^{\iota}$ has the same underlying module as~$\rho_2$, but~${\rho}_2^{\iota}(g)=\rho_2({^t}g^{-1})$.  
	In this setting, with coefficients in~$\R$, Moss in \cite{Moss} also defines a gamma factor
	\[\gamma_{RS}(X,\rho_1,\rho_2^{\iota},\psi)\in  (\mathcal{S}_{X,\R}^{-1}\R[X^{\pm 1}])^\times\]
	interpolating the Rankin--Selberg method. 
	
	\begin{corollary}\label{Mossequality}
		Let~$\R_1,\R_2$ be Noetherian~$\W(\Fl)'$-algebras. Let~$\rho_1$ be a co-Whittaker~$\R_1$-representation of~$\GL_{n}(\F)$ with respect to~$(\U_n,\Psi_n)$ of type $\mathfrak t$ and~$\rho_2$ be a co-Whittaker~$\R_2$-representation of~$\GL_{m}(\F)$ with respect to~$(\U_m,\Psi_m)$ of type $\mathfrak t'$. Then with the notation as above the equation
		\[\gamma(X,\rho,r,\psi)=\gamma_{RS}(q^{(n-m)/2}X,\rho_1,\rho_2^{\iota},\psi)\]
		holds.
	\end{corollary}
	\begin{proof}
		
		We write~$\Gamma_n=\Gamma^{\GL_n(\F)}_{\,\mathfrak{t}}$ for the universal co-Whittaker representation of type $\mathfrak t$ of $\GL_n(\F)$ and $\E_n=\E^{\GL_n(\F)}_{\mathfrak w_n,\mathfrak{t}}$ for its endomorphism ring. Similarly we set~$\E_m=\E^{\GL_m(\F)}_{\mathfrak w_m,\mathfrak{t'}}$ and~$\Gamma_m=\Gamma^{\GL_m(\F)}_{\mathfrak w_m,\mathfrak{t'}}$. Note that by Theorem 12.8 of \cite{HelmBC} both $\E_n$ and $\E_m$ are $\ell$-torsion free and reduced. From this we obtain by Lemma 5.1 in \cite{Moss} that $\E_n\otimes_{\W(\Fl)'}\E_m$ is also $\ell$-torsion free and reduced. In particular the residue fields at minimal primes have characteristic zero and the intersection of all minimal prime ideals is trivial. We can then consider by Theorem \ref{mainthm}
		$$\gamma(X,\Gamma_n\otimes\Gamma_m,r,\psi)\in\mathcal S_{X,\E_n\otimes\E_m}^{-1}(\E_n\otimes\E_m[X^{\pm 1}])$$
		and Theorem 5.3 in \cite{Moss},
		$$\gamma_{RS}(q^{(n-m)/2}X,\Gamma_n,\Gamma_m^{\iota},\psi)\in\mathcal S_{X,\E_n\otimes\E_m}^{-1}(\E_n\otimes\E_m[X^{\pm 1}]).$$
		
		Since by Theorem 12.8 of \cite{HelmBC}, $\E_n\otimes_{\W(\Fl)'}\E_m$ is a finitely generated $\W(\Fl)'$-algebra, we can choose for a minimal prime ideal $\mathfrak p$ of $\E_n\otimes_{\W(\Fl)'}\E_m$ an embedding $\kappa(\mathfrak p)\hookrightarrow\mathbb C$. We obtain a $\W(\Fl)'$-morphism $$\iota_{\mathfrak p}\colon\E_n\otimes_{\W(\Fl)'}\E_m\to\mathbb C,$$
		which gives rise to two morphisms $\iota_{\mathfrak p}^1\colon \E_n\to\mathbb C$ and $\iota_{\mathfrak p}^2\colon\E_m\to\mathbb C$.
		We then have by Theorem \ref{mainthm} (2) and Corollary 5.5 of \cite{Moss} that
		\begin{align*}
		\iota_{\mathfrak p,X}(\gamma(X,\Gamma_n\otimes\Gamma_m,r,\psi)-&\gamma_{RS}(q^{(n-m)/2}X,\Gamma_n,\Gamma_m^{\iota},\psi))\\&=\gamma(X,(\Gamma_n\otimes\Gamma_m)\otimes_{\iota_{\mathfrak p}}\mathbb C,r,\psi)-\gamma_{RS}(q^{(n-m)/2}X,\Gamma_n\otimes_{\iota_{\mathfrak p}^1}\mathbb C,\Gamma_m^{\iota}\otimes_{\iota_{\mathfrak p}^2}\mathbb C,\psi).\end{align*}
		Note that $(\Gamma_n\otimes\Gamma_m)\otimes_{\iota_{\mathfrak p}}\mathbb C\cong (\Gamma_n\otimes_{\iota_{\mathfrak p}^1}\mathbb C)\otimes_{\mathbb C}(\Gamma_m\otimes_{\iota_{\mathfrak p}^2}\mathbb C)$ and $\Gamma_m^{\iota}\otimes_{\iota_{\mathfrak p}^2}\mathbb C\cong (\Gamma_m\otimes_{\iota_{\mathfrak p}^2}\mathbb C)^{\iota}$. Now $\varsigma^1=\Gamma_n\otimes_{\iota_{\mathfrak p}^1}\mathbb C$ (respectively $\varsigma^2=\Gamma_m\otimes_{\iota_{\mathfrak p}^2}\mathbb C$) has by Lemma \ref{cowhittakerfields} an $(\U_n,\Psi_n)$-generic (respectively $(\U_m,\Psi_m)$-generic) irreducible cosocle $\varsigma^1_{gen}$ (respectively $\varsigma^2_{gen}$). Again by Theorem \ref{mainthm} (2) and Corollary 5.5 of \cite{Moss} (where we use that $f_{\varsigma^i,0}=f_{\varsigma^i_{gen},0}$ for $i=1,2$) we obtain that
		$$\gamma(X,\varsigma^1\otimes\varsigma^2,r,\psi)-\gamma_{RS}(q^{(n-m)/2}X,\varsigma^1,(\varsigma^2)^{\iota},\psi)=\gamma(X,\varsigma^1_{gen}\otimes\varsigma^2_{gen},r,\psi)-\gamma_{RS}(q^{(n-m)/2}X,\varsigma^1_{gen},(\varsigma^2)^{\iota}_{gen},\psi).$$
		However, by \cite{HenniartLomeliUniquenessofRSproducts} the above expression is zero and since $\E_n\otimes_{\W(\Fl)'}\E_m$ is reduced, the intersection of all of its minimal prime ideals is zero and hence
		$$\gamma(X,\Gamma_n\otimes\Gamma_m,r,\psi)=\gamma_{RS}(q^{(n-m)/2}X,\Gamma_n,\Gamma_m^{\iota},\psi).$$

		Now there is a $\W(\Fl)'$-algebra homomorphism $f_{\rho_1,0}\colon\E_n\to\R_1$ such that $\Gamma_n\otimes_{f_{\rho_1,0}}\R_1$ dominates $\rho_1$ and analogously a $\W(\Fl)'$-algebra homomorphism $f_{\rho_2,0}\colon\E_m\to\R_2$ such that $\Gamma_m\otimes_{f_{\rho_2,0}}\R_2$ dominates $\rho_2$. We can again use Theorem \ref{mainthm} \ref{maintheocomp} and Corollary 5.5 of \cite{Moss} to compute
		\begin{align*}
			\gamma(X,\rho_1\otimes\rho_2,r,\psi)&=\gamma(X,(\Gamma_n\otimes_{f_{\rho_1,0}}\R_1)\otimes(\Gamma_m\otimes_{f_{\rho_2,0}}\R_2),r,\psi)\\
			&=(f_{\rho_1,0}\otimes f_{\rho_2,0})_X(\gamma(X,\Gamma_n\otimes\Gamma_m,r,\psi))\\
			&=(f_{\rho_1,0}\otimes f_{\rho_2,0})_X(\gamma_{RS}(q^{(n-m)/2}X,\Gamma_n,\Gamma_m^{\iota},\psi))\\
			&=\gamma_{RS}(q^{(n-m)/2}X,\rho_1,\rho_2^{\iota},\psi),
		\end{align*}
		which proves the result.
	\end{proof}

	\section{Fundamental representations for the classical groups}\label{notationclassical}
	In this section, for a classical~$p$-adic group~$\G$, we fix notation for the~$L$-morphisms~$\phi_{\LG}:\LG(\mathbb{C})\rightarrow \presuper{L}\GL_{N}(\mathbb{C})$ associated with functorial lifting and local transfer.  We assume that~$\F/\mathbb{Q}_p$ is a~$p$-adic field.
	
	\subsection{Split classical groups} Let~$\G=\mathbb{G}(\F)$ be a general linear, symplectic, or split special orthogonal group. We let
	\[\J_m=\left(\begin{smallmatrix} &&1\\&\Ddots&\\1&&\end{smallmatrix}\right)\in \GL_m(\mathbb{Z}).\]
	We choose models for the split classical groups as follows:
	\begin{align*}
		\SO_{k}(\F)&=\{g\in\SL_{k}(\F)\mid g^T J_k g=J_k\};\\
		\mathrm{Sp}_{2n}(\F)&=\left\{g\in\GL_{2n}(\F)\mid g^T \left(\begin{smallmatrix} 0&J_n\\-J_n&0\end{smallmatrix}\right) g=\left(\begin{smallmatrix} 0&J_n\\-J_n&0\end{smallmatrix}\right) \right\}.\end{align*}
	where we omit the case~$\SO_{2}(\F)\simeq \F^\times$ here.  
	%
	With respect to these models, we choose our Borel subgroup~$\B$ to be upper triangular matrices in~$\G$, and our torus~$\T$ to be diagonal matrices in~$\G$, and fix~$\Xi$ any pinning compatible with~$(\B,\T)$.
	
	We denote by
	\[\phi_{\widehat{\G}}\colon\widehat{\G}(\mathbb{C})\hookrightarrow \GL_{N}(\mathbb{C})\]
 the \emph{fundamental representation} (or \emph{standard representation}) of~$\widehat{\G}(\mathbb{C})$.  We extend this trivially on~$\W_F$,  to define the embedding 
	\[\phi_{\LG}\colon\LG(\mathbb{C})\hookrightarrow \GL_{N}(\mathbb{C}) \times \W_{\F}=\presuper{L}\GL_N(\mathbb{C}).\]
	For the split classical groups we have:
	\begin{table}[ht]
		\centering
		\begin{tabular}{c c c c}
			\hline
			$\G$ & $\widehat{\G}(\mathbb{C})$ & $N = N(\G)$ & $\phi_{\LG}$ \\
			\hline
			$\mathrm{GL}_{n}(\F)$ & $\mathrm{GL}_{n}(\mathbb{C})$ & $n$ & $\presuper{L}\GL_n(\mathbb{C})\xrightarrow{\sim} \presuper{L}\GL_{n}(\mathbb{C})$ \\
			$\mathrm{Sp}_{2n}(\F)$ & $\mathrm{SO}_{2n+1}(\mathbb{C})$ & $2n+1$ & $\LG(\mathbb{C})\hookrightarrow \presuper{L}\GL_{2n+1}(\mathbb{C})$ \\
			$\mathrm{SO}_{2n+1}(\F)$ & $\mathrm{Sp}_{2n}(\mathbb{C})$ & $2n$ & $\LG(\mathbb{C})\hookrightarrow \presuper{L}\GL_{2n}(\mathbb{C})$ \\
			$\mathrm{SO}_{2n}(\F)$& $\mathrm{SO}_{2n}(\mathbb{C})$ & $2n$ & $\LG(\mathbb{C})\hookrightarrow \presuper{L}\GL_{2n}(\mathbb{C})$ \\
			\hline
		\end{tabular}
	\end{table}
	\begin{example}\label{RSexample}
		\begin{enumerate}
			\item Let~$\P=\M\N$ be a standard maximal parabolic of~$\GL_{n+m}(\F)$ with Levi factor~$\M\simeq \GL_{n}(\F)\times \GL_m(\F)$, then~$r$ is irreducible and~$r\mid_{\widehat{\M}(\mathbb{C})}\simeq \phi_{\GL_n}\otimes (\phi_{\GL_m})^{\vee}$.
			\item Let~$\P=\M\N$ be a standard maximal parabolic of~$\Sp_{2(n+m)}(\F)$ with Levi factor~$\M\simeq \mathrm{Sp}_{2n}(\F)\times\GL_{m}(\F)$.  Then the representation $r\colon{^L}\M(\mathbb{C})\rightarrow \GL(^{L}\mathfrak{n})$ decomposes as
			\[r\simeq r_1\oplus r_2,\]
			with~$r_1\mid_{\widehat{\M}(\mathbb{C})}=\phi_{\SO_{2n+1}}\otimes \phi_{\GL_m}$, and~$r_2\mid_{\widehat{\M}(\mathbb{C})}=\Lambda^2(\phi_{\GL_m})$, and they are extended trivially to~$^{L}\M(\mathbb{C})$.
			\item Let~$\P=\M\N$ be a standard maximal parabolic of~$\SO_{2(n+m)+1}(\F)$ with Levi factor~$\M\simeq \mathrm{SO}_{2n+1}(\F)\times\GL_{m}(\F)$.  Suppose~$n>0$. Then the representation $r\colon{^L}\M(\mathbb{C})\rightarrow \GL(^{L}\mathfrak{n})$ decomposes as
			\[r\simeq r_1\oplus r_2,\]
			with~$r_1\mid_{\widehat{\M}(\mathbb{C})}=\phi_{\Sp_{2n}}\otimes \phi_{\GL_m}$, and~$r_2\mid_{\widehat{\M}(\mathbb{C})}=\mathrm{Sym}^2(\phi_{\GL_m})$, and they are extended trivially to~$^{L}\M(\mathbb{C})$.  If~$n=0$, then~$r$ is irreducible with~$r\mid_{\widehat{\M}(\mathbb{C})}=\mathrm{Sym}^2(\phi_{\GL_m})$.
		\end{enumerate}
	\end{example}
	
	\subsection{Split GSpin groups}
	Consider~$\G=\mathrm{GSpin}_{m}(\F),$ the split algebraic group with root datum as in \cite[§2.1]{ACS}.  If~$m=2n$ is even, then~$\widehat{\G}=\mathrm{GSO}_{2n}$, whereas if~$m=2n+1$ is odd~$\widehat{\G}=\mathrm{GSp}_{2n}$.  We extend the standard representation~$\phi_{\widehat{\G}}$ of their dual group to a morphism~$\phi_{\LG}\colon\LG(\mathbb{C})\hookrightarrow \GL_{N}(\mathbb{C}) \times \W_{\F}=\presuper{L}\GL_N(\mathbb{C})$.
	\begin{table}[h]
		\begin{tabular}{c c c c}
			\hline
			$\G$ & $\widehat{\G}(\mathbb{C})$ & $N = N(\G)$ & $\phi_{\LG}$ \\
			\hline
			$\mathrm{GSpin}_{2n+1}(\F)$ & $\mathrm{GSp}_{2n}(\mathbb{C})$ & $2n$ & $\LG(\mathbb{C})\hookrightarrow \presuper{L}\GL_{2n}(\mathbb{C})$ \\
			$\mathrm{GSpin}_{2n}(\F)$ & $\mathrm{GSO}_{2n}(\mathbb{C})$ & $2n$ & $\LG(\mathbb{C})\hookrightarrow \presuper{L}\GL_{2n}(\mathbb{C})$ \\
			\hline
		\end{tabular}
	\end{table}
	
	\begin{example}
		\begin{enumerate}
			\item Let~$\P$ be a maximal standard parabolic of~$\G=\mathrm{GSpin}_{2(n+m)+1}(\F)$ with standard Levi factor~$\M$.  Then~$\M\simeq \mathrm{GSpin}_{2n+1}(\F)\times \GL_{m}(\F)$.  Suppose~$n>0$.   Then,~$r\simeq r_1\oplus r_2$ with~$r_1=(\phi_{\mathrm{GSp}_{2n}}^{\vee})\otimes\phi_{\GL_m}$ and~$r_2$ the twisted (by the inverse of the similitude character) symmetric square of~$\phi_{\GL_m}$, and they are extended trivially to~$^{L}\M(\mathbb{C})$.
			\item Let~$\P$ be a maximal standard parabolic of~$\G=\mathrm{GSpin}_{2n}(\F)$ with standard Levi factor~$\M$.  Then~$\M\simeq \mathrm{GSpin}_{2s}(\F)\times \GL_{m}(\F)$ with $m+s=n$.  Suppose~$n>0$.  Then,~$r\simeq r_1\oplus r_2$ with~$r_1=\phi_{\mathrm{GSO}_{2n}}^{\vee}\otimes\phi_{\GL_m}$ and~$r_2$ the twisted exterior square of~$\phi_{\GL_m}$, , and they are extended trivially to~$^{L}\M(\mathbb{C})$.
		\end{enumerate}
	\end{example}

	\subsection{Quasi-split unitary groups}
	Let~$\E/\F$ be a quadratic extension with~$\Gal(\E/\F)=\langle\sigma \rangle$, and set
	\[\J'_n=\diag(1,-1,\ldots, (-1)^{n-1})\J_n.\]
	Let~$\H=\Res_{\E/\F}(\GL_n)(\F)=\GL_n(\E)$, and for~$h\in\H$ define
	\[\Sigma_n\cdot h=J'_n\cdot {^T}\overline{h}^{-1} \cdot (J'_n)^{-1},\]
	where~$\overline{h}=(\sigma(h_{i,j}))$, and set~$\G=\H^{\Sigma_n}$.  Then~$\G=\U_n(\E/\F)$ defines the quasi-split unitary group in~$n$-variables (i.e.,~$\G$ is a representative of the unique up to isomorphism quasi-split unitary group in $n$-variables with respect to~$\E/\F$).  We choose our Borel subgroup~$\B$ to be upper triangular matrices in~$\U_n(\E/\F)$, and our torus~$\T$ to be diagonal, and fix~$\Xi$ any pinning compatible with~$(\B,\T)$.  
	
	We have~$\widehat{\G}(\mathbb{C})=\GL_n(\mathbb{C})$,~$\LG(\mathbb{C})=\GL_n(\mathbb{C})\rtimes \W_{\F}$ where~$\W_\F$ acts through the finite quotient~$\Gal(\E/\F)$ with~$\sigma$ acting via
	\[\sigma\cdot g= J'_n\cdot {^T}g^{-1} \cdot (J'_n)^{-1}.\]
	We have~${^L}\H=(\GL_n\times\GL_n)(\mathbb{C})\rtimes \W_\F$ where the action of~$\W_\F$ factors through~$\Gal(\E/\F)$ with~$\sigma$ acting via
	\[\sigma\cdot(h_1,h_2)=(h_2,h_1).\]
	We then have the stable base change~$L$-morphism
	\begin{align*}
		\phi_{\LG}\colon\LG&\rightarrow {^L}\H\\
		g\rtimes w&\mapsto (g,\sigma\cdot g)\rtimes w.
	\end{align*}
	

	\subsection{Quasi-split even special orthogonal groups and GSpin}
	Let~$n\geqslant 2$, and~$\G=\SO_{2n,\E'/\F}(\F)$ be the quasi-split even special orthogonal group associated to the quadratic extension~$\E'/\F$.  We have~$\widehat{\G}=\SO_{2n}$. The action of~$\W_{\F}$ on~$\widehat{\G}$ factors through~$\Gal(\E'/\F)$, and the non-trivial element~$\sigma\in\Gal(\E'/\F)$ acts on~$\SO_{2n}$ by conjugation by an element of~$\mathrm{O}_{2n}$ of negative determinant.  We fix
	\[\phi_{^{L}\G}\colon^{L}\G=(\SO_{2n}(\mathbb{C})\rtimes \W_{\F})\rightarrow \GL_{2n}(\mathbb{C})\times \W_{\F},\]
	as in \cite[§7.1]{CKPSS2}.
	
	Let~$n\geqslant 2$, and~$\G=\mathrm{GSpin}_{2n,\E'/\F}$ be the quasi-split even GSpin group associated to the quadratic extension~$\E'/\F$, as in \cite[§2.1.3]{AsgariShahidiFunct}, following this source we have~$\LG(\mathbb{C})\simeq \mathrm{GSO}_{2n}(\mathbb{C})\rtimes \W_{\F}$, and an embedding~$\phi_{\LG}:\LG(\mathbb{C}) \hookrightarrow \presuper{L}\GL_{2n}(\mathbb{C})$.
	
	\begin{table}[h]
		\begin{tabular}{c c c c}
			\hline
			$\G$ & $\widehat{\G}(\mathbb{C})$ & $N = N(\G)$ & $\phi_{\LG}$ \\
			\hline
			$\mathrm{SO}_{2n,\E'/\F}(\F)$ & $\mathrm{SO}_{2n}(\mathbb{C})$ & $2n$ & $\LG(\mathbb{C})\hookrightarrow \presuper{L}\GL_{2n}(\mathbb{C})$ \\
			$\mathrm{GSpin}_{2n,\E'/\F}(\F)$ & $\mathrm{GSO}_{2n}(\mathbb{C})$ & $2n$ & $\LG(\mathbb{C})\hookrightarrow \presuper{L}\GL_{2n}(\mathbb{C})$ \\
			\hline
		\end{tabular}
	\end{table}
		
		\section{Local transfer in families} 
	For this section, we suppose that~$\F$ is a $p$-adic field,~i.e., a finite extension of~$\mathbb{Q}_p$, and~$\G$ is a classical group.  A semisimple parameter for~$\G$ is an~$L$-homomorphism~$\rho\colon\W_\F\rightarrow \LG(\mathbb{C})$, such that if we choose a finite form of the~$L$-group~$\widehat{\G}(\mathbb{C})\rtimes \Gal(\K/\F)$ then the morphism~$\overline{\rho}:\W_{\F}\rightarrow \widehat{\G}(\mathbb{C})\rtimes \Gal(\K/\F)$ has open kernel and image consisting of semisimple elements in~$\widehat{\G}(\mathbb{C})\rtimes \Gal(\K/\F)$ (the conditions, independent of the choice of~$\Gal(\K/\F)$).  Write~$\Phi_{ss}(\G)$ for the set of~$\widehat{\G}(\mathbb{C})$-conjugacy classes of semisimple parameters for~$\G$.  
	
	We fix a Borel-Torus pair~$(\B,\T)$ in~$\G$ which for split classical groups and quasi-split unitary groups we assume (under the bases chosen in Section \ref{notationclassical}) consists of the upper triangular and diagonal matrices in~$\G$.  We choose a pinning~$\Xi$ compatible with~$(\B,\T)$.  We let~$\psi:\F\rightarrow \R_0^\times$ denote our fixed character, and~$\mathfrak w=(\U,\Psi)$ be the Whittaker datum for~$\G$ defined by~$(\Xi,\psi)$.  For a standard Levi subgroup~$\M$ of a standard parabolic subgroup~$\P=\M\N$, we write~$\mathfrak w_{\M}$ for the Whittaker datum for~$\M$ defined by~$(\U\cap\M, \Psi\mid_{\U\cap\M})$.

	When~$\G$ is split (including~$\GL_n(\F)$), for the standard embedding~$\phi_{\LG}\colon\LG\rightarrow \GL_N(\mathbb{C})\times \W_{\F}$, we often denote this by~$\phi_{\widehat{\G}}$, as it is given by extending the standard embedding on dual groups trivially to~$\W_{\F}$.
	
	 We write $\Pi_{\mathfrak{w}}(\G)$ for the set of isomorphism classes of $\mathfrak w$-generic irreducible $\mathbb C$-representations of $\G$, and write~$ \Pi_{gen}(\GL_N(\E))$ for the set of isomorphism classes of generic irreducible~$\mathbb{C}$-representations of~$\GL_N(\E)$.
	 
	 We introduce some notation to allow us to state our main results for the unitary and non-unitary group cases uniformly.  If~$\G$ is not equal to a unitary group~$\U_n(\E/\F)$ then we set~$\E=\F$:
	 \begin{equation*}
	 \E=\begin{cases} \F&\text{ if $\G$ is not a unitary group};\\
	 \E&\text{ if $\G=\U_n(\E/\F)$.}
	 \end{cases}
	 \end{equation*}
	 We set~$\psi_{\E}:=\psi\circ\Tr_{\E/\F}$.
	 
	 For symplectic, and special orthogonal groups~$\G$, and~$\GL_m(\F)$, we obtain the standard gamma factors directly from the Langlands--Shahidi method: 
	 \[\gamma(X,\pi\otimes\rho, \phi_{\LG}\otimes\phi_{\GL_m},\psi)=\gamma(X,\pi\otimes\rho, r_1,\psi).\]
	 
	 In this section, for unitary group~$\G=\U_n(\E/\F)$, we define our gamma factors by:
	 \[\gamma(X,\pi\otimes \rho,  \phi_{\LG}\otimes\phi_{\GL_m},\psi)=\lambda(\E/\F,\psi)^{-mn}\gamma(X,\pi\otimes \rho^{\iota},  \overline{r_1},\psi),\]
	 where the gamma factor on the right hand side is the Langlands--Shahidi gamma factor defined in previous sections.
	 
	 For~$\GSpin$ groups, excluding~$\GSpin_2(\F)$, we set
	 	 \[\gamma(X,\pi\otimes \rho,  \phi_{\LG}\otimes\phi_{\GL_m},\psi):=\gamma(X,\pi\otimes \rho^{\iota},  \overline{r_1},\psi),\]
		 where the right hand side is the Langlands--Shahidi gamma factor defined in the previous sections.
		 
		 For~$\GSpin_2(\F)\simeq \F^\times\times\F^\times$, fixing our isomorphism as in \cite{AsgariShahidiFunct} we define
		 \[\gamma(X,(\chi_1\otimes\chi_2)\otimes \rho,  \phi_{\LG}\otimes\phi_{\GL_m},\psi):=\gamma(X,\chi_1\otimes \rho,\phi_{\GL_1}\otimes\phi_{\GL_m},\psi)\gamma(X,\chi_1^{-1}\chi_2\otimes\rho,\phi_{\GL_1}\otimes\phi_{\GL_m},\psi).\]

	 
	\subsection{Arthur and Mok's local transfer in the special case of generic representations}\label{ArtMoktransfer}
	%
	
	The local Langlands correspondences of Harris--Taylor, Henniart, Arthur \cite{Arthur2}, and Mok \cite{Mok}, each induce a bijection between isomorphism classes of generic representations (for a fixed Whittaker datum) of the $p$-adic group~$\G$ and conjugacy classes of semisimple parameters for~$\G$:
	\begin{thm}\label{LLforgenerics}
		Let~$\G$ be one of the quasi-split groups
		\[\GL_n(\F),\Sp_{2n}(\F),\SO_{2n+1}(\F),\U_n(\E/\F).\] 
		Then there is a unique bijection
		\begin{align*}
			\Pi_{\mathfrak w}(\G)&\rightarrow \Phi_{ss}(\G)\\
			\pi&\mapsto \phi_{\pi}\\
			\pi_{\phi}&\mapsfrom \phi
		\end{align*}
		satisfying, for all~$m$, and all generic representations $\rho$ of~$\GL_m(\E)$, 
		\[\gamma(X,\pi\otimes \rho,  \phi_{\LG}\otimes\phi_{\GL_m},\psi)=\gamma(X^{f(\E/\F)},(\phi_{\LG}\circ\phi_{\pi})\otimes (\phi_{\GL_m}\circ\phi_{\rho}),\psi_{\E}),\]
		where the gamma factor on the right is defined by the constructions of Artin and Deligne (cf., \cite[§7.5]{DHKMconj}).  
	\end{thm}
	
	\begin{proof}
		The local Langlands correspondence of Arthur and Mok, induces a semisimple correspondence\[\mathrm{LL}_{ss}\colon\Irr_{\mathbb{C}}(\G)\rightarrow \Phi_{ss}(\G)\] as explained in \cite[Corollary 7.16]{DHKMconj}.  It is compatible with twisted gamma factors by \cite[Appendix]{AGIKMS}.  Moreover, by \cite[Proposition 7.15]{DHKMconj}, there is a unique~$\mathfrak w$-generic representation of~$\G$ in each extended~$L$-packet, i.e., in each fibre of the map~$\mathrm{LL}_{ss}$, so we obtain a bijection~$\mathcal{L}_{\G}\colon\Pi_{\mathfrak{w}}(\G)\rightarrow \Phi_{ss}(\G)$ compatible with twisted gamma factors.   That the identity of gamma factors characterizes this correspondence follows from \cite[Proposition 7.7]{DHKMconj}.
		%
	\end{proof}

	Thus, for symplectic, odd special orthogonal, or quasi-split unitary, composing the semisimple parameter of a~$\mathfrak{w}$-generic representation with the fundamental representation of Section \ref{notationclassical}, the correspondences of Arthur and Mok define generic local transfer, and base change maps: 
	\begin{align*}
		\mathrm{t}_{\G,\mathfrak{w}}\colon\Pi_{\mathfrak{w}}(\G)&\rightarrow \Pi_{gen}(\GL_N(\E))\\
		\pi &\mapsto \pi_{\phi_{\LG}\circ \phi_{\pi}}
	\end{align*}
	where, as in the last section,~$\phi_{\LG}$ denotes the fundamental representation in the split case or the stable base change~$L$-morphism in the quasi-split case, and this transfer satisfies for all~$\pi\in\Pi_{\mathfrak w}$, and all irreducible generic representations~$\rho$ of~$\GL_r(\E)$
	\[\gamma(X,\pi\otimes\rho,\phi_{\LG}\otimes\phi_{\GL_r},\psi)=\gamma(X^{f(\E/\F)},\mathrm{t}_{\G,\mathfrak{w}}(\pi)\otimes \rho ,\phi_{\GL_N}\otimes \phi_{\GL_r},\psi_{\E}).\]
	
		\begin{remark}\label{remarkfulltransfer}
		This construction is compatible for cuspidal representations with the full $L$-parameter/local transfer morphism: Let~$\tau$ denote the~$\SL_2$-parameter of~$\pi$.  Then it follows from Arthur's and Mok's global construction that the~$\SL_2$-parameter of~$\mathrm{t}_{\G,\mathfrak{w}}(\pi)$ is equal (as a~$\GL_N(\mathbb{C})$-conjugacy class) to $\phi_{^{L}\G}\circ \tau$.  This follows as when~$\pi$ is cuspidal, this transfer is generic (cf., \cite[Theorem 1.1]{MS}) and as there is a unique generic representation in a semisimple~$L$-packet.
	\end{remark}

	\subsection{The local generic transfer for groups of classical type}\label{generictransfercomplex}
	For $\mathfrak{w}$-generic representations we also have the local transfer theorem of Asgari--Cogdell--Heiermann--Kim--Krishnamurthy--Shahidi--Piatetski-Shapiro \cite{CKPSS1,MR2127169,MR2149370,MR2219256,CKPSS2,AsgariShahidiFunct,ACS,Heiermanngeneric}:
	\begin{thm}\label{achkkspstransfer}
		Let~$\G$ be one of the quasi-split groups 
		\[\Sp_{2n}(\F),\SO_r(\F),\SO_{2n,\E'/\F}(\F),\U_n(\E/\F)\text{ (with $n>1$)},\mathrm{GSpin}_r(\F),\mathrm{GSpin}_{2n,\E'/\F}(\F).\] 
		There exists a unique map\footnote{Note for~$\SO_{2n,\E'/\F}(\F)$ and~$\mathrm{GSpin}_{2n,\E'/\F}(\F)$, the local transfer is to~$\Pi_{gen}(\GL_{2n}(\F))$.}
		\[\mathrm{t}_{\G,\mathfrak{w}}\colon\Pi_{\mathfrak{w}}(\G)\rightarrow \Pi_{gen}(\GL_N(\E)),\]
		satisfying, for all positive integers~$r$, and all supercuspidal representations $\rho$ of~$\GL_r(\E)$, 
		\begin{equation}
			\label{gammafactoreq}\gamma(X,\pi\otimes \rho, \phi_{\LG}\otimes \phi_{\GL_r},\psi)=\gamma(X^{f(\E/\F)},\mathrm{t}_{\G,\mathfrak{w}}(\pi)\otimes \rho,\phi_{\GL_N}\otimes \phi_{\GL_r},\psi_{\E}).\end{equation}
		Moreover, for~$\G\not\in\{ \U_1(\E/\F),\GSpin_1(\F)\}$, the morphism is characterized by the equalities \eqref{gammafactoreq} being satisfied for~$1\leqslant r\leqslant[N/2]$, and, for~$\Sp_{2n}(\F),\SO_{2n+1}(\F)$, and~$\U_n(\E/\F)$, this map agrees with the local generic transfer defined by Arthur and Mok (Section \ref{ArtMoktransfer}).  
	\end{thm}
	\begin{proof}
	We suppose~$N>1$, i.e.,~$\G\not\in\{ \U_1(\E/\F),\GSpin_1(\F)\}$, for which the theorem can be verified directly. 	Suppose there exists a map~$\mathrm{t}_{\G,\mathfrak{w}}$ satisfying the given equalities of gamma factors.  Then as twisted gamma factors for~$\GL_N(\E)$ characterize supercuspidal support by \cite{Henniart}, and as there is a unique generic representation with a given supercuspidal support, so~$\mathrm{t}_{\G,\mathfrak{w}}$ is uniquely defined by this property (and while the definition of \cite{ACS} of local transfer also states the analogous equality of~$L$-functions, this is not needed to characterize the map).  We immediately also deduce the final statement on equality with the transfer maps defined by Arthur and Mok's local Langlands (cf., Theorem \ref{LLforgenerics}) for~$\Sp_{2n},\SO_{2n+1},\U_{n,\E/\F}$.  The local strong converse theorem of \cite{MR3862069} gives the final statement on the subcollection of gamma factors characterizing the map.
		
		If~$\G$ is a split classical group, then the construction of~$\mathrm{t}_{\G,\mathfrak{w}}$ is \cite[Proposition 7.5, Lemma 7.2]{CKPSS1}, which is built on first considering the supercuspidal case.  In all cases, as brought together in \cite[Theorem 3.2]{ACS}, if~$\pi\in\Pi_{\mathfrak w}$ is supercuspidal, then there exists a generic irreducible representation~$\mathrm{t}_{\G,\mathfrak{w}}(\pi)\in \Pi_{\Xi,\psi}(\GL_N(\E))$ satisfying
		\[\gamma(X,\pi\otimes \rho, \phi_{\LG}\otimes \phi_{\GL_r},\psi)=\gamma(X^{f(\E/\F)},\mathrm{t}_{\G,\mathfrak{w}}(\pi)\otimes \rho,\phi_{\GL_N}\otimes \phi_{\GL_r},\psi_{\E}).\]
		for all irreducible unitary supercuspidal representations~$\rho$ of~$\GL_r(\E)$, for all~$1\leqslant r\leqslant N-1$.  While they state this equality of gamma factors for twists up to~$N-1$ (as obtaining the twisting relation up to~$N-1$ suffices to characterize the local transfer), the transfer is compatible with twisted gamma factors of all~$\GL_r(\E)$.  Indeed, the same argument applies in this case as given in ibid., cf.,~\cite[Theorem 4.2]{Heiermanngeneric} for split~$\GSpin$-groups.
		
		Let~$\rho$ denote any cuspidal representation of~$\GL_r(\E)$, and~$\nu=|~|_{\E}\circ\det$, then there exists~$t\in\mathbb{C}$ such that~$\rho'=\rho\otimes\nu^t$ is unitary.  
		Moreover, by Shahidi \cite[Equation (3.12)]{ShahidiAnnals}, 
		\begin{align*}
			\gamma(X^{f(\E/\F)},\mathrm{t}_{\G,\mathfrak{w}}(\pi)\otimes \rho,\phi_{\GL_N}\otimes \phi_{\GL_r},\psi_{\E})&=\gamma(q^{tf(\E/\F)}X^{f(\E/\F)},\mathrm{t}_{\G,\mathfrak{w}}(\pi)\otimes \rho',\phi_{\GL_N}\otimes \phi_{\GL_r},\psi_{\E})\\
			\gamma(X,\pi\otimes \rho, \phi_{\LG}\otimes \phi_{\GL_r},\psi)&=\gamma(q^{t}X,\pi\otimes \rho', \phi_{\LG}\otimes \phi_{\GL_r},\psi);
		\end{align*}
		and thus we deduce that
		\[\gamma(X,\pi\otimes \rho, \phi_{\LG}\otimes \phi_{\GL_r},\psi)=\gamma(X^{f(\E/\F)},\mathrm{t}_{\mathfrak{w}}(\pi)\otimes \rho,\phi_{\GL_N}\otimes \phi_{\GL_r},\psi_{\E}).\]
		for all irreducible supercuspidal representations~$\rho$ of~$\GL_r(\E)$, for all~$1\leqslant r\leqslant N-1$. 
				
		In general, let~$\pi\in \Pi_{\mathfrak w}$ and~$(\M,\tau)$ in the supercuspidal support of~$\pi$ and~$\P=\M\N$ such that~$\pi$ is a quotient of~$i_{\M,\P}^{\G}(\tau)$.  Then~$\M\simeq \G_0\times \prod_{i=1}^k\GL_{n_i}(\E)$ where~$\G_0$ is of the same type as~$\G$, and we write~$\tau\simeq \tau_0\otimes \bigotimes_{i=1}^k\tau_i$ with respect to this decomposition.  Suppose first~$\G_0\neq \GSpin_1(\F)$, then we define the candidate transfer~$\mathrm{t}_{\G,\mathfrak{w}}(\pi)$ to be the unique generic subquotient of
		\[\tau_1\times  \cdots \times \tau_k \times \mathrm{t}_{\G_0,\mathfrak{w}}(\tau_0)\times \widetilde{\tau_k} \times \cdots \times \widetilde{\tau_1}\]
		where
		\begin{equation}\label{widetildedef}
			\widetilde{\rho_i}=\begin{cases}
				\rho_i^{\vee}&\text{if }\G=\Sp_{2n}(\F),~\SO_{r}(\F),~\SO_{2n,\E'/\F}(\F);\\
				\rho_i^{\vee,\sigma}&\text{if~}\G=\U_n(\E/\F);\\
				\rho_i^{\vee}\otimes \chi  &\text{if~}\G=\mathrm{GSpin}_r(\F),\mathrm{GSpin}_{2n,\E'/\F}(\F),
			\end{cases}
		\end{equation}
		where
		\begin{equation*}
		\chi=\begin{cases}
		\tau_0&\text{ if }\G=\GSpin_1(\F);\\
		\chi_1&\text{ if }\G=\GSpin_2(\F),~\tau_0=\chi_1\otimes\chi_2;\\
		(\omega_{\pi}\mid_{\F^\times}\circ\det)&\text{ otherwise.}
		\end{cases}
		\end{equation*}
		By the first paragraph of this proof, it suffices to prove Equation \eqref{gammafactoreq}.
		
		By the multiplicativity relation for~$\gamma$-factors of~$\GL_{N}(\E)$, we have 
		\[\gamma(X^{f(\E/\F)},\mathrm{t}_{\G,\mathfrak{w}}(\pi)\otimes \rho,\psi_{\E})=\gamma(X^{f(\E/\F)},\mathrm{t}_{\G_0,\mathfrak{w}}(\tau_0)\otimes \rho,\psi_{\E})\prod_{i=1}^k\gamma(X^{f(\E/\F)},\tau_i\otimes\rho,\psi_{\E})\gamma(X^{f(\E/\F)},\widetilde{\tau_i}\otimes \rho,\psi_{\E}).\]
		By multiplicativity of~$\gamma$-factors \cite[Theorem 3.5]{Shahidimultiplicativity}, we have
		\begin{align*}\gamma(X,\pi\otimes \rho,\psi)=\gamma(X,\tau_0\otimes \rho,\psi)\prod_{i=1}^k\gamma(X^{f(\E/\F)},\tau_i\otimes\rho,\psi_{\E})\gamma(X^{f(\E/\F)},\widetilde{\tau_i}\otimes \rho,\psi_{\E}), \end{align*}
		using \cite[Lemma 4.4]{ACS}, (for split~$\mathrm{GSpin}$, see \cite[Example 3.5]{Heiermanngeneric}).  Hence we obtain Equation \ref{gammafactoreq}.
		
		For~$\G_0=\GSpin_1(\F)$, the candidate transfer is defined to be the unique generic subquotient of
		\[\tau_1\times \cdots \times \tau_k \times  \widetilde{\tau_k} \times \cdots \times \widetilde{\tau_1}\]
		where~$\widetilde{\tau_i}=\tau_i^{\vee}\otimes (\tau_0\circ\det)$, and the theorem follows similarly using the multiplicativity identity~$\gamma(X,\pi\times\rho,\psi)=\prod_{i=1}^k\gamma(X^{f(\E/\F)},\tau_i\times \rho,\psi)\gamma(X^{f(\E/\F)},\widetilde{\tau_i}\times\rho,\psi).$
	\end{proof}

	\begin{lemma}\label{lemmainjectivity}
		Let~$\G=\SO_{2n+1}(\F), \Sp_{2n}(\F),\U_n(\E/\F)$ then the local transfer
		\[\mathrm{t}_{\G,\mathfrak{w}}\colon\Pi_{\mathfrak w}(\G)\rightarrow \Pi_{gen}(\GL_N(\E))\]
		is injective.
	\end{lemma}
	
	\begin{proof}
		For~$\G=\SO_{2n+1}(\F), \Sp_{2n}(\F),\U_n(\E/\F)$, this follows as Theorem \ref{LLforgenerics} is a bijection, and \cite[Theorem 8.1]{GGP}, which shows that the (conjugacy class of) the transfer of an $L$-parameter of an irreducible representation determines the conjugacy class of the~$L$-parameter.  And as there is a unique generic representation in an extended~$L$-packet, we deduce that~$\mathrm{t}_{\G,\mathfrak{w}}$ is injective.  
	\end{proof}
	For even special orthogonal groups, it follows from \cite[Theorem 8.1]{GGP} the conjugacy class of the transfer only determines the conjugacy class in the full orthogonal group of the parameter.
	
	\begin{remark}\label{remarklocalconverse}
		Suppose again~$\G\neq \U_1(\E/\F)$.  Then the injectivity of the local transfer in these cases, is equivalent, from the statement of Theorem~\ref{achkkspstransfer}, to the following statement known as a \emph{local converse theorem}:
		``Let~$\pi,\pi'\in\Pi_{\mathfrak w}$.  Suppose that, for all $1\leqslant r\leqslant N-1$, and all irreducible supercuspidal representations $\rho$ of~$\GL_{r}(\E)$, we have
		\[\gamma(X,\pi\otimes \rho, \phi_{\LG}\otimes \phi_{\GL_r},\psi)=\gamma(X,\pi'\otimes \rho, \phi_{\LG}\otimes \phi_{\GL_r},\psi),\]
		then~$\pi\simeq \pi'$.''   Such statements have been proved directly in some cases and are considered in \cite{matringelocalconverse} from an L-parameters perspective.  Combining these theorems with the multiplicativity relation of gamma factors (when the theorems are only written for supercuspidal generic representations) gives an alternative proof of the lemma.  Injectivity can fail for quasi-split even special orthogonal groups~$\G=\SO_{2n}(\F),\SO_{2n,\E'/\F}(\F)$.  Indeed, fix an element~$\epsilon\in\mathrm{O}_{2n}(\F)/\SO_{2n}(\F)$ (respectively,~$\mathrm{O}_{2n,\E'/\F}(\F)/\SO_{2n,\E'/\F}(\F)$) preserving~$\mathfrak{w}$,~$\mathrm{t}_{\G,\mathfrak{w}}(\pi)\simeq \mathrm{t}_{\G,\mathfrak{w}}(\pi^{\epsilon})$ and there exist~$\pi\not\simeq \pi^{\epsilon}$.
	\end{remark}

	\subsection{Properties of local transfer}
	For this section, we let~$\G$ be one of the quasi-split groups
	\[\Sp_{2n}(\F),\SO_r(\F),\SO_{2n,\E'/\F}(\F),\U_n(\E/\F),\mathrm{GSpin}_r(\F),\mathrm{GSpin}_{2n,\E'/\F}(\F).\]
	
	\begin{lemma}\label{galoisinvariance}
		Let~$\sigma\colon\mathbb{C}\rightarrow \mathbb{C}$ be a field automorphism fixing~$\mathbb{Q}(\sqrt{q},\mu_{p^\infty})$ if~$\G$ is split, and~$\mathbb{Q}(\sqrt{q},\mu_{4p^\infty})$ if~$\G$ is non-split.  Then, for~$\pi\in\Pi_{\mathfrak w}(\G)$,
		\[\mathrm{t}_{\G,\mathfrak{w}}(\pi^{\sigma})\simeq \mathrm{t}_{\G,\mathfrak{w}}(\pi)^{\sigma}.\]
	\end{lemma}
	
	\begin{proof}
		For~$\Sp_{2n}(\F)$,~$\SO_{2n+1}(\F)$, and~$\U_n(\E/\F)$, this follows from \cite[Propositions 7.11, 7.12]{DHKMconj}.  For all cases, except~$\G=\U_1(\E/\F),\GSpin_1(\F)$, we give a proof using Langlands--Shahidi gamma factors, inspired by arguments of Bushnell-Henniart summarised in \cite[§35.2]{BHBook}:  We extend~$\sigma$ to~$\mathbb{C}(X)$, by letting it act trivially on~$X$.
		By Theorem \ref{mainthm} (2), we have, for all~$1\leqslant r\leqslant N-1$ and every generic representation~$\rho$ of~$\GL_{r}(\E)$ that
		\[\gamma(X,\pi^{\sigma}\otimes\rho^{\sigma},\phi_{\LG}\otimes \phi_{\GL_r},\psi^{\sigma})=\sigma(\gamma(X,\pi\otimes\rho,\phi_{\LG}\otimes \phi_{\GL_r},\psi)).\]
		 Hence by Theorem \ref{achkkspstransfer} we have
		\[\gamma(X^{f(\E/\F)},\mathrm{t}_{\G,\mathfrak{w}}(\pi^{\sigma})\otimes \rho^{\sigma},\phi_{\GL_N}\otimes \phi_{\GL_r},\psi^{\sigma})=\sigma(\gamma(X^{f(\E/\F)},\mathrm{t}_{\G,\mathfrak{w}}(\pi)\otimes \rho,\phi_{\GL_N}\otimes \phi_{\GL_r},\psi_{\E})).\]
		By the equality with Rankin-Selberg gamma factors and \cite[3.2 Theorem]{BH00}, we have
		\[\sigma(\gamma(X^{f(\E/\F)},\mathrm{t}_{\G,\mathfrak{w}}(\pi)\otimes \rho,\phi_{\GL_N}\otimes \phi_{\GL_r},\psi_{\E}))=\gamma(X^{f(\E/\F)},\mathrm{t}_{\G,\mathfrak{w}}(\pi)^{\sigma}\otimes \rho^{\sigma},\phi_{\GL_N}\otimes \phi_{\GL_r},\psi^{\sigma}_{\E}).\]
		Hence
		\[\gamma(X^{f(\E/\F)},\mathrm{t}_{\G,\mathfrak{w}}(\pi^{\sigma})\otimes \rho^{\sigma},\phi_{\GL_N}\otimes \phi_{\GL_r},\psi^{\sigma}_{\E})=\gamma(X^{f(\E/\F)},\mathrm{t}_{\G,\mathfrak{w}}(\pi)^{\sigma}\otimes \rho^{\sigma},\phi_{\GL_N}\otimes \phi_{\GL_r},\psi^{\sigma}_{\E}),\]
		and applying Henniart's local converse theorem \cite{Henniart} we find~$\mathrm{t}_{\G,\mathfrak{w}}(\pi^{\sigma})\simeq \mathrm{t}_{\G,\mathfrak{w}}(\pi)^{\sigma}$. 
		\end{proof}
	
	We extend local transfer to Levi subgroups of~$\G$ in the following way:  
	
	Suppose first that~$\G\not\in\{\GSpin_r(\F),\GSpin_{2n,\E'/\F}(\F)\}$, then we have fixed our models for~$\G$, and with respect to these models, we choose a basis so that~$\M$ is standard with block diagonal decomposition
	\[\left\{\left(\begin{smallmatrix} a_1 \\ &\ddots\\&&a_k\\&&&a_0\\&&&&a_{-k}\\&&&&&\ddots\\&&&&&&a_{-1}\end{smallmatrix}\right):a_0\in\G_0,~a_i\in\GL_{n_i}(\E),~a_{-i}=\widetilde{a_i}\right\}\simeq \prod_{i=1}^k\GL_{n_i}(\E)\times \G_0\]
	where~$\G_0$ is a classical group of the same type as~$\G$, and
	\begin{equation*}
	\widetilde{a_i}=\begin{cases} \presuper{t}a_i^{-1} &\text{if~}\G\in\{\SO_r(\F),\Sp_{2n}(\F),\SO_{2n,\E'/\F}(\F)\};\\
	\presuper{t}{\overline{a_i}}^{-1} &\text{if~}\G=\U_n(\E/\F).
\end{cases}
	\end{equation*}
For~$\G\in\{\GSpin_{r}(\F),\GSpin_{2n,\E'/\F}(\F)\}$, as explained in \cite{MR2514488}, we have a decomposition
\[\M\simeq \prod_{i=1}^k\GL_{n_i}(\F)\times \G_0\]
where~$\G_0$ has the same type as~$\G$.

Note that when~$\G= \Sp_{2n}(\F)$ and~$\G_0$ is the trivial group, due to dimension reasons we need to define:
\[\mathrm{t}_{\G_0,\mathfrak{w}_{\M}\mid\G_0}:\Pi_{\mathfrak{w}_{\M}\mid\G_0}(\G_0)\rightarrow \Pi_{gen}(\GL_1(\F)),\]
by taking the trivial representation of the trivial group~$\Sp_0(\F)$ to the trivial representation of~$\GL_1(\F)$. 
	
	Suppose first that~$\G_0$ is non-trivial and~$\G_0\neq \GSpin_1(\F)$.  In these cases, we write~$\M'$ for the block diagonal Levi subgroup of~$\GL_{\N}(\E)$ with decomposition
	\[\prod_{i=1}^{k}\GL_{n_i}(\E)\times \GL_{N_0}(\E)\times \prod_{i=1}^{k}\GL_{n_i}(\E).\]
	and, using the notation of Equation \ref{widetildedef}, define
	\begin{align*}
		\mathrm{t}_{\M,\mathfrak{w}_{\M}}:\Pi_{\mathfrak{w}_{\M}}(\M)&\rightarrow \Pi_{gen}\left(\M'\right)\\
		\bigotimes_{i=1}^k \pi_i\otimes \pi_0&\mapsto \bigotimes_{i=1}^k \pi_i\otimes \mathrm{t}_{\G_0,\mathfrak{w}_{\M}\mid_{\G_0}}(\pi_0)\otimes \bigotimes_{i=1}^{k} \widetilde{\pi_i}
	\end{align*}
	where~$\widetilde{\pi_i}$ is defined as in \eqref{widetildedef}.   If~$\G_0$ is one of the trivial groups~$\SO_0(\F),\SO_1(\F), \U_0(\F)$, we make the same definition omitting~$\GL_{N_0}(\E)$ from our transfer, so the transfer~$\mathrm{t}_{\M,\mathfrak{w}_\M}$ is defined completely by the above formula.  For~$\G_0=\GSpin_1(\F)$, we write
	\[\prod_{i=1}^{k}\GL_{n_i}(\E)\times  \prod_{i=1}^{k}\GL_{n_i}(\E),\]
	and define
			\begin{align*}
		\mathrm{t}_{\M,\mathfrak{w}_{\M}}:\Pi_{\mathfrak{w}_{\M}}(\M)&\rightarrow \Pi_{gen}\left(\M'\right)\\
		\bigotimes_{i=1}^k \pi_i\otimes \pi_0&\mapsto \bigotimes_{i=1}^k \pi_i\otimes  \bigotimes_{i=1}^{k} \widetilde{\pi_i}.
	\end{align*}

	From the construction in the proof of Theorem \ref{achkkspstransfer}, we have the following compatibility of local transfer with supercuspidal support:
	\begin{lemma}\label{scscompatibility}
		Let~$\pi\in\Pi_{\mathfrak w}$, and~$(\M,\pi_{\M})$ be in the supercuspidal support of~$\pi$, and write~$\pi_{\M}\simeq \bigotimes_{i=1}^k \pi_i\otimes \pi_0$ as above.  Then~$\pi_0$ is $\mathfrak w_\M\mid\G_0$-generic, and any supercuspidal support of~$(\M',\mathrm{t}_{\M,\mathfrak{w}_{\M}}(\pi_{\M}))$ is in the supercuspidal support of~$\mathrm{t}_{\G,\mathfrak{w}}(\pi)$.  In particular,~$\mathrm{t}_{\G,\mathfrak{w}}$ defines a map
		\[\mathrm{t}_{\G,\mathfrak{w}}^{scs}:\Pi^{scs}_{\mathfrak w}(\G)\rightarrow \Pi^{scs}(\GL_N(\E)),\]
		from the set of supercuspidal supports of~$\mathfrak w$-generic representations to the set of supercuspidal supports for~$\GL_{N}(\E)$
	\end{lemma}
	
	%
	%
	
	Let~$\chi:\M\rightarrow\mathbb{C}^\times$ be an unramified character of~$\M$, we can write
	\[\chi\simeq \bigotimes_{i=1}^{k} \chi_{i}\otimes \chi_0\]
	where~$\chi_i$ is an unramified character of~$\GL_{n_i}(\E)$ and~$\chi_0$ is an unramified character of~$\G_0$ (and~$\chi_0$ is trivial unless~$\G_0\in\{\SO_{2}(\F),\mathrm{GSpin}_r(\F),\mathrm{GSpin}_{2n,\E'/\F}(\F)\}$).  
	
	Let~$\pi_{\M}\simeq \bigotimes_{i=1}^k\pi_i\otimes \pi_0\in \Pi_{\mathfrak{w}_{\M}}(\M)$ be supercuspidal.  And choose~$(\M'',\mathrm{t}_{\M,\mathfrak{w}_{\M}}(\pi_{\M})_{scs})$ in the supercuspidal support of~$(\M',\mathrm{t}_{\M,\mathfrak{w}_{\M}}(\pi_{\M}))$.  By above, we can choose this of the form:
	\begin{align*}
		\M''&\simeq \prod_{i=1}^{k}\GL_{n_i}(\E)\times \prod_{j=1}^d\GL_{N_{0,j}}(\E)\times \prod_{i=1}^{k}\GL_{n_i}(\E)\\
		\mathrm{t}_{\M,\mathfrak{w}_{\M}}(\pi_{\M})_{scs}&\simeq  \bigotimes_{i=1}^k \pi_i\otimes ( \mathrm{t}_{\G_0,\mathfrak{w}_{\M}\mid_{\G_0}}(\pi_0))_{scs}\otimes \bigotimes_{i=1}^{k} \widetilde{\pi_i}\end{align*}
	where~$(\M''_0,( \mathrm{t}_{\G_0,\mathfrak{w}_{\M}\mid_{\G_0}}(\pi_0))_{scs})$ with~$\M''_0\simeq \prod_{j=1}^d\GL_{N_{0,j}}(\E)$ is a supercuspidal support of~$ \mathrm{t}_{\G_0,\mathfrak{w}_{\M}\mid_{\G_0}}(\pi_0)$.

	If~$\G_0\not\in\{\mathrm{GSpin}_r(\F),\mathrm{GSpin}_{2n,\E'/\F}(\F)\}$, we write~$\Upsilon:\M''\rightarrow \mathbb{C}^\times$ for the unramified character
	\[\Upsilon= \bigotimes_{i=1}^{k} \chi_{i}\otimes \Upsilon_0\otimes  \bigotimes_{i=1}^{k} \chi_{i}^{-1}.\]
	where
	\begin{equation*}
		\Upsilon_0=\begin{cases} \mathbf{1}&\text{if }\G_0\not\in\{\SO_{2}(\F),\mathrm{GSpin}_r(\F),\mathrm{GSpin}_{2n,\E'/\F}(\F)\};\\
			\chi_0\otimes \chi^{-1}_0&\text{if }\G_0=\SO_2(\F)\simeq \F^\times\xrightarrow{\chi_0}\mathbb{C}^\times.
		\end{cases}
	\end{equation*}
	
		If~$\G_0\in\{\mathrm{GSpin}_r(\F),\mathrm{GSpin}_{2n,\E'/\F}(\F)\}\backslash\{\GSpin_1(\F),\GSpin_2(\F)\}$, then we let~$\xi:\G_0\rightarrow \F^\times$ be fixed as in \cite[§3]{MR3273456}, and write every unramified character~$\G_0\rightarrow \mathbb{C}^\times$ as~$\chi_0\circ\xi$ for~$\chi_0:\F^\times\rightarrow \mathbb{C}^\times$ unramified.  Then we define
	\[\Upsilon= \bigotimes_{i=1}^{k} \chi_{i}\otimes \Upsilon_0\otimes  \bigotimes_{i=1}^{k} \chi_{i}^{-1}(\chi_0\circ\det)^2.\]
	with~$\Upsilon_0=\chi_0\circ\det\mid_{\M_0''}$.
	
	For~$\G_0=\GSpin_2(\F)$, we set~\[\Upsilon_0=\chi_{0,1}\otimes \chi_{0,1}^{-1}\chi_{0,2},\quad \Upsilon= \bigotimes_{i=1}^{k} \chi_{i}\otimes \Upsilon_0\otimes  \bigotimes_{i=1}^{k} \chi_{i}^{-1}(\chi_{0,2}\circ\det).\]
 if~$\chi_{0,1}\otimes \chi_{0,2}:\GSpin_2(\F)\rightarrow \mathbb{C}^\times$ is an unramified character, where we identify~$\GSpin_2(\F)$ with~$\F^\times\times\F^\times$ as in \cite{AsgariShahidiFunct}.
	
	If~$\G\neq \GSpin_1(\F)$, but~$\G_0=\GSpin_1(\F)\simeq \F^\times$, then we set
	\[\Upsilon=\bigotimes_{i=1}^{k} \chi_{i}\otimes   \bigotimes_{i=1}^{k} \chi_{i}^{-1}(\chi_0\circ\det).\]

	\begin{lemma}\label{unramifiedtwisting}
		Let~$[\M,\pi_{\M}]$ be a supercuspidal support of a~$\mathfrak w$-generic representation of~$\G$, and let
		$\chi:\M\rightarrow\mathbb{C}^\times$ be an unramified character of~$\M$.  
		Choose~$(\M'',\mathrm{t}_{\M,\mathfrak{w}_{\M}}(\pi_{\M})_{scs})$ in the supercuspidal support~$\mathrm{t}_{\G,\mathfrak{w}}^{scs}(\pi_{\M})$ as above.  Then~$(\M'',\mathrm{t}_{\M,\mathfrak{w}_{\M}}(\pi_{\M})_{scs}\otimes\Upsilon)$ is in~$\mathrm{t}_{\G,\mathfrak{w}}^{scs}(\pi_{\M}\otimes \chi)$.
	\end{lemma}

	\begin{proof}
		The only compatibility which is not clear from the explicit construction is the classical piece of the Levi~$\G_0$, when this has non-compact centre, i.e., for~$\G_0\in \{\SO_{2}(\F),\mathrm{GSpin}_r(\F),\mathrm{GSpin}_{2n,\E'/\F}(\F)\}$. 
		
		The local transfer of a character~$\chi:\SO_2(\F)\rightarrow \mathbb{C}^\times$ is
		\[\mathrm{t}_{\SO_2(\F)}(\chi)=\mathrm{Gen}(\chi\times \chi^{-1}),\]
		where~$\mathrm{Gen}(\chi\times \chi^{-1})$ denotes the unique generic subquotient of~$\chi\times \chi^{-1}$, and there is a unique (trivial) Whittaker datum for~$\SO_2(\F)$ so we omit it from our notation.  The formula for~$\SO_2(\F)$ follows from this.
		
			Similarly, by \cite[§4.1]{AsgariShahidiFunct}, for~$\GSpin_2(\F)\simeq \F^\times\times\F^\times$, the generic local transfer of~$\chi_1\otimes\chi_2$ 
			\[\mathrm{t}_{\GSpin_2(\F)}(\chi_1\otimes\chi_2)=\mathrm{Gen}(\chi_1\times \chi_1^{-1}\chi_2),\]
			where again there is a unique (trivial) Whittaker datum for for~$\GSpin_2(\F)$, the formula for~$\GSpin_2(\F)$ follows.

		For~$\G_0\in\{\mathrm{GSpin}_r(\F),\mathrm{GSpin}_{2n,\E'/\F}(\F)\}\backslash \{\GSpin_1(\F),\GSpin_2(\F)\}$, as twisted gamma factors characterize supercuspidal support for~$\GL_{N}(\E)$, we are reduced to showing the identity
		\[\gamma(X^{f(\E/\F)},( \mathrm{t}_{\G_0,\mathfrak{w}_{\M}\mid_{\G_0}}(\pi)\otimes (\chi_0\circ\det))\otimes \rho,\phi_{\GL_N}\otimes \phi_{\GL_r},\psi)=\gamma(X,(\pi\otimes (\chi_0\circ\xi))\otimes \rho, \phi_{\LG_0}\otimes \phi_{\GL_r},\psi).\]
		which follows from standard properties of unramified twisting shifting the variable~$X$ (for the RHS for split~$\GSpin$ see also \cite[Theorem 3.5]{CaiFriedbergKaplanAnnals}, where an analogous formula is given for gamma factors defined via the doubling method; for the LHS, see for example \cite[Theorem 2.2(vi)]{HenniartLomeliUniquenessofRSproducts}).
	\end{proof}

	\subsection{Interpolating local transfer in families in the banal setting}
	For~$\G$ a $p$-adic group, we let~$\mathrm{N}_{\G}$ denote the product of all \emph{non-banal primes} for~$\G$, that is the product of the primes which divide the pro-order of a compact open subgroup of~$\G$.  We let~$\M_{\G}=\mathrm{N}_{\G}\mathrm{N}_{\GL_N(\E)}$.  
	
	For all~$m\geqslant 1$, we let~$\mathfrak{W}_m=(\U_m,\Psi_m)$ denote the Whittaker datum of~$\GL_m(\E)$ where~$\U_m$ denotes the upper triangular unipotent matrices in~$\GL_m(\E)$, and~$\Psi_m(u)= (\psi_{\E})(\sum_{i=1}^{m-1}u_{i,i+1})$.  In the special case of~$\GL_N(\E)$, we write~$\mathfrak{W}=\mathfrak{W}_N$.  Then we have a natural bijection between the (isomorphism classes of) generic~$\mathbb{C}$-representations for~$\GL_{N}(\E)$ of depth~$t$ and the~$\mathbb{C}$-points of~$\E^{\GL_N(\E)}_{\mathfrak{W},t}$ by Corollary \ref{corbijection}.
	
	By \cite[Theorems 4.22, 5.21]{DHKMconj}, for~$\mathfrak{t}$ (respectively~$\mathfrak{t}'$) a finite sum of primitive idempotents of~$\mathfrak{Z}_{\R_0}(\G)$ (respectively $\mathfrak{Z}_{\R_0}(\GL_N(\E))$), we have decompositions
	\begin{align*}
		\E^{\GL_{N}(\E)}_{\mathfrak{W},\mathfrak{t}'}\otimes\overline{\mathbb{Z}}[1/\N_{\GL_N(\E)}]&= \prod_{[\M',\tau']_{\GL_N(\E)}\in\mathfrak{B}(\GL_N(\E))}e_{[\M',\tau']_{\GL_N(\E)}}(\E^{\GL_{N}(\E)}_{\mathfrak{W},\mathfrak{t}'}\otimes\overline{\mathbb{Z}}[1/\N_{\GL_N(\E)}])
		\\
		\E^{\G}_{\mathfrak{w},\mathfrak{t}}\otimes\overline{\mathbb{Z}}[1/\N_{\G}]&= \prod_{[\M,\tau]_{\G}\in\mathfrak{B}_{\mathfrak{w}}(\G)}e_{{[\M,\tau]_{\G}}}(\E^{\G}_{\mathfrak{w},\mathfrak{t}}\otimes\overline{\mathbb{Z}}[1/\N_{\G}])\end{align*}
	where
	\begin{itemize}
		\item $\mathfrak{B}(\GL_N(\E))$ denotes the set of inertial classes of supercuspidal support for~$\GL_N(\E)$;
		\item $\mathfrak{B}_{\mathfrak{w}}(\G)$ denotes the set of~$\mathfrak{w}$-generic inertial classes of supercuspidal support for~$\G$, i.e., the subset of~$\mathfrak{B}(\G)$ consisting of supercuspidal supports whose parabolic induction contains a~$\mathfrak{w}$-generic representation.
		\item $e_{[\M,\tau]_{\G}}$ is the primitive idempotent of the Bernstein centre corresponding to~$[\M,\tau]_{\G}$;
		\item $e_{[\M',\tau']_{\GL_{N}(\E)}}$ is the primitive idempotent of the Bernstein centre corresponding to~$[\M',\tau']_{\GL_{N}(\E)}$.
	\end{itemize}
	Moreover, assuming~$e_{\mathfrak{t}}e_{[\M,\tau]_{\G}}\neq 0$ and~$e_{\mathfrak{t}'}e_{[\M',\tau']_{\GL_N(\E)}}\neq 0$, the choice of representatives for the inertial classes, fixes isomorphisms:
	\begin{align*}
		e_{[\M',\tau']_{\GL_N(\E)}}(\E^{\GL_{N}(\E)}_{\mathfrak{W},\mathfrak{t}'}\otimes\overline{\mathbb{Z}}[1/\N_{\GL_N(\E)}])& \simeq (\overline{\mathbb{Z}}[1/\N_{\GL_N(\E)}][\M'/\M'^{\circ}]^{\H_{[\M',\tau']}})^{\W_{[\M',\tau']}};
		\\
		e_{{[\M,\tau]_{\G}}}(\E^{\G}_{\mathfrak{w},\mathfrak{t}}\otimes\overline{\mathbb{Z}}[1/\N_{\G}])
		&\simeq(\overline{\mathbb{Z}}[1/\N_{\G}][\M/\M^{\circ}]^{\H_{[\M,\tau]}})^{\W_{[\M,\tau]}}.
	\end{align*}
	where
	\begin{itemize}
		\item $\H_{[\M,\tau]}$ is the finite group of unramified characters of~$\M$ fixing~$\tau$;
		\item $\H_{[\M',\tau']}$ is the finite group of unramified characters of~$\M'$ fixing~$\tau'$;
		\item $\W_{[\M,\tau]}=\{g\in\G\mid \M^g=\M,~(\M,\tau^g)\in[\M,\tau]_{\M}\}/\M$;
		\item $\W_{[\M',\tau']}=\{g\in\GL_{N}(\E)\mid \M'^g=\M',~(\M',\tau'^g)\in[\M',\tau']_{\M'}\}/\M'$.
	\end{itemize}
	
	We let~$\mathfrak{B}_{\sim}(\GL_N(\E))$ denote the subset of~$\mathfrak{B}(\GL_N(\E))$ consisting of inertial supercuspidal supports~$[\M',\tau']_{\GL_N(\E)}$ which have a representative~$(\M',\tau')$ which is (twisted) self-dual in the following sense:
	\begin{enumerate}
		\item $\M'$ is~$\GL_{\N}(\E)$-conjugate to a block diagonal Levi~
		\[\GL_{N_1}(\E)\times \cdots \times \GL_{N_k}(\E) \times (\GL_{N_{0,1}}(\E)\times \cdots \times \GL_{N_{0,d}}(\E))\times  \GL_{N_k}(\E) \times \cdots \times \GL_{N_1}(\E).\]
		\item Under this isomorphism, we have
		\begin{align*}
			\tau'&\simeq \tau_1'\otimes \cdots \otimes \tau_k'\otimes (\tau'_{0,1}\otimes \cdots \otimes \tau'_{0,d})\otimes \widetilde{\tau'_k}\otimes \cdots \otimes \widetilde{\tau_1'},\\
			\widetilde{\tau_{0,i}'}&\simeq \tau_{0,i}'\text{ for~$1\leqslant i\leqslant d$},\end{align*}
		where~$\widetilde{~}$ for~$\G\not\in\{\mathrm{GSpin}_r(\F),\mathrm{GSpin}_{2n,\E'/\F}(\F)\}$ is the involution defined by Equation \eqref{widetildedef},  
		and for~$\G\in\{\mathrm{GSpin}_r(\F),\mathrm{GSpin}_{2n,\E'/\F}(\F)\}$, there exists a character~$\chi:\F^\times\rightarrow \mathbb{C}^\times$ such that $\widetilde{~}=(~)^{\vee}\otimes (\chi\circ\det)$.
	\end{enumerate}
	
	\begin{thm}\label{banalmainthm}
		Let~$\G$ be one of the quasi-split groups
		\[\Sp_{2n}(\F),\SO_r(\F),\SO_{2n,\E'/\F}(\F),\U_n(\E/\F),\mathrm{GSpin}_r(\F),\mathrm{GSpin}_{2n,\E'/\F}(\F).\]  
		\begin{enumerate}
		\item For every finite sum of primitive idempotents~$e_{\mathfrak{t}}\in\mathfrak{Z}_{\R_0}(\G)$,  there exists a unique unital morphism of~$\R_0[1/\M_{\G}]$-algebras
		\[\mathbf{t}_{\G,\mathfrak{W},\mathfrak{w},\mathfrak{t}}\colon \E^{\GL_{N}(\E)}_{\mathfrak{W}}[1/\mathrm{M}_{\G}]\rightarrow \E^{\G}_{\mathfrak{w},\mathfrak{t}}[1/\mathrm{M}_{\G}]\]
		such that if~$\pi$ is a~$\mathfrak{w}$-generic irreducible~$\mathbb{C}$-representation of~$\G$ of type~$\mathfrak{t}$ then~$ f_{\pi,0}\circ \mathbf{t}_{\G,\mathfrak{W},\mathfrak{w}}=f_{\mathrm{t}_{\G,\mathfrak{w}}(\pi),0}$.  
		\item  There exists a finite sum of primitive idempotents~$e_\mathfrak{t'}\in\mathfrak{Z}_{\R_0}(\GL_N(\E))$ such that
		\[\mathbf{t}_{\G,\mathfrak{W},\mathfrak{w},\mathfrak{t}}\circ (1-e_{\mathfrak{t}'})=0.\]
		\item	 The extension of~$\mathbf{t}_{\G,\mathfrak{W},\mathfrak{w},\mathfrak{t}}$ to~$\overline{\mathbb{Z}}[1/\M_{\G}]$ satisfies
		\[\mathbf{t}_{\G,\mathfrak{W},\mathfrak{w},\mathfrak{t}}\circ (1-\sum_{\substack{\mathfrak{s}\in\mathfrak{B}_{\sim}(\GL_N(\E))\\e_{\mathfrak{t}'}e_{\mathfrak{s}}\neq 0}}e_{\mathfrak{s}})=0.\]
		\end{enumerate}
	\end{thm}
	
	\begin{proof}
		If~$\G$ is symplectic, odd special orthogonal, or quasi-split unitary then this follows from \cite[Theorem 8.2, §7]{DHKMconj}, via the banal Local Langlands in Families correspondence.  
		
		In general, to define a morphism, it suffices to work inertial class, by inertial class.  Let~$[\M,\tau]_{\G}\in\mathfrak{B}_{\mathfrak{w}}(\G)$, and choose a representative~$(\M,\tau)$ with~$\tau$ defined over a number field~$\K/\mathbb{Q}$, with finite order central character, and~$\mathfrak{o}_{\K}$-integral (possible by \cite[Proposition 4.8]{DHKMconj}).  Write~$\M\simeq \G_0\times \prod_{i=1}^k\GL_{n_i}(\E)$ where~$\G_0$ is of the same type as~$\G$, and write~$\tau\simeq \tau_0\otimes \bigotimes_{i=1}^k\tau_i$.

		We consider~$[\M',\tau']_{\GL_N(\E)}=\mathbf{t}_{\G,\mathfrak{w}}^{scs}([\M,\tau]_{\G})$, as in the last section this has the form:
		\begin{align*}
			\M'&\simeq \prod_{i=1}^{k}\GL_{n_i}(\E)\times \prod_{j=1}^d\GL_{N_{0,j}}(\E)\times \prod_{i=1}^{k}\GL_{n_i}(\E)\\
			\tau'&\simeq  \bigotimes_{i=1}^k \tau_i\otimes ( \mathbf{t}_{\G_0,\mathfrak{w}_{\M}\mid_{\G_0}}(\tau_0))_{sc}\otimes \bigotimes_{i=1}^{k} \widetilde{\tau_i}\end{align*}
		where~$ (\prod_{j=1}^d\GL_{N_{0,j}}(\E), \mathbf{t}_{\G_0,\mathfrak{w}_{\M}\mid_{\G_0}}(\tau_0)_{sc})$ lies in the supercuspidal support of~$ \mathbf{t}_{\G_0,\mathfrak{w}_{\M}\mid_{\G_0}}(\tau_0)$.  As~$\tau$ is~$\mathfrak{o}_{\K}$-integral, so are~$\tau_0,\tau_i,\widetilde{\tau}_i$.  
		From \cite[Theorem 3.2]{ACS}, the central character of~$( \mathbf{t}_{\G_0,\mathfrak{w}_{\M}\mid_{\G_0}}(\tau_0))_{sc}$ is of finite order (for~$\K'/\K$ finite), and it follows from \cite[Proposition 4.8]{DHKMconj} that~$\tau'$ is~$\mathfrak{o}_{\K'}$-integral.
		
		We claim there is a unique morphism (where~$(\mathfrak{t}',\mathfrak{t})$ are chosen to support the inertial classes)
		\[e_{[\M',\tau']_{\GL_N(\E)}}(\E^{\GL_{N}(\E)}_{\mathfrak{W},\mathfrak{t}'}\otimes\overline{\mathbb{Z}}[1/\M_{\G}])\rightarrow e_{{[\M,\tau]_{\G}}}(\E^{\G}_{\mathfrak{w},\mathfrak{t}}\otimes\overline{\mathbb{Z}}[1/\M_{\G}])
		\]
		interpolating the local transfer on these direct factors.    From the choice of~$\tau,\tau'$, this morphism is identified with a morphism
		\[(\overline{\mathbb{Z}}[1/\M_{\G}][\M'/\M'^{\circ}]^{\H_{[\M',\tau']}})^{\W_{[\M',\tau']}}
		\rightarrow (\overline{\mathbb{Z}}[1/\M_{\G}][\M/\M^{\circ}]^{\H_{[\M,\tau]}})^{\W_{[\M,\tau]}}.\]
		By the choices of isomorphism above, we have
		\begin{align*}
			\overline{\mathbb{Z}}[1/\M_{\G}][\M'/\M'^{\circ}]&\simeq \overline{\mathbb{Z}}[1/\M_{\G}][X_1^{\pm1},\ldots ,X_k^{\pm 1}, X_{0,1}^{\pm 1},\ldots ,X_{0,d}^{\pm 1},X_{-k}^{\pm 1},\ldots,X_{-1}^{\pm 1}]\\
			\overline{\mathbb{Z}}[1/\M_{\G}][\M/\M^{\circ}]&\simeq\begin{cases}
				\overline{\mathbb{Z}}[1/\M_{\G}][Y_1^{\pm1},\ldots ,Y_k^{\pm 1}] &\text{if }\G_0\not\in\{\SO_2(\F),\GSpin_r(\F),\GSpin_{2n,\E'/\F}(\F)\}\\
				\overline{\mathbb{Z}}[1/\M_{\G}][Y_0^{\pm1},Y_1^{\pm1},\ldots ,Y_k^{\pm 1}] &\text{if }\G_0\in \{\SO_2(\F),\GSpin_r(\F)~(r\neq1,2),\GSpin_{2n,\E'/\F}(\F)\}\\
				\overline{\mathbb{Z}}[1/\M_{\G}][Y_{0,1}^{\pm1},Y_{0,2}^{\pm1},Y_1^{\pm1},\ldots ,Y_k^{\pm 1}]  &\text{if }\G_0=\GSpin_2(\F).
			\end{cases}
		\end{align*}
		Then we define a morphism
		\begin{align*}
			\Phi:\overline{\mathbb{Z}}[1/\M_{\G}][\M'/\M'^{\circ}]
			&\rightarrow \overline{\mathbb{Z}}[1/\M_{\G}][\M/\M^{\circ}]\\
			X_i&\mapsto Y_i\\
			X_{0,j}&\mapsto 1\text{ if }\G_0\not\in\{\SO_2(\F),\GSpin_r(\F),\GSpin_{2n,\E'/\F}(\F)\}\\
			X_{-i}&\mapsto \begin{cases}
				Y_i^{-1}&\text{if }\G_0\not \in\{\GSpin_r(\F),\GSpin_{2n,\E'/\F}(\F)\}\\
				Y_i^{-1}Y_0^2&\text{if }\G_0\in\{\GSpin_r(\F),\GSpin_{2n,\E'/\F}(\F)~(n\neq 1)\}\\
				Y_i^{-1}Y_0&\text{if }\G_0=\GSpin_1(\F)\text{ where~$d=0$};
			\end{cases}
		\end{align*}
		where if~$\G_0=\SO_2(\F)$, then~$d=2$, and we define
		\[X_{0,1}\mapsto Y_0,\quad X_{0,2}\mapsto Y_0^{-1}.\]
		
		If~$\G_0\in\{\GSpin_r(\F),\GSpin_{2n,\E'/\F}(\F)\}$, except in the cases~$r=1,2$, we set then
		\[X_{0,j}\mapsto Y_0;\]
		whereas in the exceptional cases:
		\begin{enumerate}
		\item if~$\G_0=\GSpin_2(\F)$, then~$d=2$, and we define
		\[X_i\mapsto Y_i,\quad X_{0,1}\mapsto Y_{0,1} \quad X_{0,2}\mapsto  Y_{0,1}^{-1}Y_{0,2},\quad X_{-i}\mapsto Y_i^{-1}Y_{0,2};\]
		\item if~$\G_0=\GSpin_1(\F)\simeq \F^\times$, then $d=0$, but we set
		\[X_{-i}\mapsto Y_i^{-1}Y_0.\]
		\item if~$\G_0=\GSpin_{2,\E'/\F}(\F)$ and~$\E'/\F$ is unramified, then we introduce~$Z$ such that~$Z^2=Y_0$.  Then 
		\[X_i\mapsto Y_i,\quad X_{0,j}\mapsto Z,\quad X_{-i}\mapsto Y_i^{-1}Z^2.\]
		Here~$Z$ is an auxiliary variable, when we consider the morphism from~$(\overline{\mathbb{Z}}[1/\M_{\G}][\M'/\M'^{\circ}]^{\H_{[\M',\tau']}})^{\W_{[\M',\tau']}}$, the image lies in the subring~$\overline{\mathbb{Z}}[1/\M_{\G}][Y_0^{\pm1},Y_1^{\pm1},\ldots ,Y_k^{\pm 1}] $.
		 \item If~$\G_0=\GSpin_{2,\E'/\F}(\F)$ and~$\E'/\F$ is ramified, then we have the same formula as the general~$\GSpin$ case:
		 		\[X_i\mapsto Y_i,\quad X_{0,j}\mapsto Y_0,\quad X_{-i}\mapsto Y_i^{-1}Y_0^2.\]

		\end{enumerate}
		
		In all cases, this defines a morphism
		\[(\overline{\mathbb{Z}}[1/\M_{\G}][\M'/\M'^{\circ}]^{\H_{[\M',\tau']}})^{\W_{[\M',\tau']}}
		\rightarrow \overline{\mathbb{Z}}[1/\M_{\G}][\M/\M^{\circ}].\]
		which we need to show has image in~$( \overline{\mathbb{Z}}[1/\M_{\G}][\M/\M^{\circ}]^{\H_{[\M,\tau]}})^{\W_{[\M,\tau]}}$, and is compatible with local transfer over~$\mathbb{C}$ via the above identifications.  Notice, as~$ \overline{\mathbb{Z}}[1/\M_{\G}][\M/\M^{\circ}]$ is torsion-free over~$ \overline{\mathbb{Z}}$, it suffices to show that after extending scalars to~$\mathbb{C}$ the image of the morphism
		\[\Phi_{\mathbb{C}}:(\mathbb{C}[\M'/\M'^{\circ}]^{\H_{[\M',\tau']}})^{\W_{[\M',\tau']}}
		\rightarrow \mathbb{C}[\M/\M^{\circ}],\]
		lies in~$(\mathbb{C}[\M/\M^{\circ}]^{\H_{[\M,\tau]}})^{\W_{[\M,\tau]}}$.
		For any unramified characters~$\chi_1,\chi_2:\M\rightarrow \mathbb{C}^\times$, the morphism
		\[(\mathbb{C}[\M'/\M'^{\circ}]^{\H_{[\M',\tau']}})^{\W_{[\M',\tau']}}\xrightarrow{\Phi_{\mathbb{C}}} \mathbb{C}[\M/\M^{\circ}]\xrightarrow{\chi_i}\mathbb{C}\]
		corresponds to~$ \Psi_i=\Upsilon(\chi_i)$ by Lemma \ref{unramifiedtwisting}.  
		
	Now notice that, by Lemma \ref{scscompatibility}, if the supercuspidal supports~$(\M,\tau\otimes\chi_1),(\M,\tau\otimes\chi_2)$ are~$\G$-conjugate, i.e., if~$\chi_1\mid_{(\mathbb{C}[\M/\M^{\circ}]^{\H_{[\M,\tau]}})^{\W_{[\M,\tau]}}}=\chi_2\mid_{(\mathbb{C}[\M/\M^{\circ}]^{\H_{[\M,\tau]}})^{\W_{[\M,\tau]}}}$, then so too are~$(\M',\tau'\otimes \Psi_i)$, i.e.,
		\[\Psi_1\mid_{(\mathbb{C}[\M'/\M'^{\circ}]^{\H_{[\M',\tau']}})^{\W_{[\M',\tau']}}}=\Psi_2\mid_{(\mathbb{C}[\M'/\M'^{\circ}]^{\H_{[\M',\tau']}})^{\W_{[\M',\tau']}}},\]
		and hence for~$\alpha\in(\mathbb{C}[\M'/\M'^{\circ}]^{\H_{[\M',\tau']}})^{\W_{[\M',\tau']}}$, we have
		\[\chi_1(\Phi_{\mathbb{C}}(\alpha))=\chi_2(\Phi_{\mathbb{C}}(\alpha)).\]
		
	Applying this with~$\chi_2=\chi_1\chi$ for~$\chi\in\H_{[\M,\tau]}$ and~$\chi_1:\M\rightarrow\mathbb{C}^\times$ any unramified character, we deduce that the image of~$\Phi_{\mathbb{C}}$ lies in~$\mathbb{C}[\M/\M^{\circ}]^{\H_{[\M,\tau]}}$.  On the other hand, if~$g\in \W_{[\M,\tau]}$ then~$(\M^g,(\tau\otimes\chi_1)^g)=(\M,\tau\otimes\chi_1^g\chi_g)$ for some~$\chi_g:\M\rightarrow \mathbb{C}^\times$ unramified, which is unique modulo~$\H_{[\M,\tau]}$. We deduce that~$\chi_1$ and~$\chi_1^g\chi_g$ agree on the image of~$\Phi_{\mathbb{C}}$, which shows that the image lies in~$(\mathbb{C}[\M/\M^{\circ}]^{\H_{[\M,\tau]}})^{\W_{[\M,\tau]}}$.  We send all other Bernstein components (i.e., all components not contributing to local transfer) in~$\mathfrak{B}(\GL_N(\E))$ to zero
		

		Running over all inertial classes in~$\mathfrak{B}_{\mathfrak{w}}(\G)$ of type~$\mathfrak{t}$, defines a morphism
		\[\mathbf{t}_{\G,\mathfrak{W},\mathfrak{w}}:\E^{\GL_{N}(\E)}_{\mathfrak{W},\mathfrak{s}}\otimes\overline{\mathbb{Z}}[1/\M_{\G}]\rightarrow \E^{\G}_{\mathfrak{w},\mathfrak{t}}\otimes\overline{\mathbb{Z}}[1/\M_{\G}]
		,\]
		where~$\mathfrak{s}$ is the necessarily finite sum of depths with an inertial class with non-zero image in~$\E^{\G}_{\mathfrak{w},t}\otimes\overline{\mathbb{Z}}[1/\M_{\G}]$.	Note also that the support is contained in~$\mathfrak{B}_{\sim}(\GL_N(\E))$ by \cite[Theorem 3.2]{ACS} (for the supercuspidal transfer, (twisted) self-dual part), and from the construction in the proof of Theorem \ref{achkkspstransfer}.  
		
						Galois invariance, Lemma \ref{galoisinvariance}, defines a morphism
				\[\mathbf{t}_{\G,\mathfrak{W},\mathfrak{w},\mathfrak{t}}:\E^{\GL_{N}(\E)}_{\mathfrak{W}}[1/\M_{\G}]\rightarrow \E^{\G}_{\mathfrak{w},\mathfrak{t}}[1/\M_{\G}],\]
				by \cite[Corollary 3.6]{DHKMconj}.
			\end{proof}
	
	\subsection{Interpolating local transfer in families}
	We can now restate the 
	Helm--Moss' descent theorem using Langlands--Shahidi gamma factors by Corollary \ref{Mossequality}.  We set~$k=\Fl$.
	
	\begin{thm}[{\cite[Theorem 5.1]{HM}}]\label{converse}
		Let~$\R,\R'$ be Noetherian~$W(k)[\sqrt{q}]$-algebras with~$\R'\subseteq \R$.  Suppose that~$\R$ is a finitely generated module over~$\R'$, and let~$\rho$ be a co-Whittaker~$\R$-representation of~$\GL_{N}(\E)$ of type~$\mathfrak{t}$ with~$N\geqslant 2$.  Suppose that for all primitive idempotents~$e_{\mathfrak{s}} \in\mathfrak{Z}_{W(k)}(\GL_{N-1}(\E))$, the coefficients of
		\begin{align*}&\gamma(X^{-f(\E/\F)},\rho\otimes (\Gamma^{\GL_{N-1}(\E)}_{\mathfrak W_{N-1},\mathfrak{s},W(k)})^{\iota},\phi_{\GL_N}\otimes\phi_{\GL_{N-1}}^{\vee},\psi_{\E}),~\text{and}\\ &\gamma(X^{f(\E/\F)},\rho^{\iota}\otimes  (\Gamma^{\GL_{N-1}(\E)}_{\mathfrak W_{N-1},\mathfrak{s},W(k)})^{\iota},\phi_{\GL_N}\otimes\phi_{\GL_{N-1}}^{\vee},\psi^{-1}\circ\Tr_{\E/\F})\end{align*}
		lie in~$\R'\otimes  \E^{\GL_{N-1}(\E)}_{\mathfrak{W}_{N-1},\mathfrak{s},W(k)}$.  Then the map~$f_\rho\colon \E^{\GL_N(\E)}_{\mathfrak{W},\mathfrak{t}}\rightarrow \R$ 
		takes values in~$\R'$.
	\end{thm}
	
	%

	\begin{thm}\label{mainthm2}
		Let~$\G$ be one of the quasi-split groups
		\[\Sp_{2n}(\F),\SO_r(\F),\SO_{2n,\E'/\F}(\F),\U_n(\E/\F),\mathrm{GSpin}_r(\F),\mathrm{GSpin}_{2n,\E'/\F}(\F).\]  
		\begin{enumerate}
			\item There exists a unique morphism of~$\R_0$-algebras
			\[\mathbf{t}_{\G,\mathfrak{W},\mathfrak{w}}\colon\E^{\GL_{N}(\E)}_{\mathfrak{W}}\rightarrow \E^{\G}_{\mathfrak{w}}\]
			such that if~$\pi$ is a generic irreducible~$\mathbb{C}$-representation of~$\G$ then~$(\Gamma_{\mathfrak{W}}^{\GL_N(\E)})_{ f_{\pi,0}\circ\mathbf{t}_{\G,\mathfrak{w}}}$ dominates~$\mathrm{t}_{\G,\mathfrak{w}}(\pi)$.
			\item Let~$e_\mathfrak{t}\in \mathfrak{Z}_{\R_0}(\G)$ be a finite sum of primitive idempotents. Then there exists a finite sum of primitive idempotents~$e_{\mathfrak{s}}\in\mathfrak{Z}_{\R_0}(\GL_{N}(\E))$ such that, for any primitive idempotent~$e_{\mathfrak{t}'}\in\mathfrak{Z}_{\R_0}(\GL_{N}(\E))$ satisfying ~$(1-e_{\mathfrak{s}})e_{\mathfrak{t}'}\neq0$, we have
			\[e_{\mathfrak{t}}\mathbf{t}_{\G,\mathfrak{W},\mathfrak{w}}(\E_{\mathfrak{W},\mathfrak{t}'}^{\GL_N(\E)})= 0.\]
			
			\item Let~$\A=\E^{\G}_{\mathfrak{w},\mathfrak{t}}\otimes_{\R_0} \E^{\GL_r(\E)}_{\mathfrak{W}_r,\mathfrak{t}'}$.  Then the morphism~$\mathbf{t}_{\G,\mathfrak{w}}$, after scalar extension, satisfies the equalities in~$\mathcal{S}_{X,\A}^{-1}\A[X^{\pm1}]$ of universal gamma factors:
			\[\gamma(X,\Gamma^{\G}_{\mathfrak w,\mathfrak{t}}\otimes \Gamma^{\GL_{r}(\E)}_{\mathfrak W_r,\mathfrak{t}'},\phi_{\LG}\otimes\phi_{\GL_r},\psi)=(\mathbf{t}_{\G,\mathfrak{W},\mathfrak{w}}\otimes 1)\gamma(X^{f(\E/\F)},\Gamma^{\GL_N(\E)}_{\mathfrak W,t(\mathfrak{t})}\otimes \Gamma^{\GL_{r}(\E)}_{\mathfrak W_r,\mathfrak{t}'},\phi_{\GL_{N}}\otimes\phi_{\GL_r},\psi_{\E}),\]
			for all~$1\leqslant r\leqslant N-1$, and all primitive idempotents~$e_{\mathfrak{t}}\in\mathfrak{Z}_{\R_0}(\G),e_{\mathfrak{t}'}\in\mathfrak{Z}_{\R_0}(\GL_r(\E))$, 
			and~$t(\mathfrak{t})\in \mathfrak{Z}_{\R_0}(\GL_N(\E))$ is the finite sum of idempotents contained in the support of~$e_{\mathfrak{t}} \circ \mathbf{t}_{\G,\mathfrak{W},\mathfrak{w}}$.  Moreover, if~$\G\not\in\{ \U_1(\E/\F),\GSpin_1(\F)\}$, then these equalities characterize the morphism uniquely.
		\end{enumerate}
	\end{thm}

	\begin{proof}
	The cases of~$\GSpin_1(\F)$ and~$\U_1(\E/\F)$ are simple and can be considered directly using local Langlands for tori, so we suppose~$N>1$ for the remainder.  Fix a primitive idempotent~$e_\mathfrak{t}\in\mathfrak{Z}_{\R_0}(\G)$.  From Theorem \ref{banalmainthm}, we have a unital morphism of~$\R_0[1/\mathrm{M}_{\G}]$-algebras
		\[ \mathbf{t}_{\G,\mathfrak{W},\mathfrak{w},\mathfrak{t}}\colon\E^{\GL_{N}(\E)}_{\mathfrak{W}}[1/\mathrm{M}_{\G}]\rightarrow \E^{\G}_{\mathfrak{w},\mathfrak{t}}[1/\mathrm{M}_{\G}]\]
		Moreover, by construction of the morphism, there exists a finite sum of primitive idempotents~$\mathfrak{t}'$ such that the image of~$\E^{\GL_N(\E)}_{\mathfrak{W},\mathfrak{t}'}[1/\mathrm{M}_{\G}]$ is contained in~$\E^{\G}_{\mathfrak{w},\mathfrak{t}}[1/\mathrm{M}_{\G}]$. 
				
		Let~$p\neq\ell\mid\M_{\G}$ be prime and set~$k=\Fl$.  Suppose first~$\ell\neq2$, then for each maximal ideal above~$\ell$ we choose an embedding~$\R_0\hookrightarrow W(k)$ and extending scalars locally over~$W(k)$. We obtain from~$\mathbf{t}_{\G,\mathfrak{W},\mathfrak{w},\mathfrak{t}}$ a morphism
		\[\mathbf{t}_{\G,\ell}\colon\E^{\GL_N(\E)}_{\mathfrak{W},W(k),\mathfrak{t}'}[1/\ell]\rightarrow \E^{\G}_{\mathfrak{w},W(k),\mathfrak{t}}[1/\ell]\]
		and if~$p\neq 2$, choosing for each maximal ideal above~$2$ an embedding~$\R_0\hookrightarrow W(
		\overline{\mathbb{F}_2})[\mu_4,\sqrt{q}]$ (or in the split case, we can choose an embedding~$\R_0\hookrightarrow W(\overline{\mathbb{F}_2})[\sqrt{q}]$ and proceed without adjoining fourth roots of unity) we obtain from~$\mathbf{t}_{\G,\mathfrak{W},\mathfrak{w},\mathfrak{t}}$ a morphism   
		\[\mathbf{t}_{\G,2}\colon\E^{\GL_N(\E)}_{\mathfrak{W},W(\overline{\mathbb{F}_2})[\mu_4,\sqrt{q}],\mathfrak{t}'}[1/2]\rightarrow \E^{\G}_{\mathfrak{w},W(\overline{\mathbb{F}_2} )[\mu_4,\sqrt{q}],\mathfrak{t}}[1/2].\]
		Noting when~$\ell\neq 2$ we have~$W(k)=W(k)[\mu_4,\sqrt{q}]$.  We set~$W(k)'=W(k)[\mu_4,\sqrt{q}]$.  It suffices by \cite[Lemma A.3]{DHKMconj}, to show that 
		 for all~$ \ell\mid \mathrm{M}_{\G}$, and each embedding above, that we have
		\begin{equation}\label{WTS}\tag{$\star$}
		\mathbf{t}_{\G,\ell}(\E^{\GL_N(\E)}_{\mathfrak{W},W(k)',\mathfrak{t}'})\subseteq \E^{\G}_{\mathfrak{w},W(k)',\mathfrak{t}}\end{equation}

		We first observe that this image lies in the normalization:
		
		\begin{lemma}
			We have
			\[\mathbf{t}_{\G,\ell}(\E^{\GL_N(\E)}_{\mathfrak{W},W(k)',\mathfrak{t}'})\subseteq (\E^{\G}_{\mathfrak{w},W(k)',\mathfrak{t}})^{\mathrm{norm}}\]
			where~$(\E^{\G}_{\mathfrak{w},W(k)',\mathfrak{t}})^{\mathrm{norm}}$ denotes the normalization of~$\E^{\G}_{\mathfrak{w},W(k)',\mathfrak{t}}$.
		\end{lemma}
		\begin{proof}
			By \cite[Proof of Theorem 9.5]{HelmCurtis}: it suffices to show that if~$A$ is a discrete valuation ring and~$W(k)'$-algebra with field of fractions~$K$ of characteristic zero, and
			\[ \E^{\G}_{\mathfrak{w},W(k)',\mathfrak{t}}\rightarrow A\]
			is a morphism of~$W(k)'$-algebras defining a morphism~$\phi_K\colon\E^{\G}_{\mathfrak{w},W(k)',\mathfrak{t}}[1/\ell]\rightarrow K$, and~$x\in\E^{\GL_N(\E)}_{\mathfrak{W},W(k),\mathfrak{t}'}$, then~$\phi_K\circ \mathbf{t}_{\G,\ell}(x)\in A$.  
			
			To prove this criterion, from Lemma \ref{cowhittakerDVR} \eqref{cowhittakerDVR1}, it suffices to show that the local transfer preserves~$A$-integrality of~$K$-representations.   By extending scalars, we can reduce to proving that local transfer preserves integrality when~$K$ is algebraically closed and~$A$ is its valuation ring. 			Let~$\pi$ be an irreducible~$K$-representation of~$\G$, and let~$(\M,\rho)$ be in the supercuspidal support of~$\pi$.  As~$K$ is algebraically closed,~$\rho$ has a central character~$\omega_{\rho}:\Z(\M)\rightarrow K^\times$, and,~$\pi$ is~$A$-integral if and only if~$\omega_{\rho}$ takes values in~$A$.  This follows from \cite[Corollary 1.6]{DHKMfiniteness} and \cite[II 4.12]{Vig96} (which while they are stated in a different setup, with~$\Ql$ in place of~$K$ and~$\Zl$ in place of~$A$).  If~$\omega_{\rho}$ takes values in~$A$, then the central character of the cuspidal support of its transfer is integral by the description of transfer in \cite[Theorem 3.2]{ACS}, and we are done.
		\end{proof}
		
		Let~$\R'= \E^{\G}_{\mathfrak{w},W(k)',\mathfrak{t}}$, and~$\R=\R'[\mathbf{t}_{\G,\ell}(\E^{\GL_N(\E)}_{\mathfrak{W},W(k)',\mathfrak{t}'})]$, from Lemma~\ref{Ebasicfacts} the rings~$\R,\R'$ are Noetherian.  Then choosing finitely many~$W(k)'$-algebra generators for~$\E^{\GL_N(\E)}_{\mathfrak{W},W(k)',\mathfrak{t}'}$, as the image of each of these elements lies in the normalization of~$\R'$ it is integral over~$\R'$, and we find that~$\R$ is a finitely generated module over~$\R'$.  We are thus in the realms of Theorem \ref{converse}.  Let~
	\[\rho=\Gamma^{\GL_N(\E)}_{\mathfrak{W},W(k)',\mathfrak{t}'}\otimes_{\E^{\GL_N(\E)}_{\mathfrak{W},W(k)',\mathfrak{t}'},\mathbf{t}_{\G,\ell}}\R.\]
	It is a co-Whittaker~$\R$-representation of~$\GL_N(\E)$, by Lemma \ref{cwlemma}.  Moreover,
	\[f_{\rho}\colon \E^{\GL_N(\E)}_{\mathfrak{W},W(k)',\mathfrak{t}}\rightarrow \R\]
	is precisely~$\mathbf{t}_{\G,\ell}$.  In particular, Theorem \ref{converse} shows, that to prove Equation~\eqref{WTS} holds, it suffices to prove that for all primitive idempotents~$e_{\mathfrak{s}}\in\mathfrak{Z}_{W(k)}(\GL_{N-1}(\E))$ that the coefficients 
				\begin{align*}&\gamma(X^{-f(\E/\F)},\rho\otimes (\Gamma^{\GL_{N-1}(\E)}_{\mathfrak W_{N-1},W(k)',\mathfrak{s}})^{\iota},\phi_{\GL_N}\otimes\phi_{\GL_{N-1}}^{\vee},\psi_{\E})_a,~\text{and}\\ &\gamma(X^{f(\E/\F)},\rho^{\iota}\otimes  (\Gamma^{\GL_{N-1}(\E)}_{\mathfrak W_{N-1},W(k)',\mathfrak{s}})^{\iota},\phi_{\GL_N}\otimes\phi_{\GL_{N-1}}^{\vee},\psi^{-1}\circ\Tr_{\E/\F})_a\end{align*}
lie in~$\R'\otimes \E_{\mathfrak{W}_{N-1},W(k)',\mathfrak{s}}^{\GL_{N-1}(\E)}.$
		
		%
		%
		As Langlands--Shahidi gamma factors are compatible with local transfer, we have
		\[\gamma(X^{-f(\E/\F)},\rho\otimes (\Gamma^{\GL_{N-1}(\E)}_{\mathfrak W_{N-1},\mathfrak{s}})^{\iota},\phi_{\GL_N}\otimes\phi_{\GL_{N-1}}^{\vee},\psi_{\E})= \gamma(X^{-1},\Gamma^{\G}_{\mathfrak{w},\mathfrak{t},W(k)'}\otimes \Gamma^{\GL_{N-1}(\E)}_{\mathfrak W_{N-1},\mathfrak{s}},\phi_{\LG}\otimes\phi_{\GL_{N-1}},\psi)\]
	Hence, we have an equality of coefficients in~$\R\otimes  \E_{\mathfrak{W}_{N-1},W(k)',\mathfrak{s}}^{\GL_{N-1}(\E)}$, however by our main result, Theorem \ref{mainthm}, these coefficients lie in~$\R'\otimes \E_{\mathfrak{W}_{N-1},W(k)',\mathfrak{s}}^{\GL_{N-1}(\E)}$ which is what we wanted to show.
		
		Finally, for the coefficients of
		\[\gamma(X^{f(\E/\F)},\rho^{\iota}\otimes  (\Gamma^{\GL_{N-1}(\E)}_{\mathfrak W_{N-1},W(k)',\mathfrak{s}})^{\iota},\phi_{\GL_N}\otimes\phi_{\GL_{N-1}}^{\vee},\psi^{-1}\circ\Tr_{\E/\F}),\]
		by the local functional equation, Theorem \ref{mainthm}, and compatibility of Langlands--Shahidi gamma factors with transfer, we have 
				\begin{align*}
				\gamma(X^{f(\E/\F)},\rho^{\iota}\otimes ( \Gamma^{\GL_{N-1}(\E)}_{\mathfrak W_{N-1},W(k)',\mathfrak{s}})^{\iota},&\phi_{\GL_N}\otimes\phi_{\GL_{N-1}}^{\vee},\psi^{-1}\circ\Tr_{\E/\F})\\&=		\gamma(q^{-f(\E/\F)}X^{-f(\E/\F)},\rho\otimes  (\Gamma^{\GL_{N-1}(\E)}_{\mathfrak W_{N-1},W(k)',\mathfrak{s}})^{\iota},\phi_{\GL_N}\otimes\phi_{\GL_{N-1}},\psi_{\E})^{-1}\\&=\gamma(q^{-1}X^{-1},\Gamma^{\G}_{\mathfrak{w},W(k)',\mathfrak{t}}\otimes (\Gamma^{\GL_{N-1}(\E)}_{\mathfrak W_{N-1},\mathfrak{s}})^{\iota},\phi_{\LG}\otimes\phi_{\GL_{N-1}},\psi)^{-1}\\
				&=\gamma(X,\Gamma^{\G}_{\mathfrak{w},W(k)',\mathfrak{t}}\otimes (\Gamma^{\GL_{N-1}(\E)}_{\mathfrak W_{N-1},\mathfrak{s}})^{\iota},\phi_{\LG}^{\vee}\otimes\phi_{\GL_{N-1}}^{\vee},\psi^{-1}).
				\end{align*}
Moreover, these coefficients lie in~$\R'\otimes \E_{\mathfrak{W}_{N-1},W(k)',\mathfrak{s}}^{\GL_{N-1}(\E)}$ by Theorem \ref{mainthm} which is what we wanted to show. This completes the proof.
			\end{proof}

	We conjecture that this is surjective.
	
	\begin{conj}\label{conjecturesurjectivity}
		Let~$\G$ be one of the quasi-split groups~$\SO_{2n+1}(\F), \Sp_{2n}(\F),\U_n(\E/\F)$.  The morphism
					\[\mathbf{t}_{\G,\mathfrak{W},\mathfrak{w}}\colon\E^{\GL_{N}(\E)}_{\mathfrak{W}}\rightarrow \E^{\G}_{\mathfrak{w}}\]
is surjective.
	\end{conj}
	
	\begin{remark}
		 An approach to this conjecture, using gamma factors, would be to prove an analogue of the Helm--Moss descent theorem for the Langlands-Shahidi gamma factors for~$\G\times\GL_{m}$.
			 For even special orthogonal groups, from Remark \ref{remarklocalconverse} this conjecture would need some mild adjustments. 
	\end{remark}
	
	The canonical map~$\mathfrak{Z}_{\GL_N(\E),\R_0}\xrightarrow{\sim} \E_{\GL_N(\E),\mathfrak{W}}$ is an isomorphism so from Theorem \ref{mainthm2}, we have defined a morphism~$\mathfrak{Z}_{\GL_N(\E),\R_0}\rightarrow  \E_{\G,\mathfrak{w}}$.  When local transfer of all irreducible~$\mathbb{C}$-representations of~$\G$ is defined (for example, for unitary, odd special orthogonal and symplectic groups) it is natural to conjecture:  there exists a transfer morphism~$\mathfrak{Z}_{\GL_N(\E),\R_0}\rightarrow \mathfrak{Z}_{\G,\R_0}$ compatible with local transfer, and~$\mathbf{t}_{\G,\mathfrak{W},\mathfrak{w}}$ via the canonical morphisms.
	\begin{conj}\label{conjfactors}
		There exists a unique transfer morphism~$\mathfrak{Z}_{\GL_N(\E),\R_0}\rightarrow \mathfrak{Z}_{\G,\R_0}$ interpolating the local transfer of irreducible~$\mathbb{C}$-representations.  In particular, this would be compatible with~$\mathbf{t}_{\G,\mathfrak{W},\mathfrak{w}}$ via the canonical morphisms.  
	\end{conj}
	
	\subsection{Fargues--Scholze transfer morphisms and Conjecture \ref{conjfactors}}\label{FSconsequences}
	%
	Let~$\G$ denote one of the quasi-split groups~$\Sp_{2n}(\F),\SO_r(\F),\SO_{2n,\E'/\F}(\F),\U_n(\E/\F),\mathrm{GSpin}_r(\F),\mathrm{GSpin}_{2n,\E'/\F}(\F)$. We have the following diagram 
	\[
	\begin{tikzcd}
		\mathrm{Exc}(\W_\E,\GL_N) \arrow[r, "\mathrm{FS}\,\sim"] \arrow[d] & \mathfrak{Z}_{\GL_N(\E),\R_0} \arrow[r, "\sim"] & \E^{\GL_N(\E)}_{\mathfrak{W}} \\
		\mathrm{Exc}(\W_\F,\widehat{\G}) \arrow[r, "\mathrm{FS}"] & \mathfrak{Z}_{\G,\R_0} \arrow[r,] & \E^{\G}_{\mathfrak{w}}
	\end{tikzcd}
	\]
	where the morphisms~$\mathrm{FS}$ follow from \cite{fargues2024geometrizationlocallanglandscorrespondence,Scholze}, (see \cite[Lemma A.3]{DHKMconj}), where the left vertical arrow is given by pushforward of cocycles along the fundamental representation, and the top line is an isomorphism \cite[Corollary 8.4]{DHKMconj} and \cite[Theorem VIII.5.1]{fargues2024geometrizationlocallanglandscorrespondence} which compares~$\mathrm{Exc}(\W_\E,\GL_N)$ to the ring considered in \cite{DHKMconj}.
	%
	%
	
	In particular, we obtain Fargues--Scholze transfer morphisms
	\[
	\begin{tikzcd}
		\mathfrak{Z}_{\GL_N(\E),\R_0} \arrow[r, "\sim"] \arrow[d, "\mathbf{s}^{\mathrm{FS}}_{\G}"]& \E^{\GL_N(\E)}_{\mathfrak{W}}  \arrow[d, "\mathbf{t}^{\mathrm{FS}}_{\G,\mathfrak{W},\mathfrak{w}}"]\\
		\mathfrak{Z}_{\G,\R_0} \arrow[r,""] & \E^{\G}_{\mathfrak{w}}
	\end{tikzcd}
	\]
	
	Suppose~$\mathrm{FS}$ is compatible with the local generic transfer maps of Arthur \cite{Arthur2}, or Mok \cite{Mok}, or those considered in Section \ref{generictransfercomplex}, for~$\G$.  Then~$\mathbf{t}_{\G,\mathfrak{W},\mathfrak{w}}=\mathbf{t}^{\mathrm{FS}}_{\G,\mathfrak{W},\mathfrak{w}}$ and~$\mathbf{t}^{\mathrm{FS}}_{\G,\mathfrak{W},\mathfrak{w}}$ is compatible with Langlands-Shahidi gamma factors.  And, in particular, a version of Conjecture \ref{conjfactors} is true for Fargues--Scholze semisimple transfer.  Amongst the quasi-split groups we consider here, this compatibility is known for: unitary groups~$\U_{n}(\E/\F)$, and split special orthogonal groups~$\SO_{r}(\F)$ by \cite[Theorem II, Corollary III]{poletal}, and~$\Sp_4(\F)$ with~$\F/\mathbb{Q}_p$ unramified and~$p\neq 2$ by \cite{MR4978375}.
	
	
	\subsection{The~$\ell$-adic local transfer and congruences}
	Fix~$\ell\neq p$, let
	\begin{equation*}\mathbb{Z}_{\ell}':=\begin{cases}\mathbb{Z}_{\ell}^{ur}&\text{if~}\ell\neq 2;\\
	\mathbb{Z}_{\ell}^{ur}[\sqrt{q},\mu_4]&\text{if~}\ell=2;\end{cases}
\end{equation*} where~$\mathbb{Z}_{\ell}^{ur}$ denotes the ring of integers of the maximal unramified extension of~$\mathbb{Q}_{\ell}$.  Then we have defined a transfer morphism 
	\begin{equation*}
		\mathbf{t}_{\G,\mathfrak{W},\mathfrak{w},\ell}\colon \E^{\GL_N(\E)}_{\mathfrak{W},\mathbb{Z}_{\ell}'}\rightarrow\E^{\G}_{\mathfrak{w},\mathbb{Z}_{\ell}'}\end{equation*}
	interpolating the local generic transfer.  (Note it is more convenient to use this morphism over the Noetherian base ring~$\mathbb{Z}_{\ell}'$, rather than extend scalars to~$\Zl$ here).

	We obtain the following corollary, which improves on the generic supercuspidal case of \cite[Theorem 1.1]{MS}, extending it to generic transfer of non-supercuspidal generic representations, and including further groups of classical type, and gives a largely local rather than global proof (albeit there are global considerations in the work of Shahidi and in the proof of the local transfer, on which we depend). 
	
	Let~$\pi$ be an integral (irreducible)~$\mathfrak{w}$-generic representation of a quasi-split~$p$-adic group with Whittaker datum~$\mathfrak{w}$.  Then its reduction modulo~$\ell$,~$r_{\ell}(\pi)$ contains a unique (irreducible)~$\mathfrak{w}$-generic~$\Fl$-representation, and we write~$J_{\ell}(\pi)$ for this~$\mathfrak{w}$-generic~$\Fl$-representation.
	
	There is a unique generic~$\ell$-adic transfer
	\[\mathrm{t}_{\G,\mathfrak{w},\ell}\colon\Pi_{\mathfrak{w},\Ql}(\G)\rightarrow \Pi_{gen,\Ql}(\GL_N(\E)),\]
	compatible with~$\mathbf{t}_{\G,\mathfrak{w},\ell}\colon \E^{\GL_N(\E)}_{\mathfrak{W},\mathbb{Z}_{\ell}'}\rightarrow\E^{\G}_{\mathfrak{w},\mathbb{Z}_{\ell}'}$.

	\begin{corollary}
	\label{MScorollary}
		Let~$\G$ be one of the quasi-split groups
		\[\Sp_{2n}(\F),\SO_r(\F),\SO_{2n,\E'/\F}(\F),\U_n(\E/\F),\mathrm{GSpin}_r(\F),\mathrm{GSpin}_{2n,\E'/\F}(\F).\]  
		\begin{enumerate}
			\item  The generic local transfer\[\mathrm{t}_{\G,\mathfrak{w},\ell}\colon\Pi_{\mathfrak{w},\Ql}(\G)\rightarrow \Pi_{gen,\Ql}(\GL_N(\E))\] preserves integrality of~$\Ql$-representations.  
			\item If~$\pi$ and~$\pi'$ are integral~$\mathfrak{w}$-generic~$\Ql$-representations of~$\G$ then
			\[J_{\ell}(\pi)\simeq J_{\ell}(\pi')\Rightarrow J_{\ell}(\mathrm{t}_{\G,\mathfrak{w},\ell}(\pi))\simeq J_{\ell}(\mathrm{t}_{\G,\mathfrak{w},\ell}(\pi')).\]
			Moreover, in the special case that~$\G$ is one of the quasi-split groups~$\SO_{2n+1}(\F), \Sp_{2n}(\F),\U_n(\E/\F)$, if Conjecture \ref{conjecturesurjectivity} is true then this is an equivalence:
			\[J_{\ell}(\pi)\simeq J_{\ell}(\pi')\Leftrightarrow J_{\ell}(\mathrm{t}_{\G,\mathfrak{w},\ell}(\pi))\simeq J_{\ell}(\mathrm{t}_{\G,\mathfrak{w},\ell}(\pi')).\]
			\item Suppose~$\pi$ and~$\pi'$ are integral generic~$\Ql$-representations of~$\G$ with~$J_{\ell}(\pi)\simeq J_{\ell}(\pi')$.  The irreducible subquotients of~$r_{\ell}(\mathrm{t}_{\G,\mathfrak{w},\ell}(\pi))\oplus r_{\ell}(\mathrm{t}_{\G,\mathfrak{w},\ell}(\pi'))$ all have the same supercuspidal support.
		\end{enumerate}
	\end{corollary}
	\begin{remark}
		In the case~$\SO_2(\F)$, as there is a unique generic member in an extended~$L$-packet for~$\GL_2(\F)$, the morphism we consider (defines and) interpolates the ``generic local transfer''.  Note that in this case, it does not agree with the usual local transfer of characters considered in \cite[§9.3]{MS}, which is why we do not find the same problems in the~$\SO_2(\F)$ case that they do.  However, in all other cases for classical groups in~$\{\Sp_{2n}(\F),\SO_r(\F),\SO_{2n,\E'/\F}(\F),\U_n(\E/\F)\}$, the local transfer we consider on supercuspidal representations agrees with the local transfer defined using~$L$-parameters that M\'inguez--S\'echerre consider -- see Remark \ref{remarkfulltransfer}, so Corollary \ref{MScorollary} does extend the generic case of theirs.\end{remark}
	
	\begin{proof}
		Let~$\pi\in\Pi_{\mathfrak{w},\Ql}(\G)$, then it is defined over a finite extension~$\K$ of~$\mathbb{Q}_{\ell}^{ur}$ by \cite[II 4.7]{Vig96}.  Moreover,~$\pi$ is~$\Zl$-integral if and only if it is~$\mathfrak{o}_{\K}$-integral, by the criteria for integrality \cite[II 4.12]{Vig96}, \cite[Corollary 1.6]{DHKMfiniteness}.  Thus the first statement is a direct consequence of Lemma \ref{cowhittakerDVR} \eqref{cowhittakerDVR1} and that we have an integral transfer morphism~$\mathbf{t}_{\G,\mathfrak{W},\mathfrak{w},\ell}$ of~$\mathbb{Z}_\ell'$-algebras.\footnote{A more standard proof of the first part (in line with M\'inguez--S\'echerre \cite{MS}) can be found by using the integrality criteria of \cite{DHKMconj}, in particular see Remark 6.24 of \emph{op.~cit}.}
		
		Note that as~$\pi$ and~$\pi'$ contain a common subquotient on reduction modulo~$\ell$ they have the same depth, and we let~$e_{\mathfrak{t}}$ denote the idempotent of~$\mathfrak{Z}_{\G,\mathbb{Z}_\ell'}$ corresponding to this depth.  Again, we can suppose that~$\pi$ and~$\pi'$ are defined over (the same) finite extension~$\K$ of~the maximal unramified extension of~$\mathbb{Q}_{\ell}$.  By Lemma \ref{cowhittakerDVR},~$\pi$ and~$\pi'$ define morphisms~$f_{\pi,0},f_{\pi',0}\colon\E^{\G}_{\mathfrak{w},\mathfrak{t}}\rightarrow \mathfrak{o}_{\K}$ which are congruent modulo~$\ell$.  
		Composing with the integral transfer morphism~$\mathbf{t}_{\G,\mathfrak{W},\mathfrak{w},\ell}\colon\E^{\GL_N(\E)}_{\mathfrak{W},\mathfrak{o}_{\K},\mathfrak{t}'}\rightarrow \E^{\G}_{\mathfrak{w},\mathfrak{o}_{\K},\mathfrak{t}}$ we obtain congruent morphisms~$ f_{\pi,0}\circ\mathbf{t}_{\G,\mathfrak{W},\mathfrak{w},\ell},f_{\pi',0}\circ \mathbf{t}_{\G,\mathfrak{W},\mathfrak{w},\ell}\colon\E^{\GL_N(\E)}_{\mathfrak{W},\mathfrak{o}_{\K},\mathfrak{t}'} \rightarrow \mathfrak{o}_{\K}$ satisfying: 
		\[((\Gamma^{\GL_N(\E)}_{\mathfrak{W},\mathfrak{t}'})_{f_{\pi,0}\circ \mathbf{t}_{\G,\mathfrak{W},\mathfrak{w},\ell}}\otimes \Ql)_{gen}=\mathrm{t}_{\G,\mathfrak{w},\ell}(\pi), \quad ((\Gamma^{\GL_N(\E)}_{\mathfrak{W},\mathfrak{t}'})_{f_{\pi',0}\circ\mathbf{t}_{\G,\mathfrak{W},\mathfrak{w},\ell}}\otimes \Ql)_{gen}=\mathrm{t}_{\G,\mathfrak{w},\ell}(\pi').\]
		%
		%
		And we deduce, from Lemma \ref{cowhittakerfields} that the unique generic subquotient of the reduction modulo~$\ell$ of~$\mathrm{t}_{\G,\mathfrak{w},\ell}(\pi)$ is isomorphic to the unique generic subquotient of the reduction modulo~$\ell$ of~$\mathrm{t}_{\G,\mathfrak{w},\ell}(\pi')$.
		
		If~$\mathbf{t}_{\G,\mathfrak{w},\ell}$ is surjective, then as~$ f_{\pi,0}\circ \mathbf{t}_{\G,\mathfrak{W},\mathfrak{w},\ell},~ f_{\pi',0}\circ\mathbf{t}_{\G,\mathfrak{W},\mathfrak{w},\ell}\colon\E^{\GL_N(\E)}_{\mathfrak{W},\mathfrak{t}'} \rightarrow \mathfrak{o}_{\K}$ being congruent implies that~$f_{\pi,0},f_{\pi',0}\colon\E^{\G}_{\mathfrak{w},\mathfrak{t}} \rightarrow \mathfrak{o}_{\K}$ are congruent and we follow the same argument as above.
		
		The action of the integral Bernstein centre on~$(\Gamma^{\GL_N(\E)}_{\mathfrak{W},\mathfrak{t}'})_{ f_{\pi,0}\circ \mathbf{t}_{\G,\mathfrak{W},\mathfrak{w},\ell}}$ and~$(\Gamma^{\GL_N(\E)}_{\mathfrak{W},\mathfrak{t}'})_{f_{\pi',0}\circ \mathbf{t}_{\G,\mathfrak{W},\mathfrak{w},\ell}}$ factors through~$\E^{\GL_N(\E)}_{\mathfrak{W},\mathfrak{t}'}$ and hence is by congruent morphisms to~$\Zl$.  Moreover, by Helm \cite[Theorem 2.2]{HelmCW} this central action determines the supercuspidal support of all irreducible subquotients of these representations.  Finally, these representations dominate the co-Whittaker reductions~mod~$\ell$ of~$\mathrm{t}_{\ell}(\pi)$ and~$\mathrm{t}_{\ell}(\pi')$, so we deduce that all irreducible subquotients have the same supercuspidal support.
	\end{proof}
		
	This framework also allows one to extend some non-generic cases:  First note, that we have the following lemma: 
	\begin{lemma}\label{lemmacentre}
		Let~$\G$ be a reductive~$p$-adic group.  Let~$\pi$ be an irreducible~$\ell$-adic representation of~$\G$, then the canonical action~of~$\mathfrak{Z}_{\G,\Zl}$ on~$\pi$ is given by a morphism~$\chi_{\pi}:\mathfrak{Z}_{\G,\Zl}\rightarrow \Ql$: for~$z\in\mathfrak{Z}_{\G,\Zl}$, $z_{\pi}=\chi_{\pi}(z)\,\mathrm{id}_\pi$.
		\begin{enumerate}
			\item\label{lemmacentre1} The morphism~$\chi_{\pi}$ is~$\Zl$-valued if and only if~$\pi$ is integral.
			\item \label{lemmacentremodell} Let~$\rho\in \Irr_{\Fl}(\G)$. Then the canonical action of~$\mathfrak{Z}_{\G,\Zl}$ on~$\rho$ is via a morphism
			\[\chi_{\rho}:\mathfrak{Z}_{\G,\Zl}\rightarrow \Fl.\]
			\item Suppose that~$\pi$ is integral.
			\begin{enumerate} 
				\item \label{lemmacentre2} let~$\rho,\rho'$ be irreducible subquotients of~$r_{\ell}(\pi)$, then
				\[\chi_{\rho}=\chi_{\rho'}\]
				\item \label{lemmacentre3} Suppose~$\pi'$ is an irreducible integral~$\ell$-adic representation of~$\G$, and~$r_{\ell}(\pi)$ and~$r_{\ell}(\pi')$ share a common subquotient.  Then
				\[r_{\ell}\circ\chi_{\pi}=r_{\ell}\circ\chi_{\pi'}.\]\end{enumerate}
		\end{enumerate}
	\end{lemma}
	
	\begin{proof}
		Let~$z\in \mathfrak{Z}_{\G,\Zl}$, then~$z_{\pi}\in\End_{\Zl[\G]}(\pi)$, and~$\End_{\Zl[\G]}(\pi)=\End_{\Ql[\G]}(\pi)$ as multiplication by elements of~$\Zl\backslash\{0\}$ is invertible on~$\pi$.  
		Schur's lemma for~$\pi$ thus shows that~$z_{\pi}\in\Ql \,\mathrm{id}_{\pi}$, and the centre acts by a morphism~$\chi_{\pi}:\mathfrak{Z}_{\G,\Zl}\rightarrow \Ql$.
		
		Let~$L$ be a lattice in~$\pi$ and~$z\in \mathfrak{Z}_{\G,\Zl}$.  Then we have~$z_{L}\in\End_{\Zl[\G]}(L)$, and~$z_{\pi}\in\End_{\Zl[\G]}(\pi)$ compatible with the natural map~$\phi:L\hookrightarrow \pi$, i.e.,~$\phi\circ z_{L}=z_{\pi}\circ \phi$.   Moreover, as~$ \mathfrak{Z}_{\G,\Zl}$ acts via a~$\Zl[\G]$-endomorphism on~$L$ and~$z_{L}=\chi_{\pi}(z)\, \mathrm{id_L}$ we have~$\chi_{\pi}(z)\in\Zl$.   Conversely suppose that~$\chi_{\pi}$ takes values in~$\Zl$, and choose a finite extension~$K$ of the maximal unramified extension of~$\mathbb{Q}_{\ell}$ such that~$\pi$ is realisable over~$\K$ with model~$\pi_{\K}$.   It follows from \cite[Theorem 1.2]{DHKMfiniteness} that, for~$v\in\pi_{\K}\backslash\{0\}$,~$L=\mathfrak{o}_{\K}[\G] v$ is a lattice in~$\pi$.  Indeed, this reference shows that~$L$ is admissible over~$\mathfrak{Z}_{\G,\mathfrak{o}_{\K}}$, but by hypothesis the centre acts through~$\Zl\cap \K=\mathfrak{o}_{\K}$ so~$L$ is admissible over~$\mathfrak{o}_{\K}$ and it follows from \cite[§9]{Vig96} that~$\pi$ is~$\mathfrak{o}_{\K}$-integral.
		
		Part \eqref{lemmacentremodell} follows from Schur's lemma for~$\rho$, as~$\End_{\Zl[\G]}(\rho)\simeq \End_{\Fl[\G]}(\rho)$, by reduction modulo~$\ell$, because the maximal ideal of~$\Zl$ annihilates~$\rho$.
		
		Suppose that~$z\in\mathfrak{Z}_{\G,\Zl}$ acts on~$\pi$ by~$y\in\Zl$, we claim that~$z$ acts on any irreducible subquotient of~$r_{\ell}(\pi)$ by~$y\otimes 1\in \Fl$.  Indeed, choose a lattice~$L$ in~$\pi$, then~$z$ acts on~$L$ by multiplication by~$y$.  Then~$z$ acts on~$L\otimes \Fl$ by~$y\otimes 1$.  The first part follows by compatibility of the action of the integral centre, and the second part follows from the first.
	\end{proof}
	
	We now put ourselves in the setting where Conjecture \ref{conjfactors} holds, and we have a transfer morphism
	\begin{equation}\label{Zlmorphism}
		\mathbf{s}_{\G,\ell}\colon \mathfrak{Z}_{\GL_N(\E),\mathbb{Z}_\ell'}\rightarrow\mathfrak{Z}_{\G,\mathbb{Z}_\ell'},\end{equation}
	compatible with the local transfer~$\mathrm{s}_{\G,\ell}\colon\Irr_{\Ql}(\G)\rightarrow \Irr_{\Ql}(\GL_N(\E))$.
	
	\begin{corollary}
		Let~$\G$ be one of the quasi-split groups~$\SO_{r}(\F)$ with~$r\neq 2$, $\U_n(\E/\F)$, or~$\Sp_{4}(\F)$ with~$\F$ unramified and~$p$ odd.
		\begin{enumerate}
			\item The local transfer~$\mathrm{s}_{\G,\ell}\colon\Irr_{\Ql}(\G)\rightarrow \Irr_{\Ql}(\GL_N(\E))$ preserves integrality.  
			\item Moreover, for~$\pi,\pi'\in\Irr_{\Ql}(\G)$ integral, if~$r_{\ell}(\pi),r_{\ell}(\pi')$ share a common subquotient, then the irreducible subquotients of~$r_{\ell}(\mathrm{s}_{\G,\ell}(\pi)) \oplus r_{\ell}(\mathrm{s}_{\G,\ell}(\pi'))$ all have the same supercuspidal support.
			\item Let~$\pi,\pi'\in\Irr_{\Ql}(\G)$ be integral and cuspidal and~$r_{\ell}(\pi),r_{\ell}(\pi')$ share a common subquotient.  Then the transfers~$\mathrm{s}_{\G,\ell}(\pi),\mathrm{s}_{\G,\ell}(\pi')$ are generic, and
			\[J_{\ell}(\mathrm{s}_{\G,\ell}(\pi))\simeq J_{\ell}(\mathrm{s}_{\G,\ell}(\pi')).\]
		\end{enumerate}
	\end{corollary}
	
	\begin{proof}
		If~$\pi\in \Irr_{\Ql}(\G)$ is integral, then~$\chi_{\pi}\circ \mathbf{s}_{\G,\ell}$ is~$\Zl$-valued and hence~$\mathrm{s}_{\G,\ell}(\pi)$ is integral.  If~$r_{\ell}(\pi),r_{\ell}(\pi')$ share a common subquotient, then from~$\mathbf{s}_{\G,\ell}$ and Lemma \ref{lemmacentre} \eqref{lemmacentre3} we obtain congruent morphisms~$ \mathfrak{Z}_{\GL_N(\E),\Zl}\rightarrow \Zl$.  But from \cite[§12]{HelmBC}, this implies that all irreducible subquotients of~$r_{\ell}(\mathrm{s}_{\G,\ell}(\pi)) \oplus r_{\ell}(\mathrm{s}_{\G,\ell}(\pi'))$ have the same supercuspidal support.  Moreover, the local transfer of a cuspidal representation is generic (Remark \ref{remarkfulltransfer}), and then the final part follows from the second as there is a unique generic representation with a given supercuspidal support.
	\end{proof}

	\bibliographystyle{alpha}
	\bibliography{LS}
	
\end{document}